\documentclass[11pt]{article}

\usepackage[default]{jasa_harvard}    % 	for formatting citations in text

\usepackage{endfloat}

\usepackage{epsfig}
\usepackage{graphicx}
\usepackage{amsbsy}
\usepackage{amsfonts}
\usepackage{amssymb}
\usepackage{amsmath, amsthm}
\usepackage{caption}
\usepackage{setspace}
\usepackage{bm}
\usepackage{float}
\usepackage[usenames,dvipsnames]{color}
\usepackage{multirow}
\usepackage[T1]{fontenc}
\usepackage{lmodern}
\usepackage[figuresright]{rotating}
\usepackage{enumerate}
\usepackage{url}
\usepackage{bm}
\usepackage{xr}

\newtheorem{lemma}{Lemma}
\newtheorem{cor}{Corollary}
\newtheorem{prop}{Proposition}
\newtheorem{defn}{Definition}
\newtheorem{example}{Example}

\newcommand{\real}{\mathbb{R}}    

\newcommand{\hbtheta}{\hat{\bm{\theta}}}
\newcommand{\btheta}{\bm{\theta}}
\newcommand{\tbtheta}{\tilde{\bm{\theta}}}
\newcommand{\htheta}{\hat{\theta}}
\newcommand{\ttheta}{\tilde{\theta}}
\newcommand{\tphi}{\tilde{\phi}}

\begin{document}

\title{ \textbf{ Non-Local Priors for High-Dimensional Estimation}}

% List of authors, with corresponding author marked by asterisk
\author{David Rossell$^{1}$,  Donatello Telesca$^{2}$\\
\\
\small {\bf Author's Footnote}\\
\small $^1$ University of Warwick, Department of Statistics\\
\small $^2$ UCLA, Department of Biostatistics
}

\date{}

\maketitle

\newpage

\begin{center}
\textbf{Abstract}
\end{center}
Simultaneously achieving parsimony and good predictive power in high dimensions is a main challenge in statistics.
Non-local priors (NLPs) possess appealing properties for high-dimensional model choice,
but their use for estimation has not been studied in detail.
We show that, for regular models, Bayesian model averaging (BMA) estimates based on NLPs shrink spurious parameters
either at fast polynomial or quasi-exponential rates as the sample size $n$ increases (depending on the chosen prior density).
Non-spurious parameter estimates only differ from the oracle MLE by a factor of $n^{-1}$.
We extend some results to linear models with dimension $p$ growing with $n$.
 Coupled with our theoretical investigations, we outline the constructive representation of 
NLPs as mixtures of truncated distributions.   From a practitioners' perspective, our work 
enables simple posterior sampling and extending NLPs beyond previous proposals.
Our results show notable high-dimensional estimation for linear models with $p>>n$ at reduced computational cost.
NLPs provided lower estimation error than benchmark and hyper-g priors, SCAD and LASSO
in simulations, and in gene expression data achieved higher cross-validated $R^2$ with an order of magnitude less predictors.
Remarkably, these results were obtained without the need to pre-screen predictors.
Our findings contribute to the debate of whether different priors
should be used for estimation and model selection,
showing that selection priors may actually be desirable for high-dimensional estimation.

\vspace*{.3in}

\noindent \textsc{Keywords}: Model Selection, MCMC, Non Local Priors, Bayesian Model Averaging, Shrinkage

\newpage

\section{Introduction}

Developing high-dimensional methods that balance parsimony and good predictive power is a main challenge in statistics.
Non-local prior (NLP) distributions have appealing
properties for Bayesian model selection. Relative to local priors (LPs), 
NLPs discard spurious covariates faster as the sample size $n$ grows, while preserving exponential learning rates to detect non-zero coefficients \cite{johnson:2010}.
As shown below, when combined with Bayesian model averaging (BMA), this extra shrinkage has important consequences for parameter estimation.
Denote the observations by  ${\bf y}_n \in \mathcal{Y}_n$, where $\mathcal{Y}_n$  is the sample space.
We entertain a collection of models $M_k$ for $k=1,\ldots,K$ with
Radon-Nikodym densities $f_k({\bf y}_n \mid \bm{\theta}_k,\phi_k)$, 
%defined with respect to an adequate $\sigma$-algebra and $\sigma$-finite dominating measure,
where  $\bm{\theta}_k \in \Theta_k \subseteq \Theta$  are parameters of interest
and $\phi_k \in \Phi$ is a fixed-dimension nuisance parameter.
We assume that models are nested in $M_K$ ($\Theta_K=\Theta$),
$\Theta_k \bigcap \Theta_{k'}$ has 0 Lebesgue measure for $k \neq k'$, 
denote $|k|=\mbox{dim}(\Theta_k \times \Phi)$
 and $(\bm{\theta},\phi)=(\bm{\theta}_K,\phi_K) \in \Theta \times \Phi$.
%and for any $k' \in \{1,\ldots,K-1\}$
%there exists $k$ such that $\Theta_{k'} \subset \Theta_k$,
%$\mbox{dim}(\Theta_{k'}) < \mbox{dim}(\Theta_k)$
A prior density $\pi(\bm{\theta}_k \mid M_k)$ for  $\bm{\theta}_k \in \Theta_k$ under $M_k$
is a NLP if it converges to 0 as $\bm{\theta}_k$ approaches any value $\bm{\theta}_0$ consistent with a sub-model $M_{k'}$.
%\begin{defn}
%Let $\|\bm{\theta}_k -\bm{\theta}_0\|$ 
%be a distance between $\bm{\theta}_k \in \Theta_k$ 
%and $\bm{\theta}_0 \in \Theta_{k'} \subset \Theta_k$.
%$\pi(\bm{\theta}_k \mid M_k)$ is a non-local prior 
%under $M_k$ if for all $\bm{\theta}_0$, 
%$k' \neq k$ and $\epsilon >0$ there exists $\eta>0$ such that 
%$\| \bm{\theta}_k - \bm{\theta}_0 \| < \eta$ implies that $\pi(\bm{\theta}_k \mid M_k) < \epsilon$.
%\end{defn}

\begin{defn}
Let $\bm{\theta}_k \in \Theta_k$, 
an absolutely continuous $\pi(\bm{\theta}_k \mid M_k)$ is a non-local prior if 
$\mathop {\lim }\limits_{\bm{\theta}_k \to \bm{\theta}_0} \pi(\bm{\theta}_k \mid M_k)=0$
for any $\bm{\theta}_0 \in \Theta_{k'} \subset \Theta_k$, $k' \neq k$.
\end{defn}

For preciseness, we assume that  %(as usual)
any $\Theta_k \bigcap \Theta_{k'} \subseteq \Theta_{k''}$  for some $|k''|<\mbox{min}\{|k|,|k'|\}$.
To fix ideas, we consider variable selection
where $E({\bf y}_n)=g(X_n \bm{\theta})$ for a given function $g(\cdot)$ and predictors $X_n$
and $\Theta_k \subset \Theta_K$ by setting elements in $\bm{\theta}$ to 0.
We entertain the following NLP densities 
\begin{align}
\pi_M(\bm{\theta} \mid \phi_k, M_k) = \prod_{i \in M_k}^{} \frac{\theta_i^2}{\tau
    \phi_k} N(\theta_i; 0, \tau \phi_k) 
\label{eq:pmom} \\
\pi_I(\bm{\theta} \mid \phi_k, M_k) = \prod_{i \in M_k}^{} \frac{(\tau \phi_k)^{\frac{1}{2}}}{\sqrt{\pi} \theta_i^2} 
\mbox{exp}\left\{ - \frac{\tau \phi_k}{\theta_i^2} \right\} 
\label{eq:pimom} \\
\pi_E(\bm{\theta} \mid \phi_k, M_k) = \prod_{i \in M_k}^{} \exp \left\{
    \sqrt{2} - \frac{\tau \phi_k}{\theta_i^2} \right\} N(\theta_i; 0, \tau \phi_k),
\label{eq:pemom}
\end{align}
where $i \in M_k$ are the non-zero coefficients,
$N(\theta_i; 0, v)$ is the univariate Normal density with mean $0$ and variance $v$
and $\pi_M$, $\pi_I$ and $\pi_E$ are called the product MOM, iMOM and eMOM priors (pMOM, piMOM and peMOM, respectively).
Consider the usual BMA estimate
\begin{align}
E(\bm{\theta} \mid {\bf y}_n)= \sum_{k=1}^{K} E(\bm{\theta} \mid M_k, {\bf y}_n) P(M_k \mid {\bf y}_n)
\label{eq:bma}
\end{align}
where $P(M_k \mid {\bf y}_n) \propto m_k({\bf y}_n) P(M_k)$ and
$m_k({\bf y}_n)$ is the integrated likelihood under $M_k$.
BMA shrinks estimates by assigning small $P(M_k \mid {\bf y}_n)$ to unnecessarily complex models.
%combining the regularization induced by $m_k({\bf y}_n)$ and that by $P(M_k)$.
Let $M_t$ be the smallest model such that $f_t({\bf y}_n \mid \bm{\theta}_t)$ minimizes Kullback-Leibler (KL) divergence
to the data-generating density for some $\bm{\theta}_t \in \Theta_t$.
Under regular models with fixed $P(M_k)$ and $\mbox{dim}(\Theta)$, if $\pi(\bm{\theta}_k \mid M_k)$ is a LP
and $M_t \not \subset M_k$ then $P(M_k \mid {\bf y}_n)=O_p(e^{-n})$ and
if $M_t \subset M_k$ then $P(M_k \mid {\bf y}_n)=O_p(n^{-\frac{1}{2}(|k|-|t|)})$ \cite{dawid:1999}.
Models containing spurious parameters are hence regularized at a slow polynomial rate,
which implies that if truly $\theta_i=0$ then $E(\theta_i \mid {\bf y}_n)= O(n^{-1})p_t$ (Section \ref{sec:asymptotics}),
where $p_t$ depends on ratios of model prior probabilities.
We shall show that any LP can be transformed into a NLP to achieve either $E(\theta_i \mid {\bf y}_n)= O_p(n^{-2})p_t$ (pMOM)
or $E(\theta_i \mid {\bf y}_n)= O_p(e^{-\sqrt{n}})p_t$ (peMOM, piMOM).
A complementary strategy is to penalize complex models via $p_t$. % \cite{scott:2010}.
For instance, \citeasnoun{castillo:2012} and \citeasnoun{castillo:2014} consider multiple Normal means and
linear regression in a fully Bayesian framework where variable inclusion probabilities decrease with $|K|$,
and show that certain $\pi(\bm{\theta}_k \mid {\bf y}_n)$ induce shrinkage at optimal asymptotic minimax rates, 
in particular finding that light-tailed priors ({\it e.g.} Normal) are sub-optimal.
\citeasnoun{martin:2013} propose a related empirical Bayes strategy. %, again obtaining optimal minimax rates.
Yet another option is to consider the single model $M_K$ and specify absolutely continuous shrinkage priors,
which can also achieve good posterior concentration \cite{bhattacharya:2012}.
For a related review on penalized-likelihood strategies see \citeasnoun{fan:2010}.

In contrast, our strategy is based upon faster $m_k({\bf y}_n)$ rates and (optionally) sparsity-inducing $P(M_k)$.
To illustrate the key role of NLPs, in Normal regression models with $|K|=O(n^{\alpha})$, $0.5 \leq \alpha < 1$ and bounded $P(M_k)/P(M_t)$ 
then $P(M_t \mid {\bf y}_n) \stackrel{P}\longrightarrow 1$ when using NLPs and to 0 when using LPs \cite{johnson:2012}.
We note that when sparse $P(M_k)$ are used
consistency may still be achieved with LPs,
{\it e.g.} \citeasnoun{liang:2013} or \citeasnoun{narisetty:2014}
prove consistency in linear regression with prior inclusion probabilities $O(|K|^{-\gamma})$ for $\gamma>0$.
While both $m_k({\bf y}_n)$ and $P(M_k)$ induce sparsity,
we note that the latter is guided by  strong  a priori assumptions.

The main contribution of this manuscript is two-fold.
First, we characterize complexity penalties and implied asymptotic rates for BMA estimates induced by NLPs (Section \ref{sec:asymptotics}).
Second, we provide a one-to-one representation of NLPs as mixtures of truncated distributions (Section \ref{sec:theory})
that addresses important practical issues.
%represented as such a mixture, and any non-degenerate mixture induces a NLP. 
It provides an intuitive justification for NLPs, adds flexibility in prior choice and facilitates posterior sampling
under strong multi-modalities (Section \ref{sec:postsampling}). 
%$2^{p_k}$ modes, where $p_k={\mbox{dim}(\Theta_k)}$.
%Also, finding functional forms that lead to simple calculations limits the flexibility in choosing the prior.
Finally, we study finite-sample performance in simulations and gene expression data (Section \ref{sec:examples}).
%Throughout we place emphasis on high dimensions.
%See the early arguments in \citeasnoun{stein:1956} or related Bayesian perspective in  \citeasnoun{george:2012}.
%To provide intuition, we first introduce a simple example built on basic principles. 

\section{Non local priors for estimation} 
\label{sec:asymptotics}

\begin{figure}
\begin{center}
\begin{tabular}{c}
\includegraphics[width=0.7\textwidth,height=0.4\textwidth]{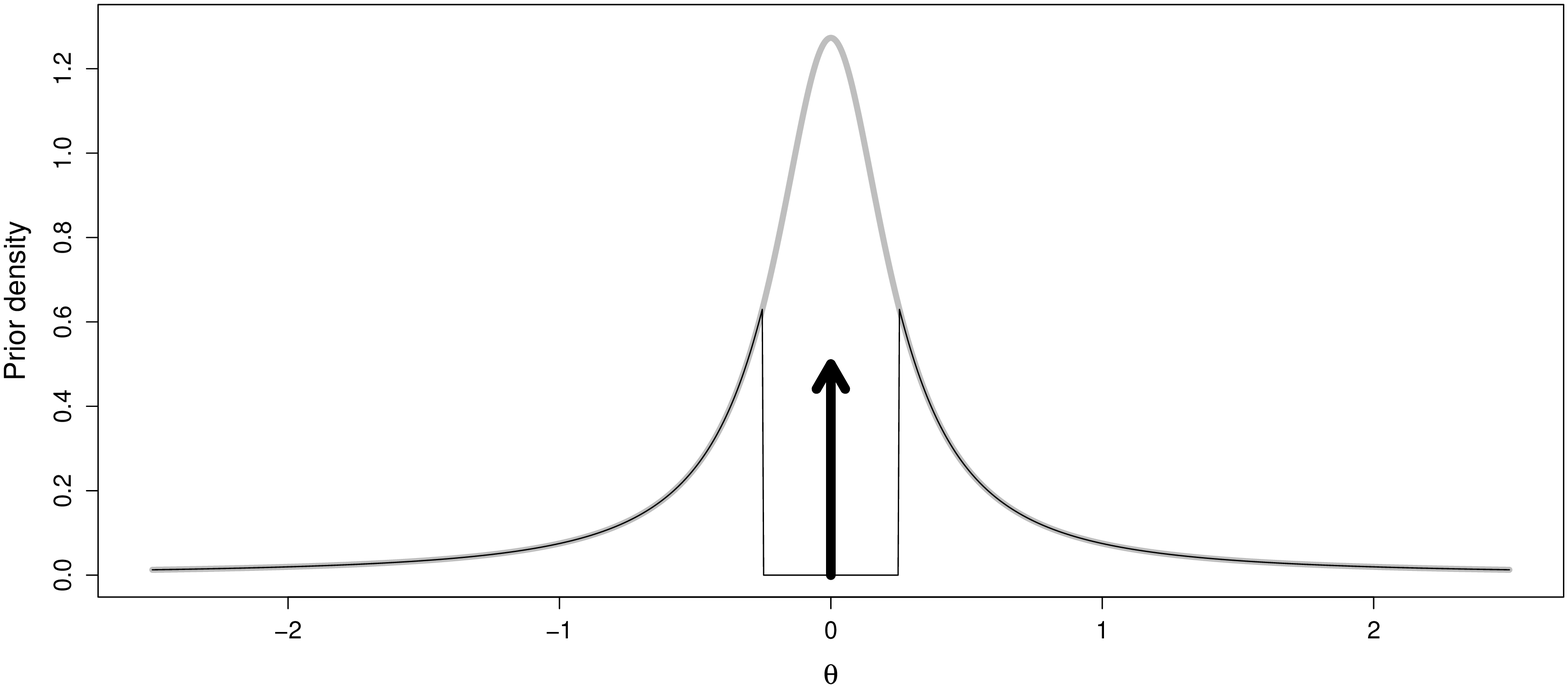} \\
\includegraphics[width=0.7\textwidth,height=0.4\textwidth]{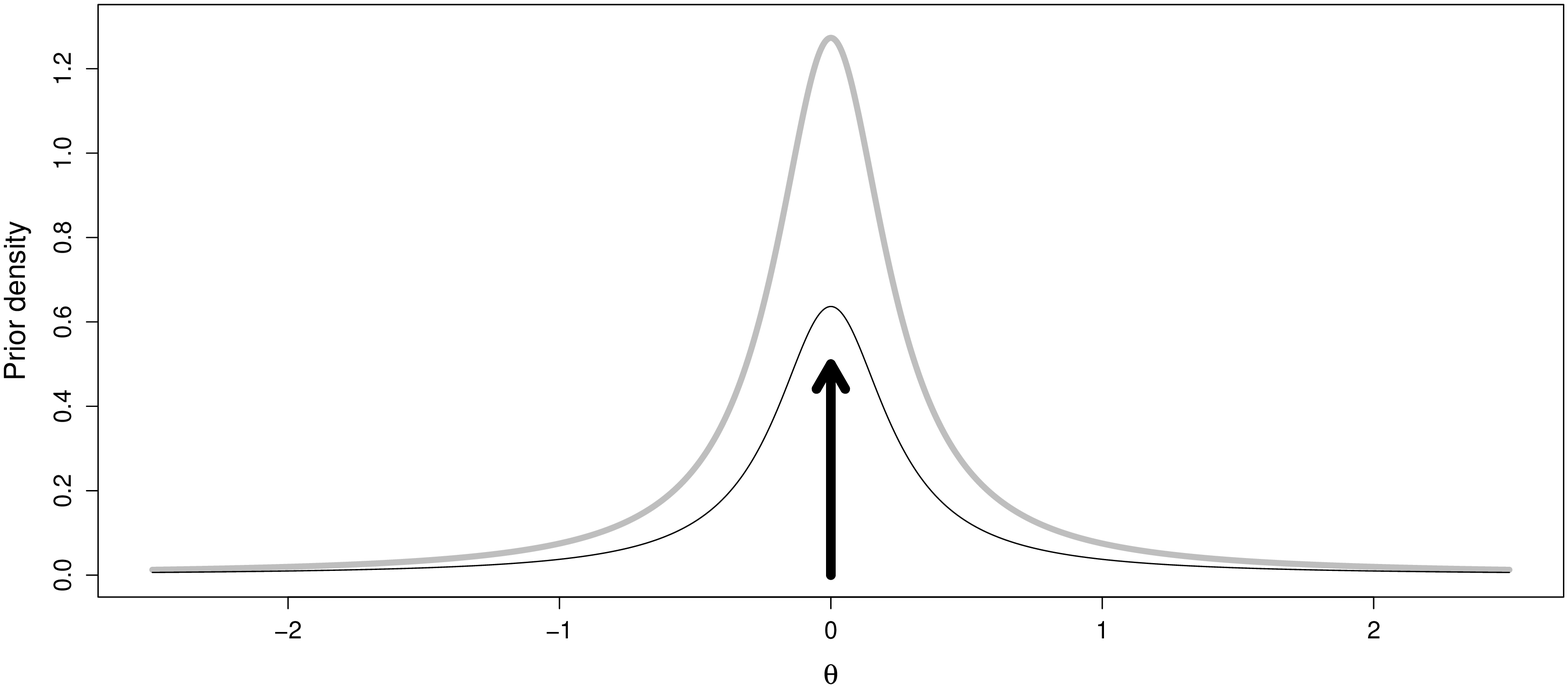} \\
\includegraphics[width=0.7\textwidth,height=0.4\textwidth]{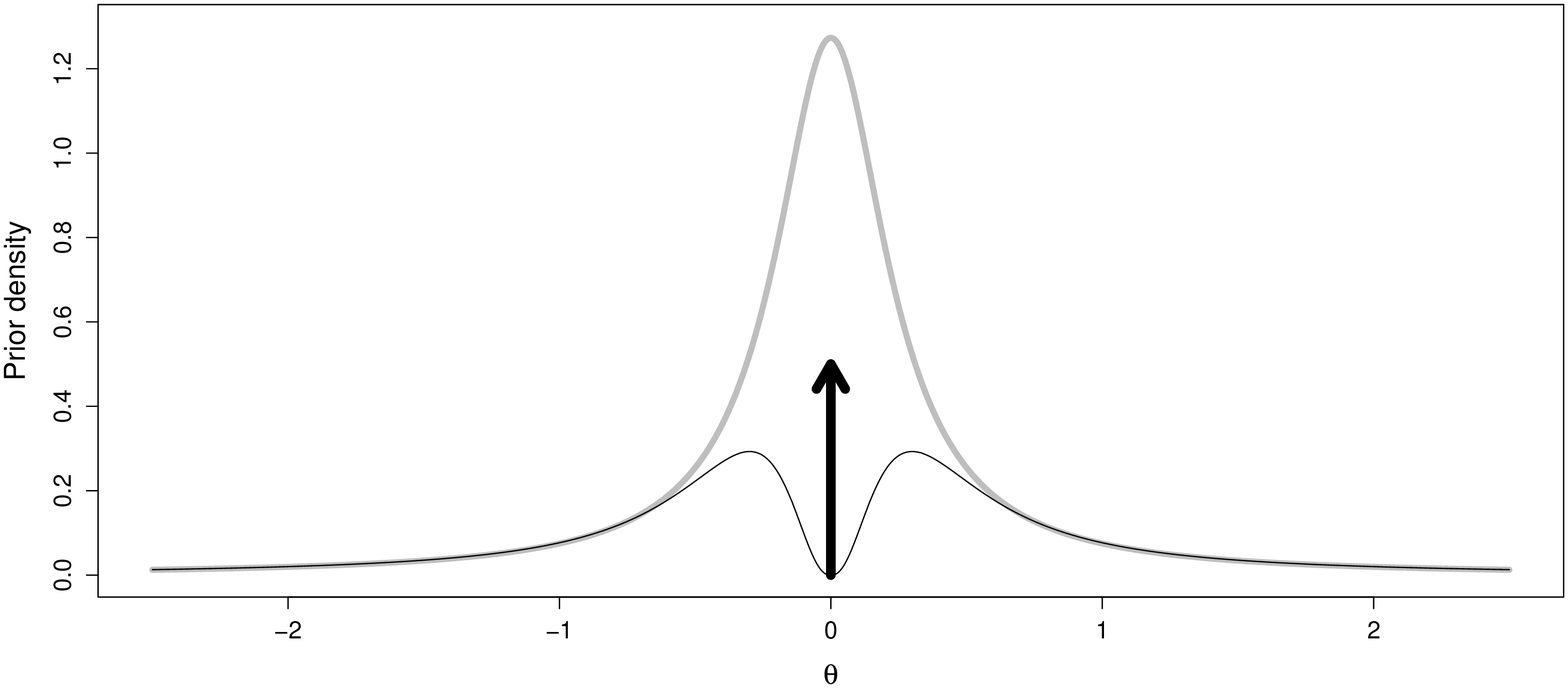} \\
\end{tabular}
\end{center}
\caption[Marginal priors for $\theta \in \real$]{Marginal priors for $\theta \in \real$
(estimation prior $\mbox{Cauchy}(0,0.0625)$ shown in grey).
Top: mixture of point mass at 0 and $\mbox{Cauchy}(0,0.0625)$ truncated at $\lambda=0.25$;
Middle: same with un-truncated $\mbox{Cauchy}(0,0.0625)$;
Bottom: same as top with $\lambda \sim \mbox{IG}(3,10)$}
\label{fig:estvstest}
\end{figure}

We provide intuition with a simple example.
%In the Bayesian paradigm the prior depends only on
%one's knowledge (or lack thereof), but in practice the analysis goals may
%affect the prior choice, {\it e.g.} different priors are often used  for
%estimation and selection (see \citeasnoun{bernardo:2010} and references therein for an alternative decision-theoretic approach).
%While the practice deviates from the basic paradigm,
%some preferences are hard to formalize in the utility function and may be
%easier to account for in the prior, {\it e.g.} in high dimensions point masses may simplify interpretation.
Suppose we wish to both estimate $\theta \in \real$ and test  $H_0: \theta = 0$ {\it vs.} $H_1: \theta \neq 0$, 
and that we are comfortable with a (possibly vague) prior for the estimation problem. 
Figure \ref{fig:estvstest} (grey line) shows a  $\mbox{Cauchy}(0,0.25)$  prior expressing confidence that $\theta$ is close to 0, {\it e.g.}  $P(|\theta| > 0.25)=0.5$. %  and  $P(|\theta| < 3)=0.947$. 
A testing prior assigns positive $P(\theta=0)$, but to be consistent we aim to preserve the estimation prior as much as possible. 
%This could indicate a belief that $\theta=0$ or simply reflect that $|\theta| < \lambda$ are irrelevant, 
We set a practical significance threshold $\lambda=0.25$ and combine a point mass at 0 with a  $\mbox{Cauchy}(0,0.25)$ 
truncated to exclude $(-0.25,0.25)$, with $P(H_0)=P(H_1)=0.5$ (Figure \ref{fig:estvstest}(top), black line).
This assigns the same $P(|\theta| > \theta_0)$ as before for $\theta_0 \geq 0.25$ and concentrates all probability in $(-0.25,0.25)$ at $\theta=0$.
Truncated priors have been discussed before \cite{verdinelli:1996,rousseau:2007,klugkist:2007}. 
They encourage coherence between estimation and testing, but they cannot detect small but non-zero coefficients.
Instead, most Bayesian tests use non-truncated priors
\cite{jeffreys:1961,zellner:1984,kass:1995,ohagan:1995,moreno:1998,perez:2002,bayarri:2007,liang:2008}.
Figure \ref{fig:estvstest} (middle) combines a  $\mbox{Cauchy}(0,0.25)$ 
with a point mass at 0. It is much more concentrated around 0, {\it e.g.}  $P(|\theta| > 0.25)$ decreased from 0.5 to 0.25. 
%Setting a larger scale parameter for the Cauchy would not  fix the issue, as the Cauchy mode would remain at 0.
We view this discrepancy between estimation and testing priors
as troublesome, as their underlying beliefs cannot be easily reconciled.
Suppose that we go back to the truncated Cauchy and set
$\lambda \sim G(2.5,10)$ ($E(\lambda)=0.25$) to express our uncertainty about $\lambda$.
Figure \ref{fig:estvstest} (bottom) shows the marginal prior on $\theta$ after integrating out $\lambda$.
It is a smooth version of the truncated Cauchy that goes to 0 as $\theta \rightarrow 0$,
showing an example where a NLP arises as a mixture of truncated distributions (see Section \ref{sec:theory}).
Relative to the estimation prior,
most of the probability assigned to $\theta \approx 0$ is absorbed by the point mass,
and $P(|\theta| > \theta_0)$ is roughly preserved for $\theta_0 > 0.5$.
%In contrast with the truncated Cauchy, it avoids fixing $\lambda$ and may detect any $\theta \neq 0$.
An advantage is that testing and estimation are conducted under the same framework.
Also, as we now show $E(\theta \mid {\bf y}_n)=E(\theta \mid H_1, {\bf y}_n)P({\bf y}_n \mid H_1)$ induces strong shrinkage.

We note that any NLP can be written as
%A straightforward way to specify a NLP under model $M_k$ is to set
$\pi(\bm{\theta}_k,\phi_k \mid M_k) \propto d(\bm{\theta}_k,\phi_k) \pi^L(\bm{\theta}_k,\phi_k \mid M_k)$,
where $d_k(\bm{\theta}_k,\phi_k) \rightarrow 0$ 
as $\bm{\theta}_k \rightarrow \bm{\theta}_0$ for any $\bm{\theta}_0 \in \Theta_{k'} \subset \Theta_k$
and $\pi^L(\bm{\theta}_k,\phi_k)$ is a LP.
To ensure that $\pi(\bm{\theta}_k,\phi_k \mid M_k)$ is proper
we assume $\int_{}^{}\!\, d_k(\bm{\theta}_k,\phi_k) \pi^L(\bm{\theta}_k\mid \phi_k,M_k) d \bm{\theta}_k < \infty$.
NLPs are often expressed in this form ((\ref{eq:pmom}) or (\ref{eq:pemom})),
but the representation is always possible since
$\pi(\bm{\theta}_k,\phi_k\mid M_k)=
\frac{\pi(\bm{\theta}_k,\phi_k\mid M_k)}{\pi^L(\bm{\theta}_k,\phi_k\mid M_k)}
\pi^L(\bm{\theta}_k,\phi_k \mid M_k)= d_k(\bm{\theta}_k,\phi_k) \pi^L(\bm{\theta}_k,\phi_k\mid M_k)$.
%where $d_k(\bm{\theta}_k,\phi_k)= \frac{\pi(\bm{\theta}_k,\phi_k)}{\pi^L(\bm{\theta}_k,\phi_k)}$.
Throughout we assume that $\pi(\phi_k \mid M_k)$ is bounded for all $\phi_k$.
The following result guarantees that $P(M_k \mid {\bf y}_n)$ under NLPs
induce an additional penalization for overly complex models.
All proofs are provided in the Appendix.

\begin{prop}
Let $m_k({\bf y}_n),m_k^L({\bf y}_n)$ be the integrated likelihoods for a NLP and the corresponding LP
under model $M_k$ for $k=1,\ldots,K$, as above. For $k=1,\ldots,K$, 
\begin{enumerate}[(i)]
\item Let $g_k({\bf y}_n)= \int_{}^{}\!\, \int_{}^{}\!\, d_k(\bm{\theta}_k,\phi_k) \pi^L(\bm{\theta}_k,\phi_k \mid {\bf y}_n) d \bm{\theta}_k d\phi_k$ be
the mean of $d_k(\bm{\theta}_k,\phi_k)$ under the LP posterior. Then
$m_k({\bf y}_n)= m_k^L({\bf y}_n) g_k({\bf y}_n)$.

\item
Consider the peMOM or piMOM priors under $M_k$ with fixed $|k|$.
Let $A \subset \Theta_k\times \Phi$ be such that
$f_k({\bf y}_n \mid \bm{\theta}_k^*,\phi_k^*)$ for any $(\bm{\theta}_k^*,\phi_k^*) \in A$
minimizes KL divergence to the data-generating density $f^*({\bf y}_n)$, and assume that 
for any $(\tilde{\bm{\theta}}_k,\tilde{\phi}_k) \not\in A$ as $n \rightarrow \infty$
$$
\frac{f_k({\bf y}_n \mid \bm{\theta}_k^*, \phi_k^*)}{f_k({\bf y}_n \mid \tilde{\bm{\theta}}_k, \tilde{\phi}_k)} \stackrel{a.s.}{\longrightarrow} \infty.
$$

If $A=\{(\bm{\theta}_k^*,\phi_k^*)\}$ is a singleton (identifiable models) then
$g_k({\bf y}_n) \stackrel{P}{\longrightarrow} d_k(\bm{\theta}_k^*,\phi_k^*)$.
For any $A$,
if $f^*({\bf y}_n)=f_t({\bf y}_n \mid \bm{\theta}_t^*,\phi_t^*)$ for some $t \in \{1,\ldots,K\}$ then
$g_k({\bf y}_n) \stackrel{P}{\longrightarrow} 0$ when $M_t \subset M_k$, $k \neq t$
and $g_k({\bf y}_n) \stackrel{P}{\longrightarrow} c>0$ when either $k=t$ or $M_t \not\subset M_k$.
The results also hold if $\pi^L(\bm{\theta}_k,\phi_k)$ is the pMOM prior
and $m_{k,\tau(1+\epsilon)}^L({\bf y}_n)/m_{k,\tau}^L({\bf y}_n) \stackrel{a.s.}{\longrightarrow} c \in (0,\infty)$ as $n \rightarrow \infty$,
where $m_{k,\tau}^L({\bf y}_n)$ is the integrated likelihood under a Normal prior with dispersion $\tau(1+\epsilon)$ and $\epsilon \in (0,1)$.
\end{enumerate}
\label{prop:parsimony_general}
\end{prop}

That is, NLPs add a term $g_k({\bf y}_n)$ that converges to 0 for unnecessarily complex models and to a finite constant otherwise.
The conditions in (ii) essentially require MLE consistency 
(see \citeasnoun{redner:1981} for general conditions that include non-identifiable models).
%The precise conditions are given in \citeasnoun{ghosal:2002} for most parametric and certain non-parametric models,
%\citeasnoun{leroux:1992} for mixture models,
%and \citeasnoun{redner:1981} for certain non-identifiable models where 
The pMOM condition on $m_{k,\tau(1+\epsilon)}^L({\bf y}_n)/m_{k,\tau}^L({\bf y}_n)$ is equivalent to
the ratio of posterior densities under $\tau$ and $\tau(1+\epsilon)$ at an arbitrary $(\bm{\theta}_k,\phi_k)$ converging to a constant (see proof).
This typically holds, {\it e.g.} under the conditions in \citeasnoun{walker:1969} for asymptotic posterior normality
or the linear models in Proposition \ref{prop:parsimony_lm},
where we allow the number of variables $|K|$ 
to grow arbitrarily fast with $n$, but only models with $O(n^{\alpha})$ variables are considered, ($\alpha < 1$).
Specifically, let
${\bf y}_n \sim N(X_{k,n} \bm{\theta}_k, \phi_k I)$ 
with $\bm{\theta}_k \in \Theta_k$ under $M_k$
and $|k|=O(n^{\alpha})$.
Let $S_{k,n}=X_{k,n}'X_{k,n}+\tau^{-1}I$,
${\bf m}_{k,n}=S_{k,n}^{-1}X_{k,n}' {\bf y}_n$
and $\hat{\bm{\theta}}_{k,n}= (X_{k,n}'X_{k,n})^{-1} X_{k,n}' {\bf y}_n$ be the least squares estimate.

\begin{prop}
Let $\pi(\bm{\theta}_k,\phi_k \mid M_k)$ be a NLP where either
$d_k(\bm{\theta}_k,\phi_k)=\prod_{i \in M_k}^{} \frac{\theta_{ki}^{2r}}{(2r-1)!!(\tau \phi_k)^r}$
or $d_k(\bm{\theta}_k,\phi_k)=\prod_{i \in M_k}^{} d(\theta_{ki},\phi_k)$
with $d_k(\theta_{ki},\phi_k) \leq c$ for all $i$ and some constant $c$.
Assume that there exist fixed $a,b,n_0>0$ such that
$a < \frac{1}{n}l_1(X_{k,n}'X_{k,n}) < \frac{1}{n}l_{k}(X_{k,n}'X_{k,n}) < b$
for all $n>n_0$, where $l_1,l_{k}$
denote the smallest and largest eigenvalues of $X_{k,n}'X_{k,n}$.
Let $(\bm{\theta}_k^*,\phi_k^*)$ minimize KL divergence to the data-generating density
with $\mbox{Var}\left({\bf y}_n-X_{k,n} \bm{\theta}_k^*\right)=\phi_k^*<\infty$.
Further, assume that $\pi(\phi_k \mid M_k)$ is continuous, bounded and $\pi(\phi_k^* \mid M_k)>0$.
Then
$$g_k({\bf y}_n) \stackrel{P}{\longrightarrow} d_k({\bf m}_{k,n}, \phi_k^*) \stackrel{a.s.}{\longrightarrow} d_k(\bm{\theta}_k^*,\phi_k^*).$$

Further, if the true density
$f^*({\bf y}_n \mid X_{K,n})=N({\bf y}_n; X_{t,n} \bm{\theta}_t^*,\phi_t^*)$ for some $t \in \{1,\ldots,K\}$ then
$g_k({\bf y}_n) \stackrel{P}{\longrightarrow} c$ with $c=0$ when $M_t \subset M_k$ and $c>0$ when $M_t \not\subset M_k$.
\label{prop:parsimony_lm}
\end{prop}

We note that the eigenvalue conditions are strongly related to MLE consistency \cite{lai:1979}.
So far we saw that NLPs improve model selection via an extra complexity penalty.
We now turn attention to parameter estimates conditional on a given $M_k$.
%For truly non-zero $\theta_i$ the NLP posterior modes and means
%approximate the MLE at a rate $n^{-1}$, while for truly zero parameters the MOM achieves the usual $n^{-1/2}$
%rate and eMOM or eMOM a $n^{-1/4}$ rate.

\begin{prop}
Let $M_k$ %$f_k({\bf y} \mid \bm{\theta}_k, \phi_k)$ 
with fixed $|k|$ satisfy the conditions in \citeasnoun{walker:1969}.
Let $(\hat{\bm{\theta}}_k,\hat{\phi}_k)$ be an MLE and
$f_k({\bf y}_n \mid \btheta_k^*,\phi_k^*)$ minimize KL divergence to data-generating
density $f_t({\bf y}_n \mid \btheta_t, \phi_t)$.

\begin{enumerate}[(i)]
\item Let $\tilde{\bm{\theta}}_k$ be the posterior mode such that $\mbox{sign}(\tilde{\theta}_{ki})=\mbox{sign}(\hat{\theta}_{ki})$
for all $i$ under either a pMOM, peMOM or piMOM prior.
If $\theta_{ki}^* \neq 0$, then $n(\tilde{\theta}_{ki} - \hat{\theta}_{ki}) \stackrel{P}{\longrightarrow} c$
for some $0<c<\infty$ for the pMOM, peMOM and piMOM priors.
If $\theta_{ki}^* =0$ then
$n^2(\tilde{\theta}_{ki} - \hat{\theta}_{ki})^2 \stackrel{P}{\longrightarrow} c$ for pMOM
and $n \tilde{\theta}_{ki}^4 \stackrel{P}{\longrightarrow} c$ for peMOM and piMOM
and (distinct) $0<c<\infty$.
Further, if the MLE is unique then any other posterior mode is $O_p(n^{-1/2})$ (pMOM) or $O_p(n^{-1/4})$ (peMOM, piMOM).

\item The posterior mean $E(\theta_{ki} \mid M_k, {\bf y}_n)= \hat{\theta}_{ki} + O_p(n^{-1/2})= \theta_{ki}^* + O_p(n^{-1/2})$
under a pMOM and $\hat{\theta}_{ki} + O_p(n^{-1/4})=\theta_{ki}^* + O_p(n^{-1/4})$ under a peMOM or piMOM prior.

\item Let ${\bf y}_n \sim N(X_{n,k} \bm{\theta}_k, \phi_k)$ satisfy the conditions in Proposition \ref{prop:parsimony_lm}
with $|k|=O(n^{\alpha})$, $\alpha<1$ and diagonal $X_{n,k}'X_{n,k}$.
Then the rates in (i)-(ii) remain valid.
\end{enumerate}
\label{prop:postmode}
\end{prop}

Conditional on $M_k$ spurious parameter estimates under NLPs converge to 0 at
either the same (pMOM) or slightly slower rate (peMOM,piMOM) than the MLE.
As shown in the next proposition, BMA combines these estimates with weights $P(M_k \mid {\bf y}_n)$
to achieve fast polynomial (pMOM) or quasi-exponential shrinkage rates (peMOM,piMOM).

\begin{prop}
Let $E(\theta_i \mid {\bf y}_n)$ be the BMA posterior mean in (\ref{eq:bma}), $\mbox{BF}_{kt}$ the Bayes factor between $M_k$ and $M_t$
and $\htheta_{ki}$ the MLE of $\theta_i$ under $M_k$.
\begin{enumerate}[(i)]
\item Assume that all $M_k$ satisfy Walker's conditions and $|K|$ is fixed.
Denote by $M_t$ the data-generating model and assume that $\mbox{KL}(M_t,M_k)>0$ for any $k$ such that $M_t \not\subset M_k$.
If $M_t \not\subset M_k$ then $\mbox{BF}_{kt}=O_p(e^{-n})$ under a pMOM, peMOM or piMOM prior.
If $M_t \subset M_k$ then
$\mbox{BF}_{kt}=O_p(n^{-\frac{3}{2}(|k|-|t|)})$ under a pMOM prior and
$\mbox{BF}_{kt}=O_p(e^{-\sqrt{n}})$ under either a peMOM or piMOM prior.

\item Assume the conditions in (i) and that $P(M_k)/P(M_t)=o(n^{(|k|-|t|)})$.
Let $\pi_{|t|+1}= \mbox{max}_k P(M_k)$ where $|k|=|t|+1$, $M_t \subset M_k$.
If $\theta_i^* \neq 0$ then $E(\theta_i \mid {\bf y}_n)= \htheta_{ti} + O_p(n^{-1})$
under pMOM, peMOM or piMOM priors.
If $\theta_i^*=0$ then under pMOM priors
$$E(\theta_i \mid {\bf y}_n)= O_p(n^{-2}) \frac{\pi_{|t|+1}}{P(M_t)}$$
and under peMOM or piMOM priors 
$$E(\theta_i \mid {\bf y}_n)= O_p(e^{-\sqrt{n}}) \frac{\pi_{|t|+1}}{P(M_t)}.$$

\item Consider linear models as in Proposition \ref{prop:postmode}(iii)
and known residual variance $\phi$. Let $\delta_i=\mbox{I}(\theta_i \neq 0)$ for $i=1,\ldots,|K|$ and 
assume that the prior $P(\delta_1,\ldots,\delta_p)$ is exchangeable.
If $\theta_i^* \neq 0$ then
$E(\theta_i \mid {\bf y}_n,\phi)= \htheta_{ti} + O_p(n^{-1})$ for pMOM, peMOM and piMOM.
If $\theta_i^*=0$ then
$E(\theta_i \mid {\bf y}_n,\phi)= O_p(n^{-2}) P(\delta_i=1)/P(\delta_i=0)$ for pMOM
and $E(\theta_i \mid {\bf y}_n,\phi)= e^{-\sqrt{n}O_p(1)} P(\delta_i=1)/P(\delta_i=0)$ for peMOM and piMOM.
\end{enumerate}
\label{prop:bma_postmean}
\end{prop}

When setting $\theta_i^{2r}$ in the pMOM prior the $O_p(n^{-2})$ for spurious coefficients becomes $O_p(n^{-\frac{3}{2}r-\frac{1}{2}})$,
and should be compared to an $O_p(n^{-1})$ shrinkage obtained with LPs.
Interestingly, this term only affects spurious coefficients (asymptotically).
The result also clarifies the role of sparse $P(M_k)$ in inducing even stronger shrinkage.

\section{Non-local priors as truncation mixtures}
\label{sec:theory}

We establish a correspondence between NLPs and truncation mixtures.
Our discussion is conditional on $M_k$, hence for simplicity we omit $\phi$
and denote $\pi(\bm{\theta})= \pi(\bm{\theta} \mid M_k)$, $p=\mbox{dim}(\Theta_k)$.
All proofs are in the Appendix.

\subsection{Equivalence between NLPs and truncation mixtures}
\label{ssec:equiv_nlp_trunc}

We show that truncation mixtures define valid NLPs, and subsequently that any NLP may be represented in this manner.
Given that the representation is not unique, we give two constructions and discuss their merits.
Let $\pi^L(\bm{\theta})$ be a LP on $\bm{\theta}$
and $\lambda \in \real^{+}$ a latent truncation point. 
\begin{prop}
Define $\pi(\bm{\theta} \mid \lambda) \propto \pi^L(\bm{\theta})
\mbox{I}(d(\bm{\theta}) > \lambda)$, where 
$\mathop {\lim }\limits_{\btheta \to \btheta_0} d(\bm{\theta})= 0$
for any $\bm{\theta}_0 \in \Theta_{k'} \subset \Theta_k$,
and $\pi^L(\bm{\theta})$ is bounded in a neighborhood of $\bm{\theta}_0$.
Let $\pi(\lambda)$ be a marginal prior for $\lambda$
placing no probability mass at $\lambda=0$.
Then
$\pi(\bm{\theta})= \int_{}^{}\!\, \pi(\bm{\theta} \mid \lambda) \pi(\lambda) d\lambda$
defines a NLP.
\label{prop1}
\end{prop}

\begin{cor}
Assume that $d(\bm{\theta})= \prod_{i=1}^{p} d_i(\theta_i)$.
Let $\pi(\bm{\theta} \mid \bm{\lambda}) \propto \pi^L(\bm{\theta}) \prod_{i=1}^{p} \mbox{I} \left( d_i(\theta_i) > \lambda_i \right)$
where $\bm{\lambda}=(\lambda_1,\ldots,\lambda_p)'$  
have  an absolutely continuous prior $\pi(\bm{\lambda})$.
Then $\int_{}^{}\!\, \pi(\bm{\theta} \mid \bm{\lambda}) \pi(\bm{\lambda}) d \bm{\lambda}$
is a NLP.
\label{cor2}
\end{cor}

This alternative representation can be convenient for sampling (as illustrated later on)
or to avoid the marginal dependency between elements in $\bm{\theta}$ induced by a common truncation.

\begin{example}
Consider ${\bf y}_n \sim N(X \bm{\theta}, \phi I)$, where
$\bm{\theta} \in \real^p$, $\phi$ is known and $I$ is the $n \times n$ identity matrix.
We define a NLP  for $\bm{\theta}$ with a single truncation point with
 $\pi(\bm{\theta} \mid \lambda) \propto N(\bm{\theta}; {\bf 0}, \tau I)
I \left( \prod_{i=1}^{p} \theta_i^2 > \lambda \right)$
and some $\pi(\lambda)$, {\it e.g.} Gamma or Inverse Gamma.
Obviously, the choice of $\pi(\lambda)$ affects $\pi(\bm{\theta})$ (Section \ref{ssec:nlp_properties}).
An alternative prior is
$\pi(\bm{\theta} \mid \lambda_1,\ldots,\lambda_p) \propto N(\bm{\theta}; {\bf 0}, \tau I)
\prod_{i=1}^{p} I \left( \theta_i^2 > \lambda_i \right)$,
giving marginal independence when $\pi(\lambda_1,\ldots,\lambda_p)$ has independent components.
\label{ex1}
\end{example}

We address the reverse question: given any NLP, a truncation representation is always possible.
\begin{prop}
Let $\pi(\bm{\theta}) \propto d(\bm{\theta}) \pi^L(\bm{\theta})$
be an arbitrary NLP
and denote $h(\lambda)=P_u \left(d(\bm{\theta}) > \lambda \right)$,
where $P_u(\cdot)$ is the probability under $\pi^L(\bm{\theta})$.
Then $\pi(\bm{\theta})$ is the marginal prior associated to
$\pi(\bm{\theta} \mid \lambda) \propto \pi^L(\bm{\theta})
\mbox{I}(d(\bm{\theta}) > \lambda)$
and 
$\pi(\lambda) = h(\lambda) / E_u \left( d( \bm{\theta} ) \right) \propto h(\lambda),$
where $E_u \left( \cdot \right)$ is the expectation
with respect to $\pi^L(\bm{\theta})$.
\label{prop3}
\end{prop}

\begin{cor}
Let $\pi(\bm{\theta}) \propto \pi^L(\bm{\theta}) \prod_{i=1}^{p} d_i(\theta_i)$ be a NLP,
$h(\bm{\lambda})= P_u(d_1(\theta_1) > \lambda_1, \ldots,d_p(\theta_p) >\lambda_p)$
and assume that $\int_{}^{}\!\, h(\bm{\lambda}) d \bm{\lambda} < \infty$.
Then $\pi(\bm{\theta})$ is the marginal prior associated to 
$\pi(\bm{\theta} \mid \bm{\lambda}) \propto \pi^L(\bm{\theta}) 
\prod_{i=1}^{p} \mbox{I}(\theta_i > \lambda_i)$ and
$\pi(\bm{\lambda}) \propto h(\bm{\lambda})$. 
\label{cor4}
\end{cor}

The advantage of Corollary \ref{cor4} is that,
in spite of introducing additional latent variables, it greatly facilitates sampling.
The condition that $\int_{}^{}\!\, h(\bm{\lambda}) d\bm{\lambda}< \infty$
is guaranteed when $\pi^L(\bm{\theta})$ has independent components
(apply Proposition \ref{prop3} to each univariate marginal).
%As an illustration we consider prior sampling (Section \ref{sec:postsampling} discusses posterior sampling).

\begin{example}
The pMOM prior with
$d(\bm{\theta})= \prod_{i=1}^{p} \theta_{i}^2$, $\pi^L(\bm{\theta})= N(\bm{\theta}; {\bf 0}, \tau \mbox{I})$
can be represented as
$\pi(\bm{\theta} \mid \lambda) \propto N(\bm{\theta}; {\bf 0}, \tau \mbox{I}) 
\mbox{I}\left( \prod_{i=1}^{p} \theta_i^2 > \lambda \right) $
and
 $$\pi(\lambda) = \frac{P(\prod_{i=1}^{p} \theta_{i}^2/\tau > \lambda/\tau^p)}{E_u \left( \prod_{i=1}^{p} \theta_i ^2 \right)} 
= \frac{h(\lambda/\tau^p)}{\tau^p},$$
where $h(\cdot)$ is the survival function for a product of independent
chi-square random variables with 1 degree of freedom \cite{springer:1970}.
%In principle one could 
Prior draws are obtained by
\begin{enumerate}
\item Draw $u \sim \mbox{Unif}(0,1)$. Set $\lambda = P^{-1}(u)$, where $P(u)= P_{\pi}(\lambda \leq u)$ is the cdf associated to $\pi(\lambda)$.
\item Draw $\bm{\theta} \sim N({\bf 0}, \tau I) \mbox{I} \left( d(\bm{\theta}) > \lambda \right)$.
\end{enumerate}
As drawbacks, $P(u)$ requires Meijer G-functions and is cumbersome to evaluate for large $p$
and sampling from a multivariate Normal with truncation region $\prod_{i=1}^{p} \theta_i^2 > \lambda$ is non-trivial.

\begin{figure}
\begin{center}
\begin{tabular}{cc}
\includegraphics[width=0.5\textwidth]{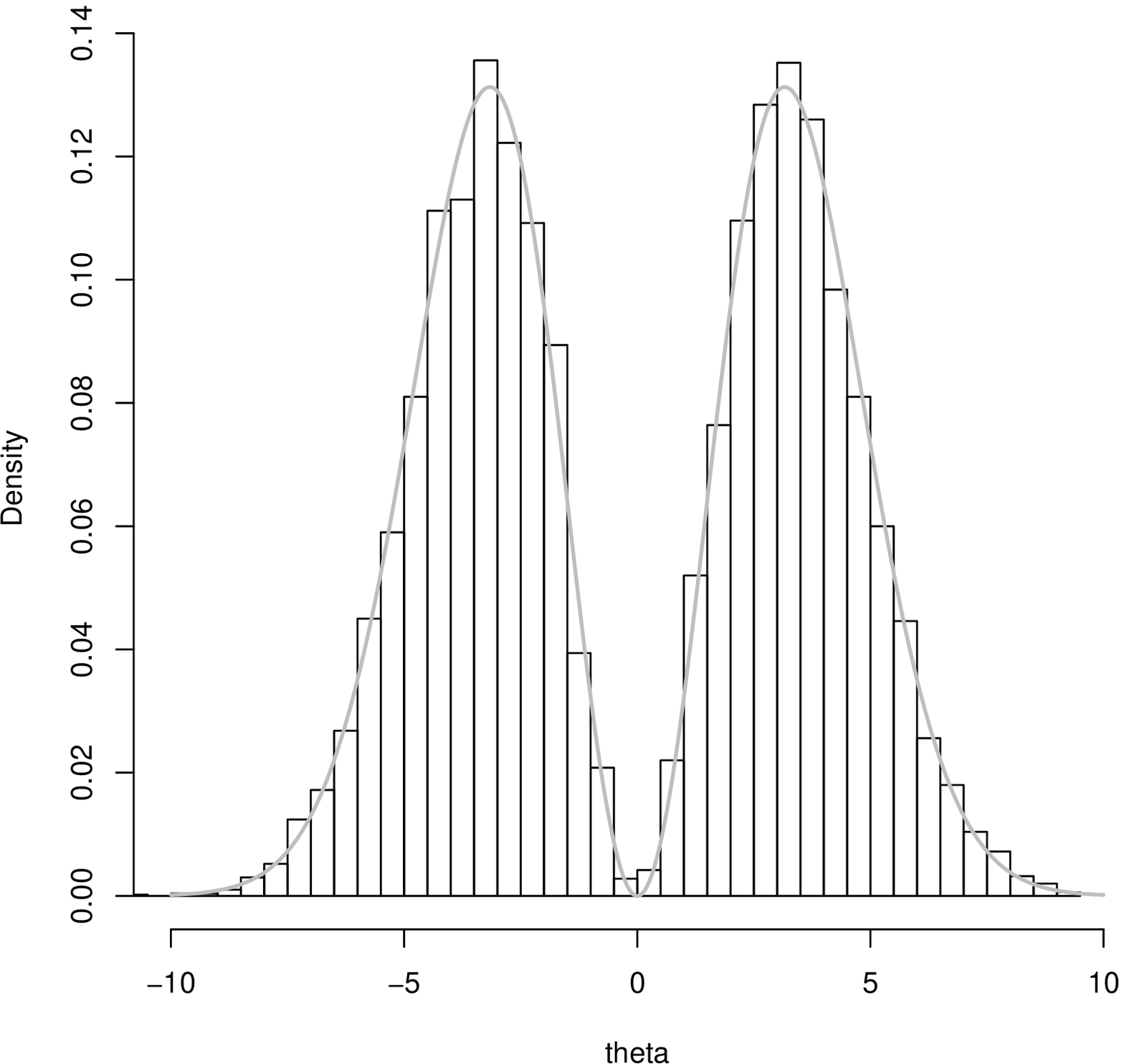} &
\includegraphics[width=0.5\textwidth]{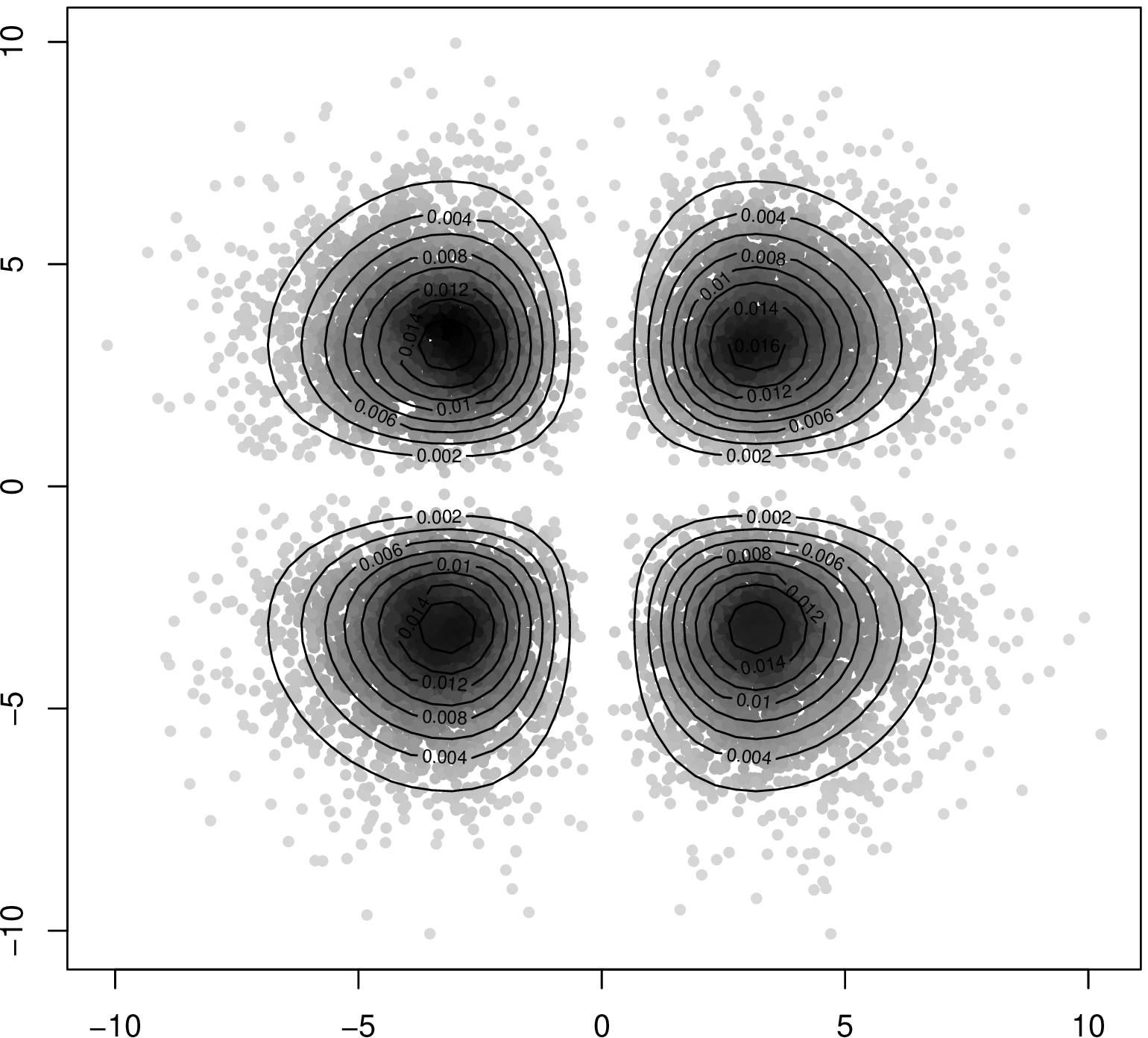} \\
\end{tabular}
\end{center}
\caption[10,000 independent MOM prior draws ($\tau=5$)]{10,000 independent univariate (left) and bivariate (right) pMOM prior draws ($\tau=5$). 
Lines indicate true density.}
\label{fig:sim_momprior}
\end{figure}

Corollary \ref{cor4} gives an alternative.
Let $P(u)=P(\lambda < u)$ be the cdf associated to $\pi(\lambda)=\frac{h(\lambda/\tau)}{\tau}$ 
where $h(\cdot)$ is the survival of a $\chi_1^2$ distribution.
For $i=1,\ldots,p$, draw $u_i \sim \mbox{Unif}(0,1)$, set $\lambda_i = P^{-1}(u_i)$ and
draw $\theta_i \sim N(0,\tau) \mbox{I}(\theta_i > |\lambda_i|)$.
The function $P^{-1}(\cdot)$ can be tabulated and quickly evaluated, rendering efficient computations.
Figure \ref{fig:sim_momprior} shows 100,000 draws from 
univariate (left) and bivariate (right) pMOM priors with $\tau=5$.

\end{example}

\subsection{Deriving NLP properties for a given mixture}
\label{ssec:nlp_properties}

We establish how two important characteristics of a NLP functional form,
the penalty $d(\bm{\theta})$ and its tail behavior, depend on a given truncation scheme.
It is necessary to distinguish whether a single or multiple truncation variables are used.
%For a common truncation variable $\lambda$, shared across $\theta_1,\ldots,\theta_p$, the following holds.

\begin{prop}
Let $\pi(\bm{\theta})$ be the marginal NLP for
$\pi(\bm{\theta},\lambda) = \frac{\pi^L(\bm{\theta})}{h(\lambda)} 
\big( \prod_{i=1}^{p} \mbox{I}(d( \theta_i ) > \lambda) \big) \pi(\lambda)$,
where $h(\lambda)= P_u(d( \theta_1 ) > \lambda, \ldots, d( \theta_p ) > \lambda)$,
$\pi(\lambda)$ is absolutely continuous and $\lambda \in \real^+$.
Denote $d_{min}(\btheta) = \mbox{min} \{ d(\theta_1), \ldots, d(\theta_p) \} $.

\begin{enumerate}[(i)]
\item Consider any sequence $\{ \btheta^{(m)} \}_{m \geq 1}$ such that $\mathop {\lim }\limits_{m \to \infty } d_{min}(\btheta^{(m)})=0$.
Then
$$\mathop {\lim }\limits_{m \to \infty } \frac{\pi(\bm{\theta}^{(m)})}{\pi^L(\btheta^{(m)}) d_{min}(\btheta^{(m)}) \pi(\lambda^{(m)})}=1,$$
for some $\lambda^{(m)} \in (0, d_{min}(\btheta^{(m)}))$.
If $\pi(\lambda)= c h(\lambda)$ then $\mathop {\lim }\limits_{m \to \infty } \pi(\lambda^{(m)})=c \in (0,\infty)$.

\item Let $\{ \btheta^{(m)} \}_{m \geq 1}$ be any sequence such that $\mathop {\lim }\limits_{m \to \infty } d(\bm{\theta}^{(m)})= \infty$.
Then
$\mathop {\lim }\limits_{m \to \infty } \pi(\bm{\theta}^{(m)})/\pi^L(\bm{\theta}^{(m)})=c$ where $c>0$ is either a positive constant or $\infty$.
In particular, if $\int_{}^{}\!\, \frac{\pi(\lambda)}{h(\lambda)} d \lambda < \infty$
then $c<\infty$.
\end{enumerate}
\label{prop5}
\end{prop}

Property (i) is important as asymptotic Bayes factor rates depend on the form of the penalty,
given by $d_{min}(\btheta) \pi^L(\btheta)$ and hence
only depending on the smallest $d(\theta_1),\ldots,d(\theta_p)$.
Property (ii) shows that $\pi(\bm{\theta})$ inherits its tail behavior from $\pi^L(\bm{\theta})$.
%which helps avoid finite sample inconsistencies \cite{liang:2008}.
Corollary \ref{cor6} is an extension to multiple truncations.

\begin{cor}
Let $\pi(\bm{\theta})$ be the marginal NLP for
$\pi(\bm{\theta},\bm{\lambda}) = \frac{\pi^L(\bm{\theta})}{h(\bm{\lambda})}
\prod_{i=1}^{p} \mbox{I}(d_i( \theta_i ) > \lambda_i) \pi_i(\lambda_i)$,
where $h(\bm{\lambda})= P_u\left( d_1(\theta_1) > \lambda_1, \ldots, d_p(\theta_p) > \lambda_p \right)$
under $\pi^L(\bm{\theta})$ and $\pi(\bm{\lambda})$ is absolutely continuous.

\begin{enumerate}[(i)]
\item Let $\{ \bm{\theta}^{(m)} \}_{m\geq 1}$ such that $\mathop {\lim }\limits_{m \to \infty } d_i(\theta_i^{(m)})=0$ for $i=1,\ldots,p$.
Then for some $\lambda_i^{(m)} \in (0, d( \theta_i ))$,
$\mathop {\lim }\limits_{m \to \infty } \pi(\bm{\theta}^{(m)})/\left(\pi^L(\bm{\theta}^{(m)}) \pi(\bm{\lambda}^{(m)}) \prod_{i=1}^{p} d_i(\theta_i^{(m)}) \right) =1$.
%where $\pi(\lambda_i^*)$ is the marginal prior for $\lambda_i$.

\item Let $\{ \btheta^{(m)} \}_{m\geq 1}$ such that $\mathop {\lim }\limits_{m \to \infty } d_i(\theta_i^{(m)})= \infty$ for $i=1,\ldots,p$.
Then
$\mathop {\lim }\limits_{m \to \infty } \pi(\bm{\theta}^{(m)})/\pi^L(\bm{\theta}^{(m)})=c$ where $c>0$ is either a positive constant or $\infty$.
In particular, if $E \left( h(\bm{\lambda})^{-1} \right) < \infty$ under the prior on $\bm{\lambda}$, then $c<\infty$.
\end{enumerate}
\label{cor6}
\end{cor}

That is, multiple independent truncation variables give a multiplicative penalty $\prod_{i=1}^{p} d_i(\theta_i)$
and tails are at least as thick as those of $\pi^L(\bm{\theta})$.
%, {\it i.e.} $\lambda_i \sim \pi(\cdot)$ for $i=1,\ldots,p$  induce stronger parsimony than a common  $\lambda \sim \pi(\cdot)$. 
Once a functional form for $\pi(\btheta)$ is chosen, we need to set its parameters.
Although the asymptotic rates (Section \ref{sec:asymptotics}) hold for any fixed parameters,
their value can be relevant in finite samples.
Given that posterior inference depends solely on the marginal prior $\pi(\bm{\theta})$,
whenever possible we recommend eliciting $\pi(\bm{\theta})$ directly.
For instance, \citeasnoun{johnson:2010} defined practical significance in linear regression as signal-to-noise ratios $|\theta_i|/\sqrt{\phi}>0.2$,
and gave default $\tau$ assigning $P(|\theta_i|/\sqrt{\phi}>0.2)=0.99$.
\citeasnoun{rossell:2013b} found analogous $\tau$ for probit regression,
and also considered learning $\tau$ either via a hyper-prior or minimizing posterior predictive loss \cite{gelfand:1998}.
\citeasnoun{consonni:2010} devised objective Bayes strategies.
Yet another possibility is to match the unit information prior {\it e.g.} setting $V(\theta_i/\sqrt{\phi})=1$,
which can be regarded as minimally informative (in fact $V(\theta_i/\sqrt{\phi})=1.074$ for the MOM default $\tau=0.358$).
When $\pi(\bm{\theta})$ is not in closed-form prior elicitation depends both on $\tau$ and $\pi(\lambda)$,
but prior draws can be used to estimate $P(|\theta_i|/\sqrt{\phi}>t)$ for some $t$ or $V(\theta_i/\sqrt{\phi})$.
An analytical alternative is to set $\pi(\lambda)$ so that $E(\lambda)=d(\theta_i,\phi)$ when $\theta_i/\sqrt{\phi}=t$,
{\it i.e.} $E(\lambda)$ matches a practical relevance threshold.
For instance, for $t=0.2$  and $\pi(\lambda) \sim \mbox{IG}(a,b)$ under the MOM prior we would set
$E(\lambda)=b/(a-1)=0.2^2/\tau$, and under the eMOM prior $b/(a-1)=e^{\sqrt{2}-\tau/0.2^2}$.
Both expressions illustrate the dependence between $\tau$ and $\pi(\lambda)$.
Here we use default $\tau$ (Section \ref{sec:examples}), but as discussed other strategies are possible.

\section{Posterior sampling}
\label{sec:postsampling}

We use the latent truncation characterization to derive posterior sampling algorithms, and show how the truncation mixture
in Proposition \ref{prop3} and Corollary \ref{cor4} leads to simplifications. Section \ref{ssec:generalgibbs} provides
two Gibbs algorithms to sample from arbitrary posteriors, and Section \ref{ssec:gibbslm} adapts them to linear models.
Sampling is conditional on a given $M_k$, hence we drop $M_k$ to keep notation simple.

\subsection{General algorithm}
\label{ssec:generalgibbs}

First consider a NLP defined by a single latent truncation, {\it i.e.}
$\pi(\bm{\theta} \mid \lambda) = \pi^L(\bm{\theta}) \mbox{I}(d( \bm{\theta} ) > \lambda) / h(\lambda)$,
where $h(\lambda)= P_u(d( \bm{\theta} ) > \lambda)$
and $\pi(\lambda)$ is a prior on $\lambda \in \real^+$.
The joint posterior is
%$\pi(\bm{\theta},\lambda \mid {\bf y})$ is proportional to
\begin{align}
\pi(\bm{\theta},\lambda \mid {\bf y}) \propto f({\bf y} \mid \bm{\theta}) \frac{\pi^L(\bm{\theta}) \mbox{I}(d( \bm{\theta} ) > \lambda)}{h(\lambda)} \pi(\lambda).
\label{eq:jointpost_singletrunc}
\end{align}
 
Sampling from $\pi(\bm{\theta} \mid {\bf y})$ directly is challenging
as it is highly multi-modal, but
straightforward algebra gives the following
$k^{th}$ Gibbs iteration to sample from $\pi(\bm{\theta}, \lambda \mid {\bf y})$.

\vspace{2mm}
{\bf Algorithm 1. Gibbs sampling with a single truncation}

\begin{enumerate}
\item Draw $\lambda^{(k)} \sim \pi(\lambda \mid {\bf y}, \bm{\theta}^{(k-1)}) \propto \mbox{I}(d(\bm{\theta}) > \lambda) \pi(\lambda)/h(\lambda)$.
When $\pi(\lambda) \propto h(\lambda)$ as in Proposition \ref{prop3}, $\lambda^{(k)} \sim \mbox{Unif}(0,d(\bm{\theta}^{(k-1)}))$.
\item Draw $\bm{\theta}^{(k)} \sim \pi(\bm{\theta} \mid {\bf y}, \lambda^{(k)}) \propto \pi^L(\bm{\theta} \mid {\bf y}) \mbox{I}(d(\bm{\theta}) > \lambda^{(k)})$.
\end{enumerate}

That is, $\lambda^{(k)}$ is sampled from a univariate distribution
that reduces to a uniform when setting $\pi(\lambda) \propto h(\lambda)$,
and $\bm{\theta}^{(k)}$ from a truncated version of $\pi^L(\cdot)$.
For instance, $\pi^L(\cdot)$ may be a LP that allows easy
posterior sampling.
As a difficulty, the truncation region $\{ \bm{\theta}: d(\bm{\theta}) >
\lambda^{(k)} \}$ is non-linear
and non-convex
so that jointly sampling $\bm{\theta}=(\theta_1,\ldots,\theta_p)$ may be challenging.
One may apply a Gibbs step to each element in $\theta_1,\ldots,\theta_p$ sequentially,
which only requires univariate truncated draws from $\pi^L(\cdot)$,
but the mixing of the chain may suffer.

The multiple truncation representation in Corollary \ref{cor4}
provides a convenient alternative.
Consider $\pi(\bm{\theta} \mid \bm{\lambda}) = \pi^L(\bm{\theta})
\prod_{i=1}^{p} \mbox{I}(d_i( \theta_i ) > \lambda_i) \pi(\bm{\lambda}) / h(\bm{\lambda})$,
where $h(\bm{\lambda})= P_u( d_1(\theta_1) > \lambda_1, \ldots d_p(\theta_p) >\lambda_p)$.
%and $\pi(\bm{\lambda})$ is an arbitrary prior.
The following steps define the $k^{th}$ Gibbs iteration:

\vspace{2mm}
{\bf Algorithm 2. Gibbs sampling with multiple truncations}

\begin{enumerate}
\item Draw $\bm{\lambda}^{(k)} \sim \pi(\bm{\lambda} \mid {\bf y},
  \bm{\theta}^{(k-1)}) = 
\prod_{i=1}^{p} \mbox{Unif}(\lambda_i; 0,d_i(\theta_i)) \frac{\pi(\bm{\lambda})}{h(\bm{\lambda})}$.
If $\pi(\bm{\lambda}) \propto h(\bm{\lambda})$ as in Corollary \ref{cor4}, 
$\lambda_i^{(k)} \sim \mbox{Unif}(0,d_i(\theta_i))$.

\item Draw $\bm{\theta}^{(k)} \sim \pi(\bm{\theta} \mid {\bf y}, \bm{\lambda}^{(k)}) \propto \pi^L(\bm{\theta} \mid {\bf y}) \prod_{i=1}^{p} \mbox{I}(d_i(\theta_i) > \lambda_i^{(k)})$
\end{enumerate}

Now the truncation region in Step 2 is  defined by hyper-rectangles,
which facilitates sampling.
As in Algorithm 1, by setting the prior conveniently
Step 1 avoids evaluating
$\pi(\bm{\lambda})$ and $h(\bm{\lambda})$.

\subsection{ Linear models }
\label{ssec:gibbslm}

We adapt Algorithm 2 to a linear regression
${\bf y} \sim N(X \bm{\theta}, \phi I)$ with unknown variance $\phi$
and the three priors in (\ref{eq:pmom})-(\ref{eq:pemom}).
We set the prior $\phi \sim \mbox{IG}(a_{\phi}/2,b_{\phi}/2)$
and let $\tau$ be a user-specified prior dispersion.
To set a hyper-prior on $\tau$ see \citeasnoun{rossell:2013b}.

For all three priors, Step 2 in Algorithm 2 
%is detailed in separate sections below.
samples from a multivariate Normal with rectangular truncation around
${\bf 0}$,
for which we developed an efficient algorithm.
\citeasnoun{kotecha:1999} and \citeasnoun{rodriguezyam:2004} proposed Gibbs after orthogonalization
strategies that result in low serial correlation,
which \citeasnoun{wilhelm:2010} implemented in the R package \texttt{tmvtnorm}
for restrictions $l \leq \theta_i \leq u$.
Here we require sampling under $d_i( \theta_i ) \geq l$, a non-convex region.
Our adapted algorithm is in Appendix \ref{app:alg_tnorm_sample}
and implemented in R package \texttt{mombf}.
An important property is that the algorithm produces independent samples when 
the posterior probability of the truncation region becomes negligible.
Since NLPs only assign high posterior probability to a model
when the posterior for non-zero coefficients is well shifted from the origin,
%Since models with high posterior probability under NLPs are typically well shifted from the origin,
the truncation region is indeed often negligible.
We outline the algorithm separately for each prior.

\subsubsection{\underline{pMOM prior.}}
\label{sssec:lmmom}

Straightforward algebra gives the full conditional posteriors
\begin{align}
\pi(\bm{\theta} \mid \phi, {\bf y}) \propto \left( \prod_{i=1}^{p} \theta_i^2 \right) N(\bm{\theta}; {\bf m}, \phi S^{-1}) \nonumber \\
\pi(\phi \mid \bm{\theta}, {\bf y})=\mbox{IG}\left(\frac{a_{\phi}+n+3p}{2},\frac{b_{\phi}+s_R^2+\bm{\theta}'\bm{\theta}/\tau}{2}\right),
\label{eq:pmom_post}
\end{align}
where $S=X'X+\tau^{-1} I$, ${\bf m}= S^{-1} X' {\bf y}$
and $s_R^2=({\bf y}-X \bm{\theta})'({\bf y}-X \bm{\theta})$ is the sum of squared residuals.
Corollary \ref{cor4} 
represents the pMOM prior in (\ref{eq:pmom}) as
\begin{align}
\pi(\bm{\theta} \mid \phi, \bm{\lambda}) = N(\bm{\theta};{\bf 0},\tau \phi I) \prod_{i=1}^{p} \mbox{I} \left( \frac{\theta_i^2}{\tau \phi} > \lambda_i \right) \frac{1}{h(\lambda_i)}
\label{eq:pmom_latent}
\end{align}
marginalized with respect to 
$\pi(\lambda_i) = h(\lambda_i)= P\left( \frac{\theta_i^2}{\tau \phi} > \lambda_i \mid \phi \right)$,
where $h(\cdot)$ is the survival of a chi-square with 1 degree of freedom.
Algorithm 2 and 
simple algebra give the $k^{th}$ Gibbs iteration
\begin{enumerate}
\item $\phi^{(k)} \sim \mbox{IG}(\frac{a_{\phi}+n+3p}{2},\frac{b_{\phi}+s_R^2+(\bm{\theta}^{(k-1)})'\bm{\theta}^{(k-1)}/\tau}{2})$
\item $\bm{\lambda}^{(k)} \sim \pi(\bm{\lambda} \mid \bm{\theta}^{(k-1)}, \phi^{(k)},{\bf y})= \prod_{i=1}^{p} \mbox{I} \left( \frac{(\theta_i^{(k-1)})^2}{\tau \phi^{(k)}} > \lambda_i \right)$
\item $\bm{\theta}^{(k)} \sim \pi(\bm{\theta} \mid \bm{\lambda}^{(k)}, \phi^{(k)}, {\bf y})= 
N(\bm{\theta}; {\bf m}, \phi^{(k)} S^{-1}) \prod_{i=1}^{p} \mbox{I} \left( \frac{\theta_i^2}{\tau \phi^{(k)}} > \lambda_i \right)$.
\end{enumerate}
Step 1 samples unconditionally on $\bm{\lambda}$, so that no efficiency is lost for introducing these latent variables.
Step 3 requires truncated multivariate Normal draws.

\subsubsection{\underline{piMOM prior.}}
\label{sssec:lmimom}

We assume $\mbox{dim}(\Theta)<n$.
The full conditional posteriors are
\begin{align}
\pi(\bm{\theta} \mid \phi, {\bf y}) \propto \left( \prod_{i=1}^{p} \frac{\sqrt{\tau \phi}}{\theta_i^2} e^{-\frac{\tau \phi}{\theta_i^2}} \right) N(\bm{\theta}; {\bf m}, \phi S^{-1}) \nonumber \\
\pi(\phi \mid \bm{\theta}, {\bf y})= e^{-\tau\phi \sum_{i=1}^{p} \theta_i^{-2} } \mbox{IG}\left(\phi; \frac{a_{\phi}+n-p}{2},\frac{b_{\phi}+s_R^2}{2}\right),
\label{eq:pimom_post}
\end{align}
where $S=X'X$, ${\bf m}= S^{-1} X' {\bf y}$
and $s_R^2=({\bf y}-X \bm{\theta})'({\bf y}-X \bm{\theta})$.
Now, the piMOM prior is $\pi_I(\bm{\theta} \mid \phi) = $
\begin{align}
N(\bm{\theta}; {\bf 0}; \tau_N \phi \mbox{I}) \prod_{i=1}^{p}
\frac{ \frac{\sqrt{\tau \phi}}{\sqrt{\pi} \theta_i^2} 
e^{ - \frac{\phi \tau}{\theta_i^2} }}
{N(\theta_i; 0, \tau_N \phi)}=
N(\bm{\theta}; {\bf 0}; \tau_N \phi \mbox{I}) \prod_{i=1}^{p} d_i(\theta_i,\phi).
%\prod_{i=1}^{p} \| \theta_i - \theta_{i0} \|
\label{eq:pimompen}
\end{align}
In principle any $\tau_N$ may be used, but $\tau_N \geq 2 \tau$ guarantees
$d(\theta_i,\phi)$ to be monotone increasing in $\theta_i^2$,
so that its inverse exists
(Appendix \ref{app:imompen}).
By default we set $\tau_N=2 \tau$.
Corollary \ref{cor4} gives
\begin{align}
\pi(\bm{\theta} \mid \phi, \bm{\lambda})= 
N(\bm{\theta}; {\bf 0}, \tau_N \phi \mbox{I}) \prod_{i=1}^{p} \mbox{I}( d(\theta_i,\phi) > \lambda_i ) \frac{1}{h(\lambda_i)}
\label{eq:pimom_latent}
\end{align}
and $\pi(\bm{\lambda})=\prod_{i=1}^{p} h(\lambda_i)$,
where $h(\lambda_i)=P(d(\theta_i,\phi) > \lambda_i)$
which we need not evaluate.
Algorithm 2 gives
the following MH within Gibbs procedure.
\begin{enumerate}
\item MH step
\begin{enumerate}
\item Propose $\phi^* \sim \mbox{IG}\left(\phi; \frac{a_{\phi}+n-p}{2},\frac{b_{\phi}+s_R^2}{2}\right)$
\item Set $\phi^{(k)}=\phi^*$ with probability $\mbox{min} \left\{ 1, e^{(\phi^{(k-1)}-\phi^*) \tau \sum_{i=1}^{p} \theta_i^{-2}}  \right\}$, else $\phi^{(k)}=\phi^{(k-1)}$.
\end{enumerate}

\item $\bm{\lambda}^{(k)} \sim 
\prod_{i=1}^{p} \mbox{Unif} \left(\lambda_i; 0, d(\theta_i^{(k-1)}, \phi^{(k)}) \right)$

\item $\bm{\theta}^{(k)} \sim 
N(\bm{\theta}; {\bf m}, \phi^{(k)} S^{-1}) \prod_{i=1}^{p} \mbox{I} \left(
  d(\theta_i, \phi^{(k)}) > \lambda_i^{(k)} \right)$.
\end{enumerate}
Step 3 requires the inverse $d^{-1}(\cdot)$,
which can be evaluated efficiently
combining an asymptotic approximation with a linear interpolation search
(Appendix \ref{app:imompen}).
As a token, 10,000 draws for $p=2$ variables required 0.58 seconds on a 2.8 GHz processor running OS X 10.6.8.

\subsubsection{\underline{peMOM prior.}}
\label{sssec:lmemom}

The full conditional posteriors are
\begin{align}
\pi(\bm{\theta} \mid \phi, {\bf y}) \propto \left( \prod_{i=1}^{p} e^{-\frac{\tau \phi}{\theta_i^2}} \right) N(\bm{\theta}; {\bf m}, \phi S^{-1});
\pi(\phi \mid \bm{\theta}, {\bf y}) \propto e^{- \sum_{i=1}^{p} \frac{\tau \phi}{\theta_i^2}}  \mbox{IG}\left(\phi; \frac{a^*}{2},\frac{b^*}{2}\right),
\label{eq:emom_post}
\end{align}
where 
$S=X'X+\tau^{-1} I$, ${\bf m}= S^{-1} X' {\bf y}$,
$a^*=a_{\phi}+n+p$,
$b^*=b_{\phi}+s_R^2+\bm{\theta}'\bm{\theta}/\tau$.
%and $s_R^2=({\bf y}-X \bm{\theta})'({\bf y}-X \bm{\theta})$.
Corollary \ref{cor4} gives
\begin{align}
\pi(\bm{\theta} \mid \phi, \bm{\lambda}) = N(\bm{\theta};{\bf 0},\tau \phi I) \prod_{i=1}^{p} \mbox{I} \left( e^{\sqrt{2} -\frac{\tau \phi}{\theta_i^2}} > \lambda_i \right) \frac{1}{h(\lambda_i)}
\label{eq:pemom_latent}
\end{align}
and $\pi(\lambda_i) = h(\lambda_i)= P\left( e^{\sqrt{2} -\frac{\tau \phi}{\theta_i^2}} > \lambda_i \mid \phi \right)$.
Again $h(\lambda_i)$ has no simple form but is not required by Algorithm 2,
which gives the $k^{th}$ Gibbs iteration
\begin{enumerate}
\item $\phi^{(k)} \sim e^{- \sum_{i=1}^{p} \frac{\tau \phi}{\theta_i^2}}  \mbox{IG}\left(\phi; \frac{a^*}{2},\frac{b^*}{2}\right)$
\begin{enumerate}
\item Propose $\phi^* \sim \mbox{IG}\left(\phi; \frac{a^*}{2},\frac{b^*}{2}\right)$
\item Set $\phi^{(k)}=\phi^*$ with probability $\mbox{min} \left\{ 1, e^{(\phi^{(k-1)}-\phi^*) \tau \sum_{i=1}^{p} (\theta_i^{(k-1)})^{-2}}  \right\}$, else $\phi^{(k)}=\phi^{(k-1)}$.
\end{enumerate}

\item $\bm{\lambda}^{(k)} \sim \prod_{i=1}^{p} 
\mbox{Unif} \left(\lambda_i; 0, e^{\sqrt{2} -\tau \phi/(\theta_i^{(k-1)})^2} \right)$

\item $\bm{\theta}^{(k)} \sim  
N(\bm{\theta}; {\bf m}, \phi^{(k)} S^{-1}) \prod_{i=1}^{p} \mbox{I} \left( \theta_i^2 > \left| \frac{\phi \tau}{\mbox{log}(\lambda_i^{(k)})-\sqrt{2}}  \right| \right)$.
\end{enumerate}

\section{Examples}
\label{sec:examples}

We assess our posterior sampling algorithms (Section \ref{sec:postsampling})
and the use of NLPs for high-dimensional estimation.
Section \ref{sec:postdraw_onemodel} shows a simple yet illustrative multi-modal example.
Section \ref{ssec:highdim_ex} studies $p \geq n$ cases
and compares the BMA estimators induced by NLPs
with benchmark priors (BP, \citeasnoun{fernandez:2001}), 
hyper-g priors (HG, \citeasnoun{liang:2008}), SCAD \cite{fan:2001} and LASSO \cite{tibshirani:1996}.
For NLPs we used functions modelSelection and rnlp in R package mombf 1.5.9,
using the default prior dispersions $\tau=0.358,0.133,0.119$ 
for pMOM, piMOM and peMOM priors (respectively),
which assign 0.01 prior probability to $| \theta_i/\sqrt{\phi} |< 0.2$ \cite{johnson:2010},
and $\phi \sim \mbox{IG}(0.01/2,0.01/2)$.
We set a Beta-Binomial(1,1) prior on the model space truncated so that $P(M_k)=0$ whenever $\mbox{dim}(\Theta_k)>n$,
and adapted the Gibbs model space search in \citeasnoun{johnson:2012} to never visit those models.
For benchmark and hyper-g priors we used function bms in R package BMS 0.3.3 with default parameters,
again with the Beta-Binomial(1,1) prior.
For LASSO and SCAD we set the penalization parameter with 10-fold cross-validation
using functions mylars and ncvreg in R packages parcor 0.2.6 and ncvreg 3.2.0 (respectively) with default parameters.
All the R code is provided as supplementary material.
We assess the relative merits attained by each method without the help
of any procedures to pre-screen covariates.

\subsection{Posterior samples for a given model}
\label{sec:postdraw_onemodel}

\begin{figure}
\begin{center}
\begin{tabular}{cc}
\includegraphics[width=0.4\textwidth]{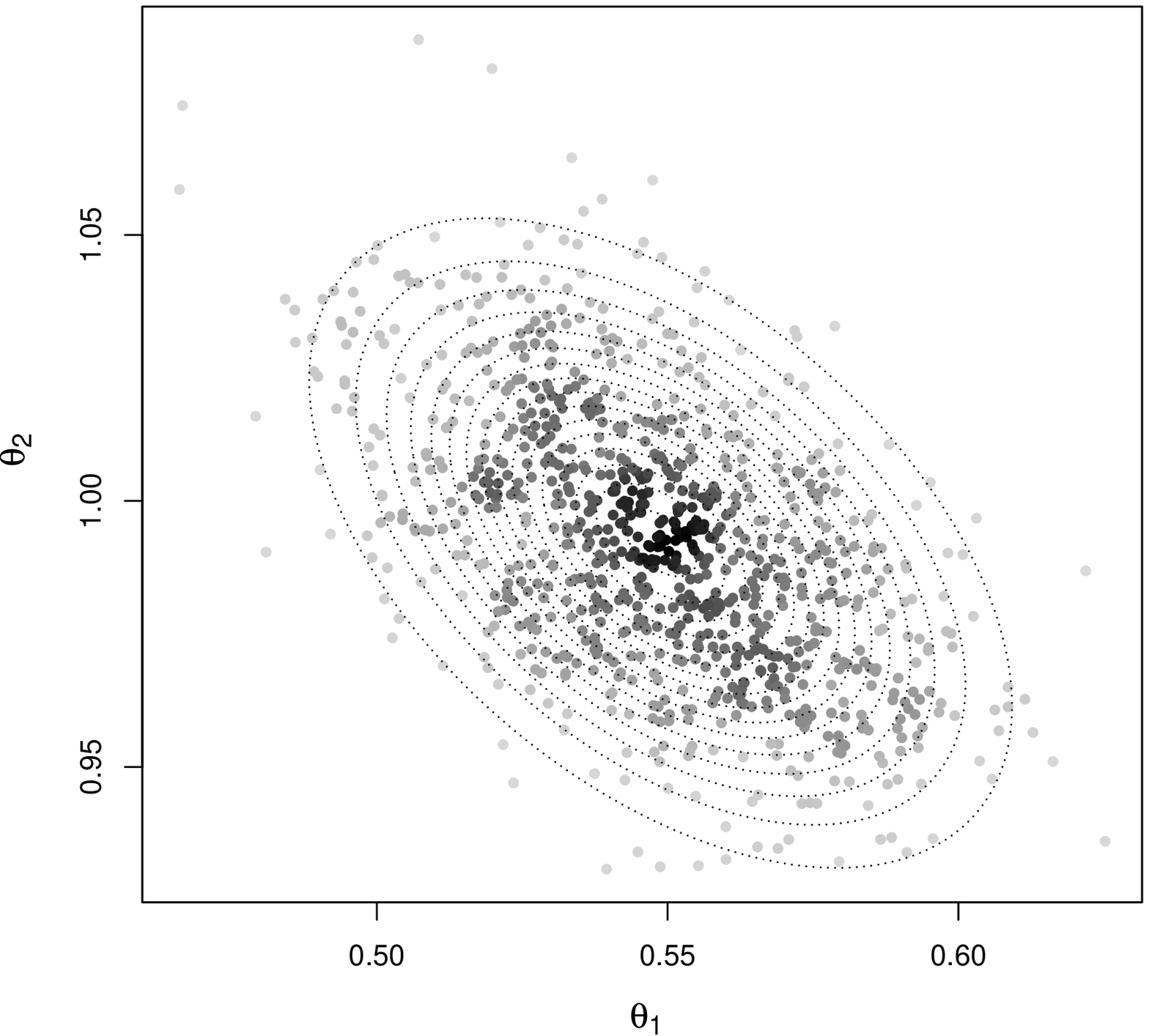} &
\includegraphics[width=0.4\textwidth]{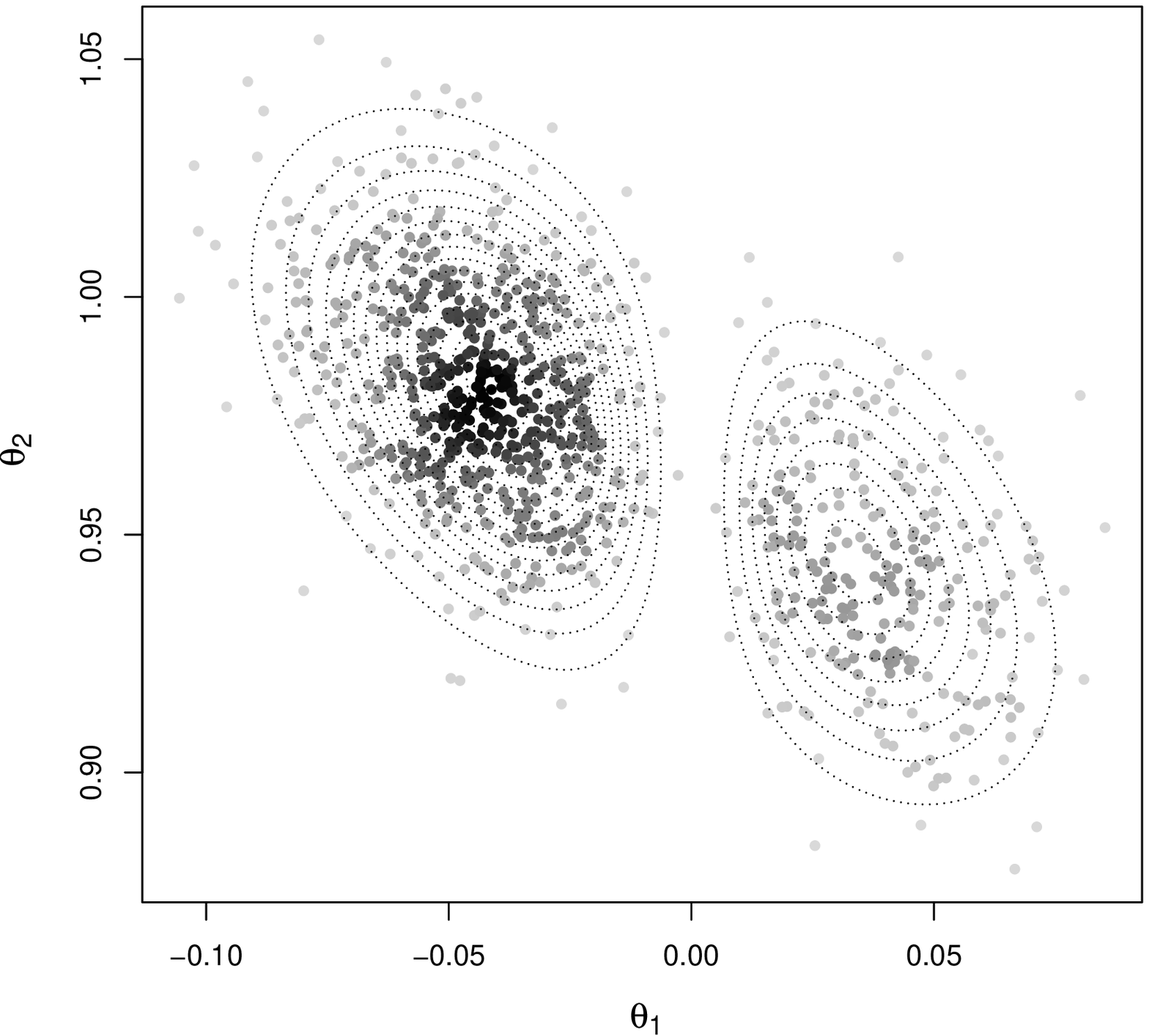} \\
\includegraphics[width=0.4\textwidth]{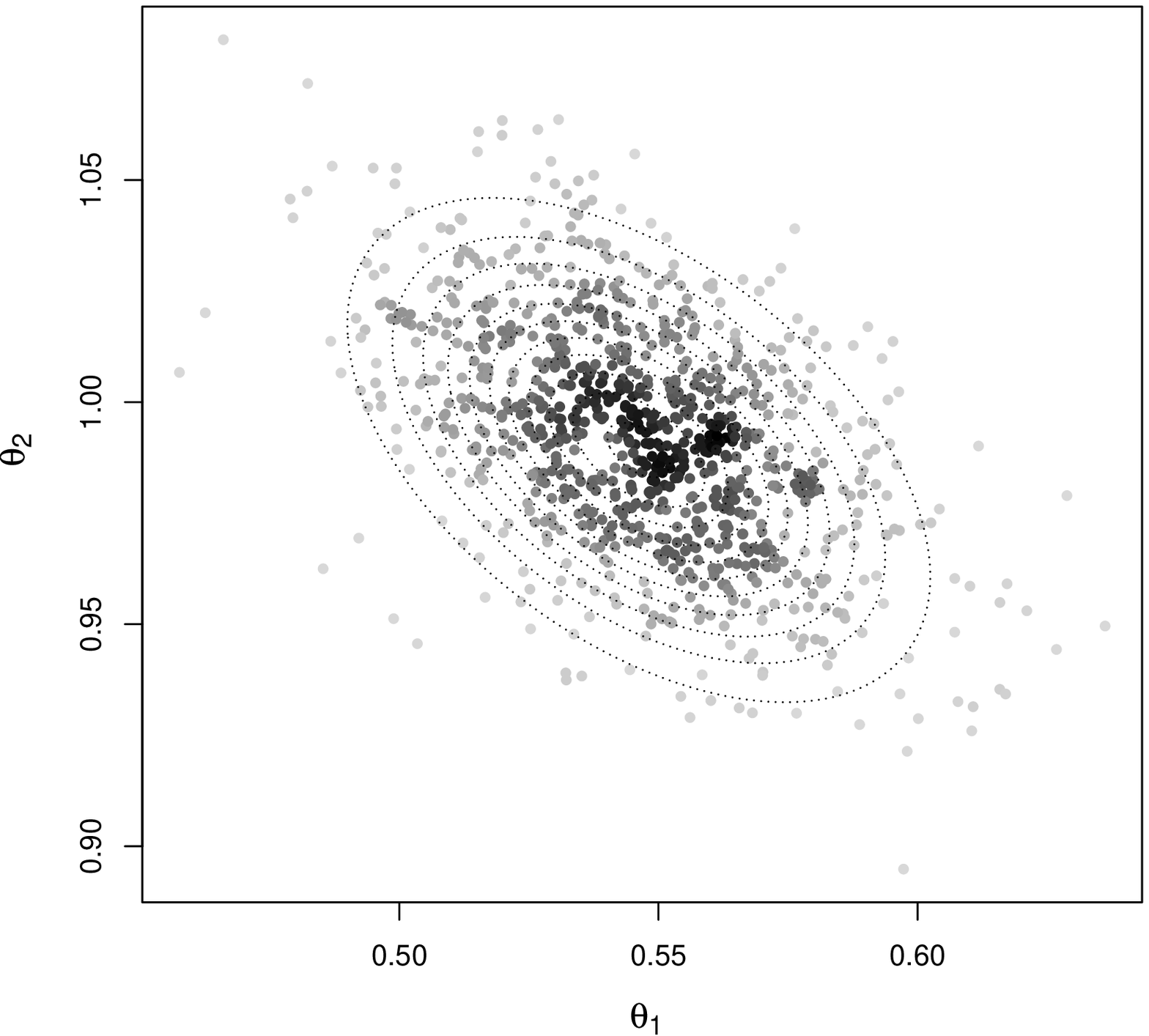} &
\includegraphics[width=0.4\textwidth]{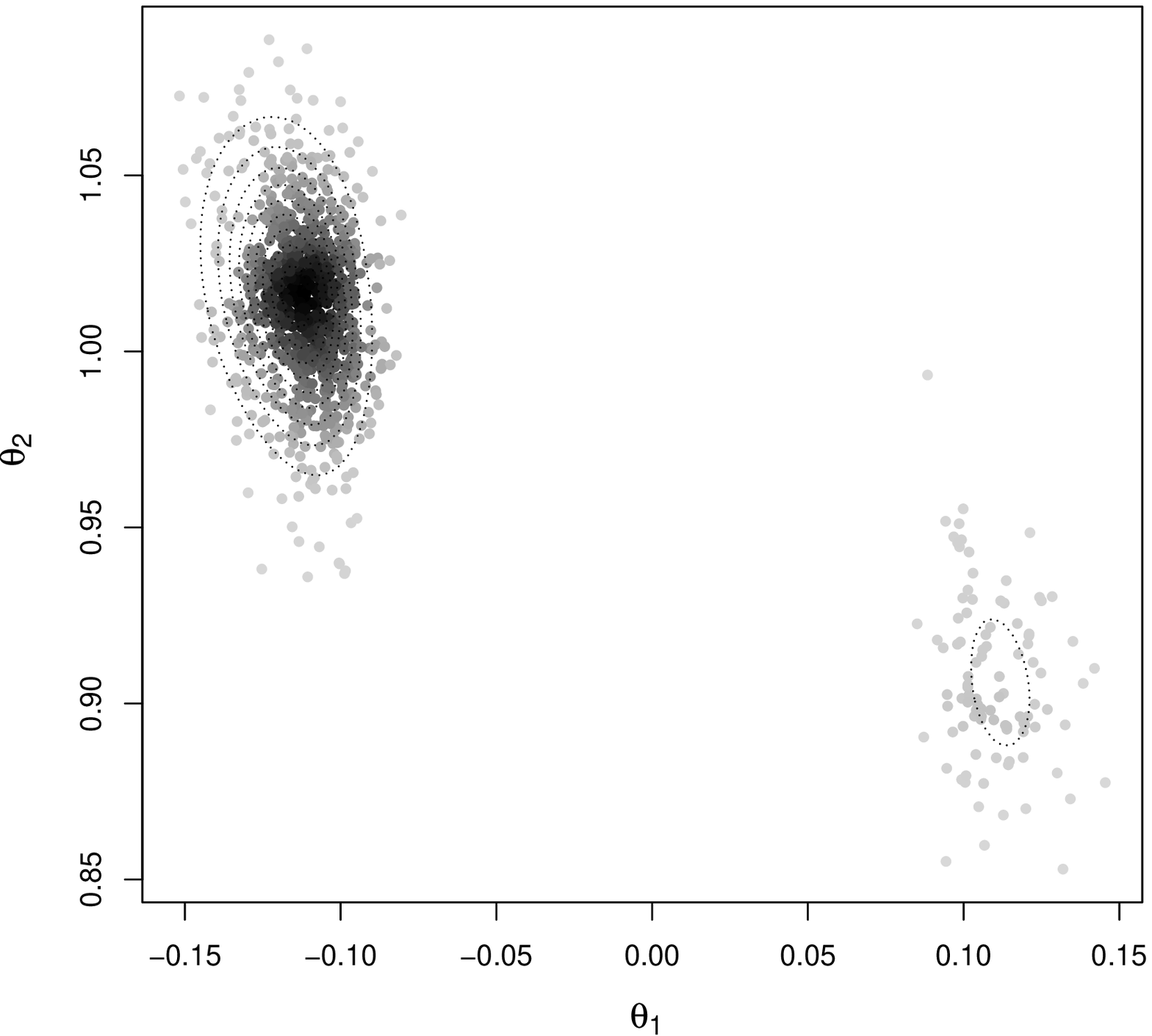} \\
\includegraphics[width=0.4\textwidth]{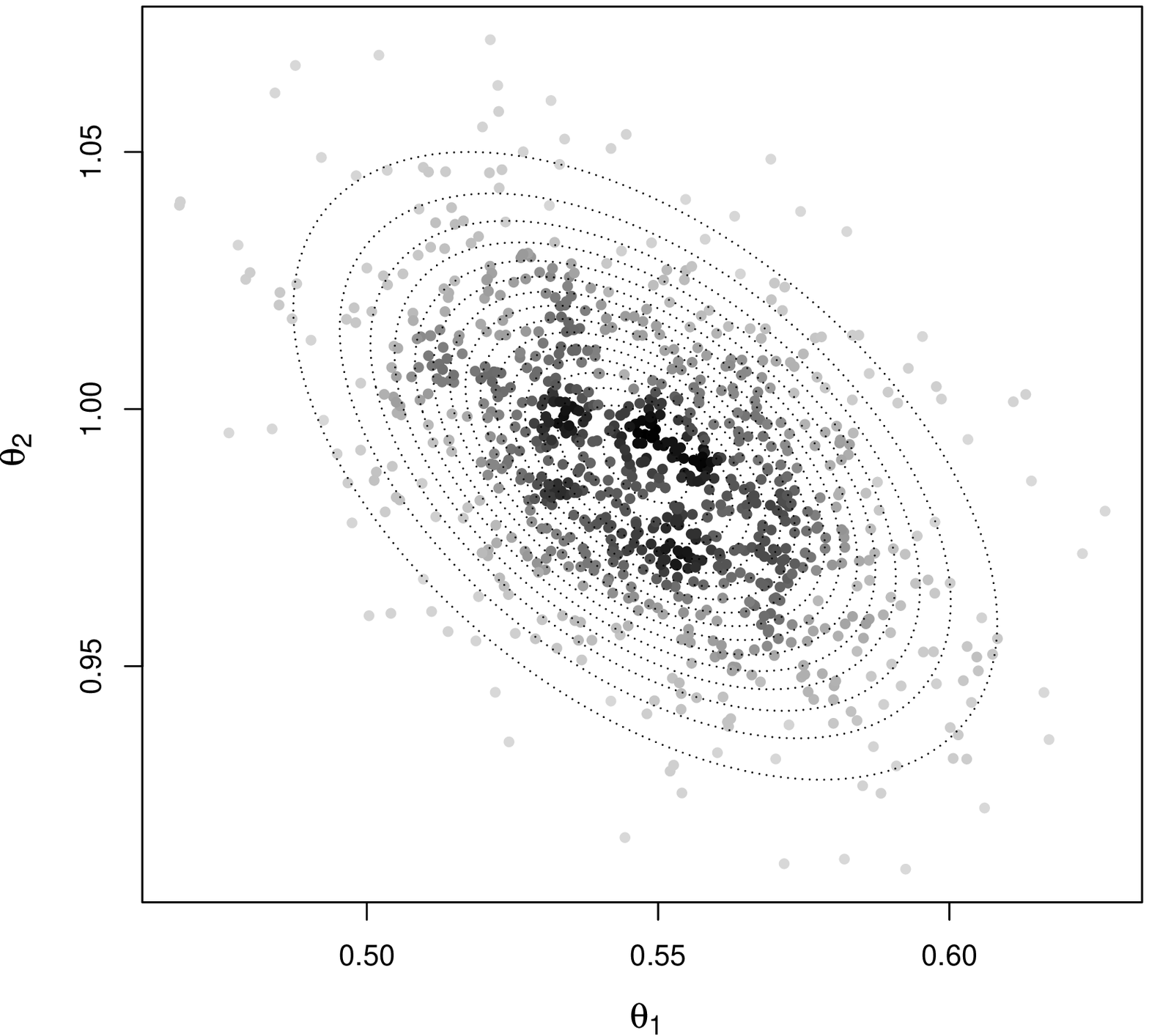} &
\includegraphics[width=0.4\textwidth]{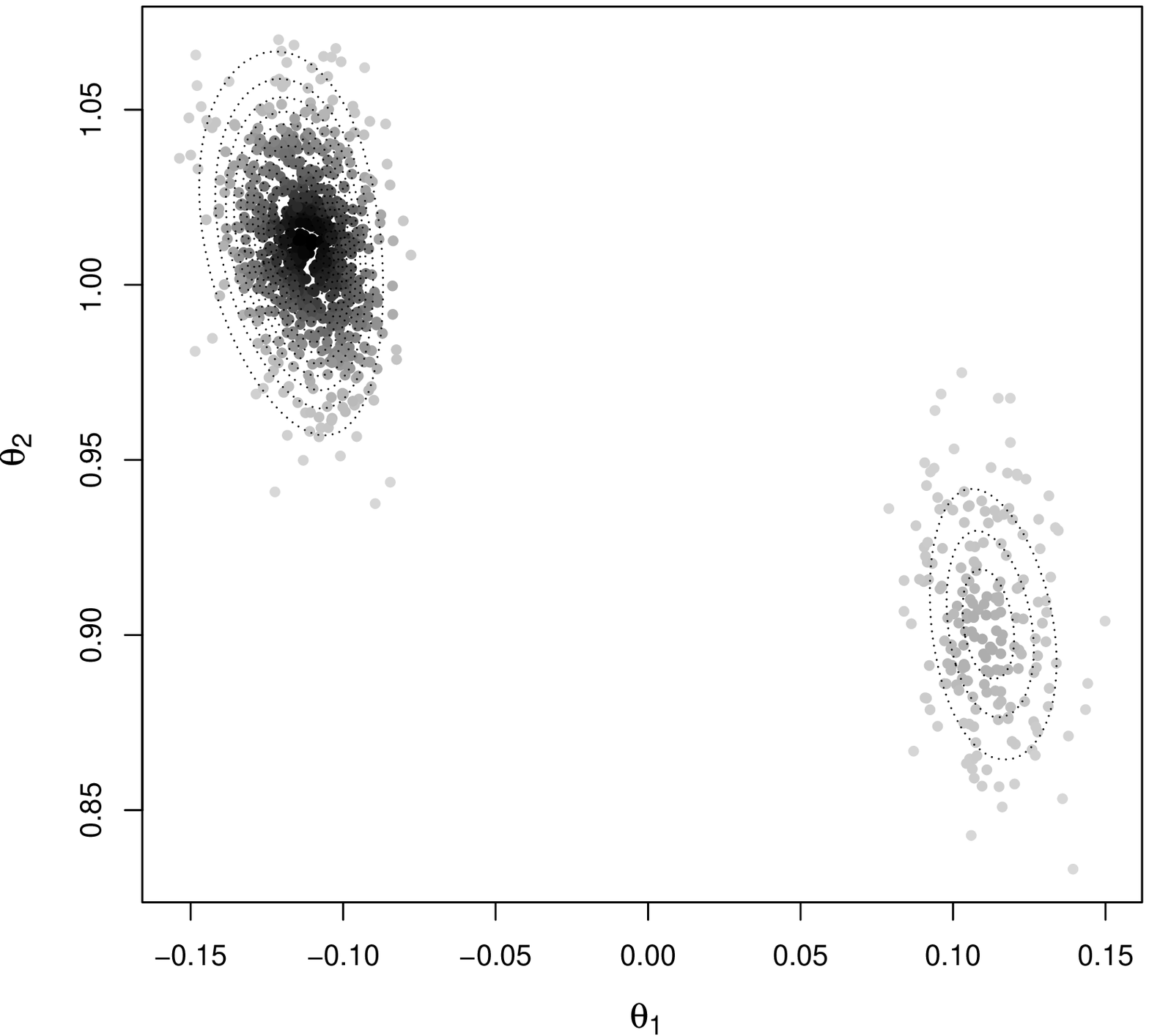} \\
\end{tabular}
\end{center}
\caption[900 Gibbs draws when $\bm{\theta}=(0.5,1)'$
and $\bm{\theta}=(0,1)'$]{
900 Gibbs draws when $\bm{\theta}=(0.5,1)'$ (left)
and $\bm{\theta}=(0,1)'$ (right) and posterior density contours.
Top: MOM ($\tau=0.358$); Middle: iMOM ($\tau=0.133$); Bottom: eMOM ($\tau=0.119$)}
\label{fig:sim_mom}
\end{figure}

\begin{table}
\begin{center}
\begin{tabular}{|c|c|c|c|}\hline
\multicolumn{4}{|c|}{$\theta_1=0.5$, $\theta_2=1$} \\ \hline
  & MOM & iMOM & eMOM \\
$\theta_1=0$, $\theta_2=0$      & 0         & 0        & 0  \\ 
$\theta_1=0$, $\theta_2\neq0$   & 2.8e-78   & 2.72e-78 & 6.86e-79 \\ 
$\theta_1 \neq 0$, $\theta_2=0$ & 1.95e-191 & 3.82-e191& 5.90e-191 \\ 
$\theta_1 \neq0$, $\theta_2=0$  & 1         & 1        & 1  \\ \hline \hline
\multicolumn{4}{|c|}{$\theta_1=0$, $\theta_2=1$} \\ \hline
$\theta_1=0$, $\theta_2=0$      & 1.69e-225 & 4.39e-225 & 1.08e-224 \\ 
$\theta_1=0$, $\theta_2\neq0$   & 0.999     & 1         & 1 \\ 
$\theta_1 \neq 0$, $\theta_2=0$ & 1.82e-193 & 1.64e-192 & 6.80e-192 \\ 
$\theta_1 \neq0$, $\theta_2=0$  & 8.83e-05     & 3.30e-09  & 3.17e-09 \\ \hline
\end{tabular}
\end{center}
\caption[Posterior model probabilities with 2 predictors]{Posterior model probabilities with 2 predictors
($\theta_1=0.5$ or $0$, $\theta_2=\phi=1$, $n=1000$)}
\label{tab:pp}
\end{table}

\begin{table}
\begin{center}
\begin{tabular}{|c|c|c|c|} \hline
\multicolumn{4}{|c|}{$\theta_1=0.5$, $\theta_2=1$} \\ \hline
& MOM & iMOM & eMOM \\
$\theta_1$ & 0.096 & 0.110 & 0.018 \\
$\theta_2$ & 0.034 & 0.134 & 0.019 \\
$\phi$     &-0.016 & 0.069 & 0.027 \\ \hline \hline
\multicolumn{4}{|c|}{$\theta_1=0$, $\theta_2=1$} \\ \hline
$\theta_1$ & 0.115 & 0.032 & 0.049 \\
$\theta_2$ & 0.134 & 0.122 & 0.042 \\
$\phi$     &-0.040 & 0.327 & 0.353 \\ \hline
\end{tabular}
\end{center}
\caption[Serial correlation with 2 predictors]{Serial correlation with 2 predictors ($\theta_1=0.5$ or $0$, $\theta_2=\phi=1$, $n=1000$)}
\label{tab:rho_nlp}
\end{table}

We simulate 1,000 realizations from $y_i \sim N(\theta_1 x_{1i} + \theta_2 x_{2i}, 1)$,
where $(x_{1i},x_{2i})$ are drawn from a bivariate Normal with $E(x_{1i})=E(x_{2i})=0$,
$V(x_{1i})=V(x_{2i})=2$, $\mbox{Cov}(x_{1i},x_{2i})=1$.
We first consider $\theta_1=0.5$, $\theta_2=1$,
and compute posterior probabilities for the four possible models.
%{\it i.e.} the null model, including only $x_1$ or $x_2$ and the full model.
We assign equal a priori probabilities and
obtain exact integrated likelihoods 
using functions \texttt{pmomMarginalU}, \texttt{pimomMarginalU} and \texttt{pemomMarginalU} in the \texttt{mombf} package
(the former is available in closed-form, for the latter two we used $10^6$ importance samples).
The posterior probability assigned to the full model 
under all three priors is 1 (up to rounding) (Table \ref{tab:pp}).
Figure \ref{fig:sim_mom} (left) shows 900 Gibbs draws (100 burn-in) obtained under the full model.
The posterior mass is well-shifted away from 0 and resembles an elliptical shape for the three priors.
Table \ref{tab:rho_nlp} gives the first-order auto-correlations, which are very small.
This example reflects the advantages of the orthogonalization strategy,
which is particularly efficient as the latent truncation becomes negligible.

We now set $\theta_1=0$, $\theta_2=1$
and keep $n=1000$ and $(x_{1i},x_{2i})$ as before.
We simulated several data sets and in most cases
did not observe a noticeable posterior multi-modality.
We portray a specific simulation that did exhibit multi-modality,
as this poses a greater challenge from a sampling perspective.
Table \ref{tab:pp} shows that
the data-generating model adequately concentrated the posterior mass.
Although the full model was clearly dismissed in light of the data,
as an exercise we drew from its posterior.
Figure \ref{fig:sim_mom} (right) shows 900 Gibbs draws after a 100 burn-in,
and Table \ref{tab:rho_nlp} indicates the auto-correlation.
The sampled values adequately captured the multiple modes.

\subsection{High-dimensional estimation}
\label{ssec:highdim_ex}

\begin{figure}
\begin{center}
\begin{tabular}{cc}
$\rho=0$ & $\rho=0.25$ \\
\includegraphics[width=0.45\textwidth,height=0.35\textwidth]{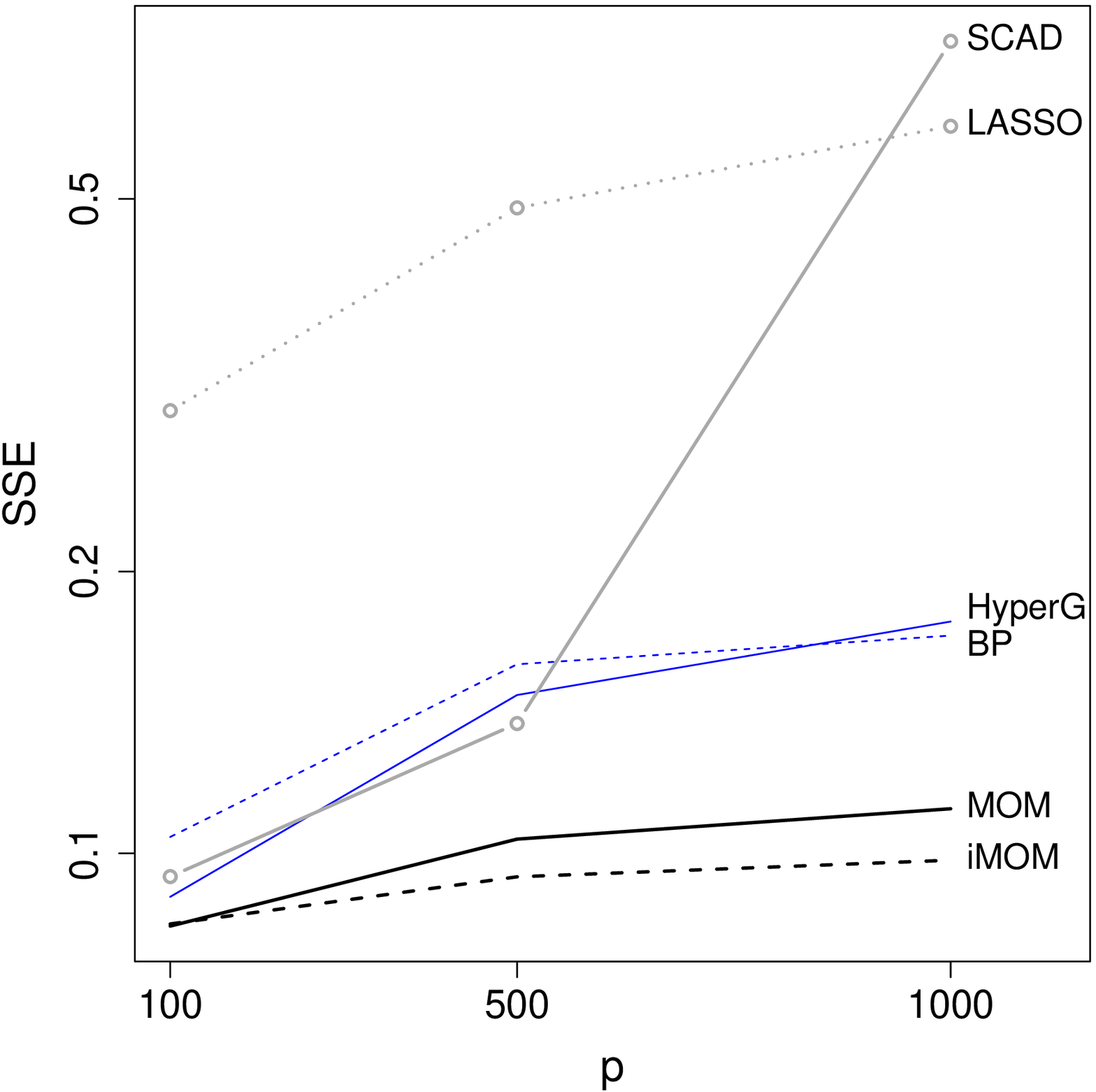} &
\includegraphics[width=0.45\textwidth,height=0.35\textwidth]{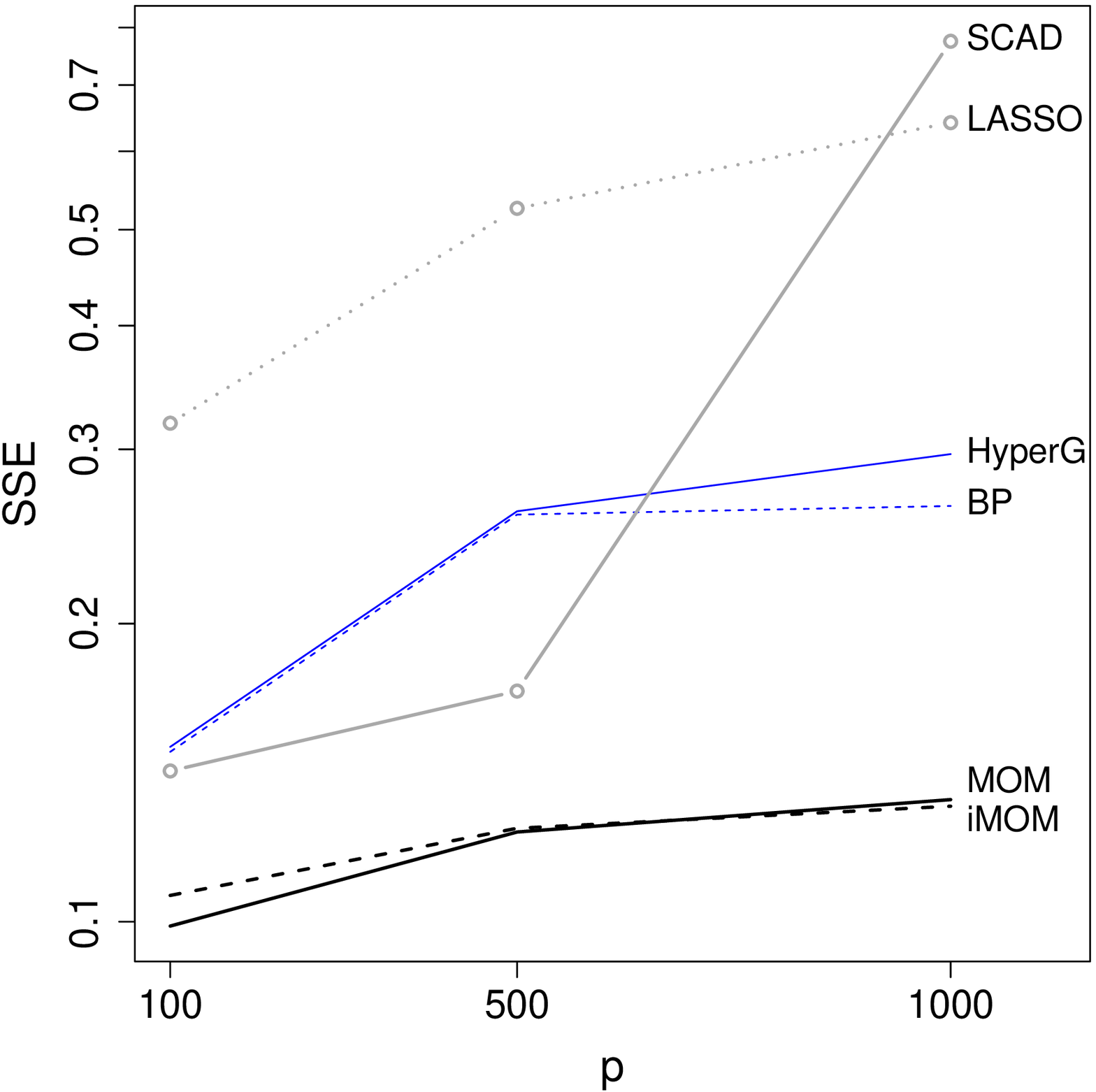} \\
\includegraphics[width=0.45\textwidth,height=0.35\textwidth]{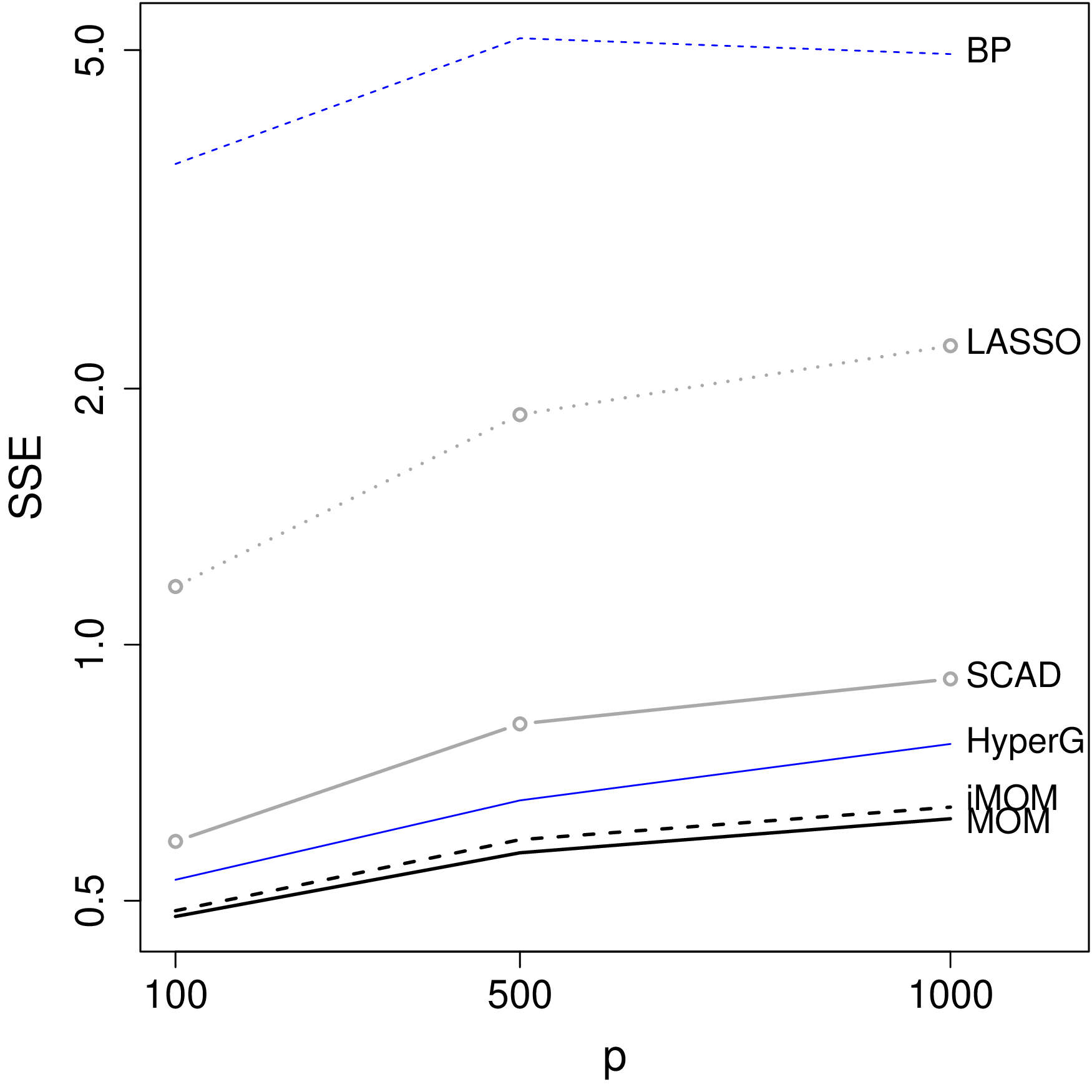} &
\includegraphics[width=0.45\textwidth,height=0.35\textwidth]{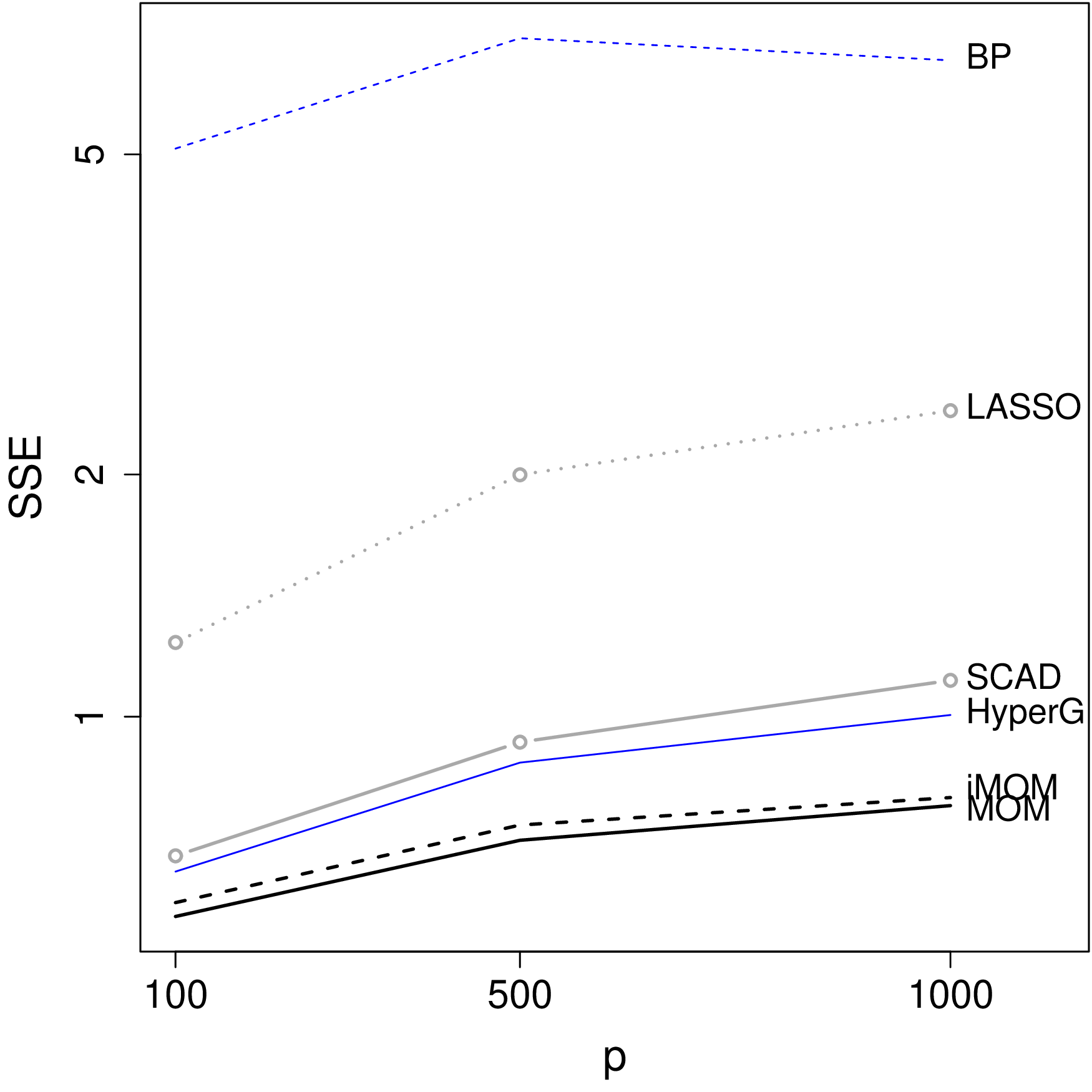} \\
\includegraphics[width=0.45\textwidth,height=0.35\textwidth]{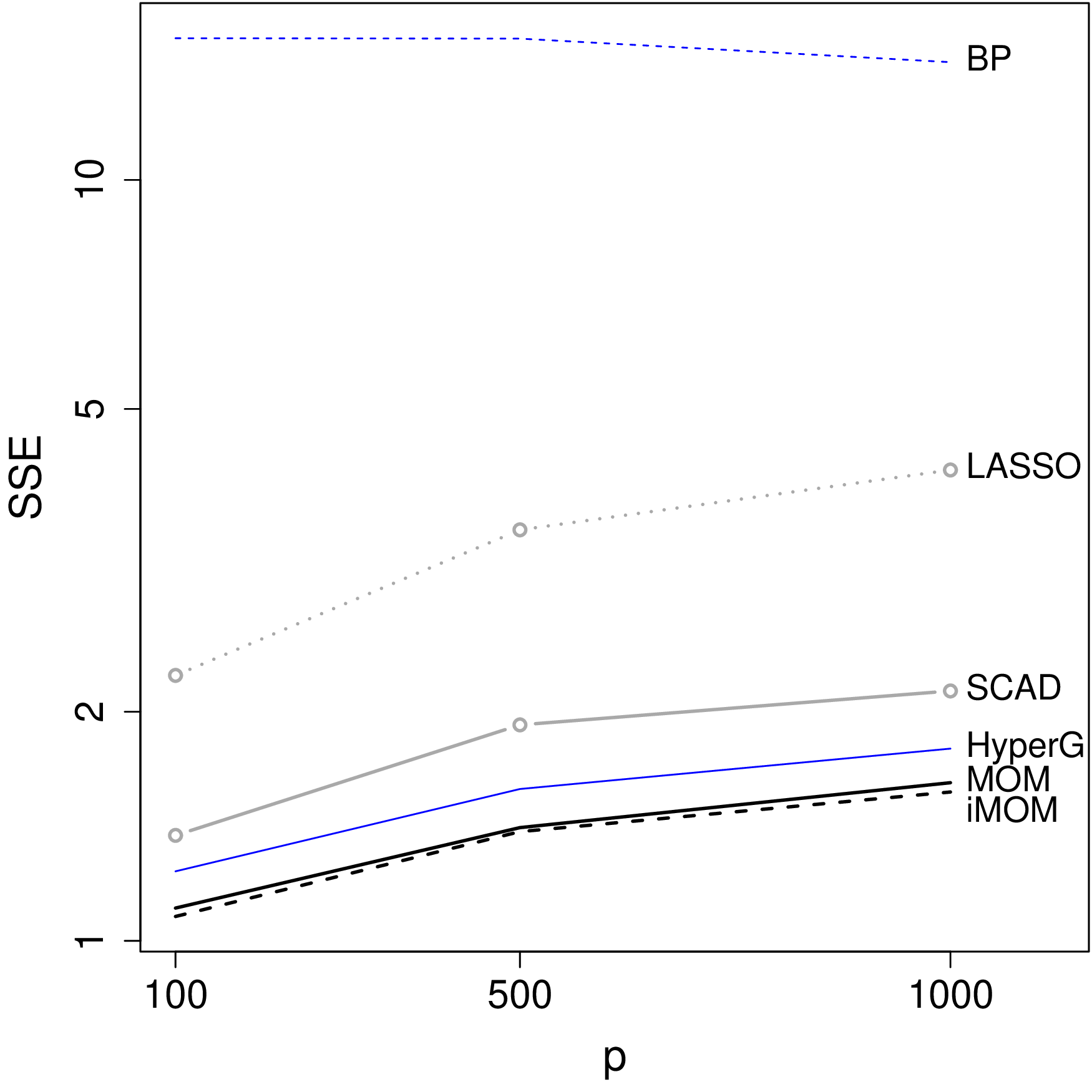} &
\includegraphics[width=0.45\textwidth,height=0.35\textwidth]{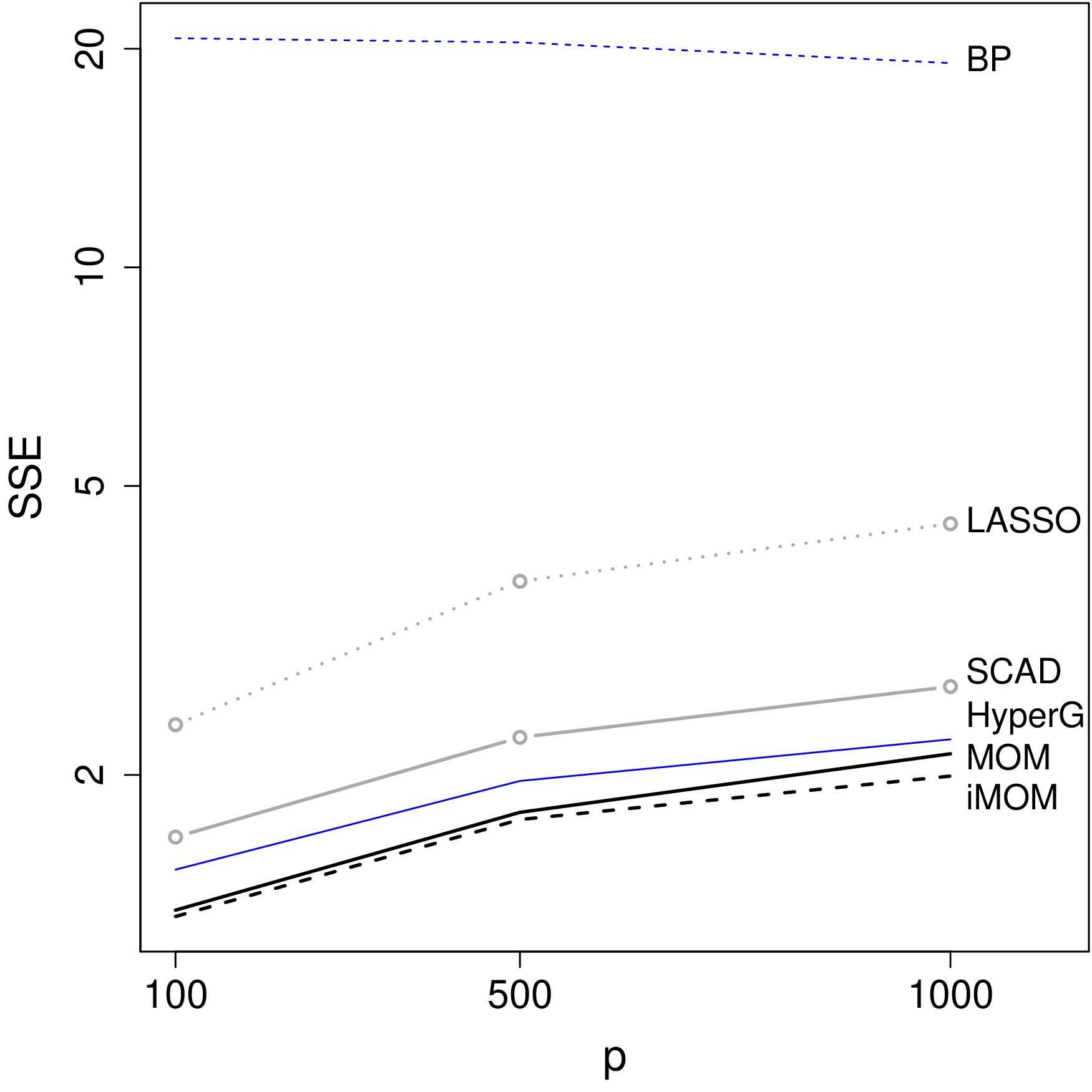} \\
\end{tabular}
\end{center}
\caption[Mean SSE$=\sum_{i=1}^{p} (\hat{\theta}_i - \theta_i)^2$]
{Mean SSE$=\sum_{i=1}^{p} (\hat{\theta}_i - \theta_i)^2$ when $\phi=1,4,8$ (top, middle, bottom), $\rho=0,0.25$ (left, right).
Simulation settings: $n=100$, $p=100,500,1000$ and 5 non-zero coefficients $0.6,1.2,1.8,2.4,3.0$.}
\label{fig:mse_pgrow}
\end{figure}

\begin{figure}
\begin{center}
\begin{tabular}{cc}
\multicolumn{2}{c}{$\rho=0$, $\phi=1$} \\
\includegraphics[width=0.45\textwidth,height=0.32\textwidth]{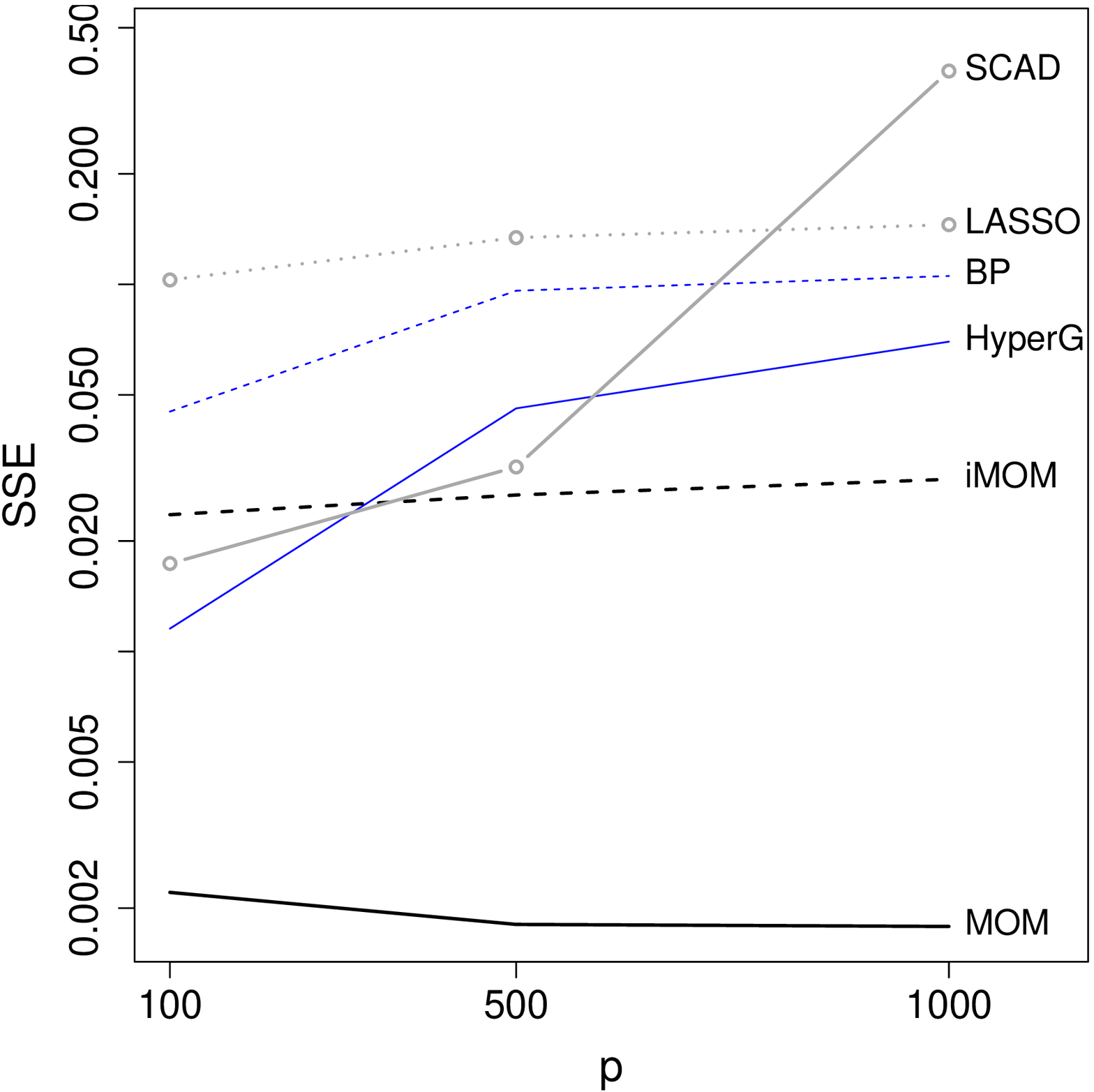} &
\includegraphics[width=0.45\textwidth,height=0.32\textwidth]{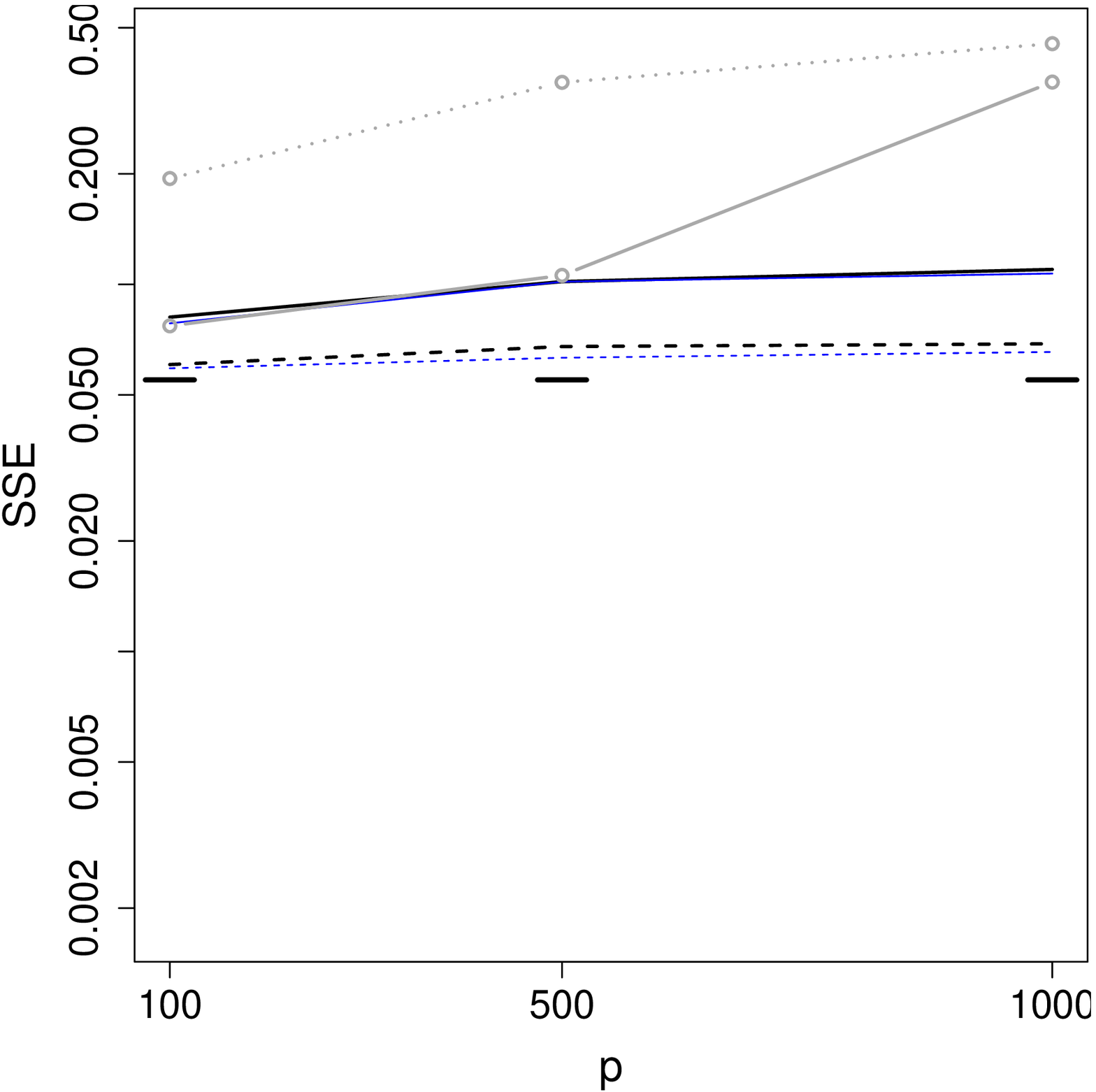} \\
\multicolumn{2}{c}{$\rho=0$, $\phi=4$} \\
\includegraphics[width=0.45\textwidth,height=0.32\textwidth]{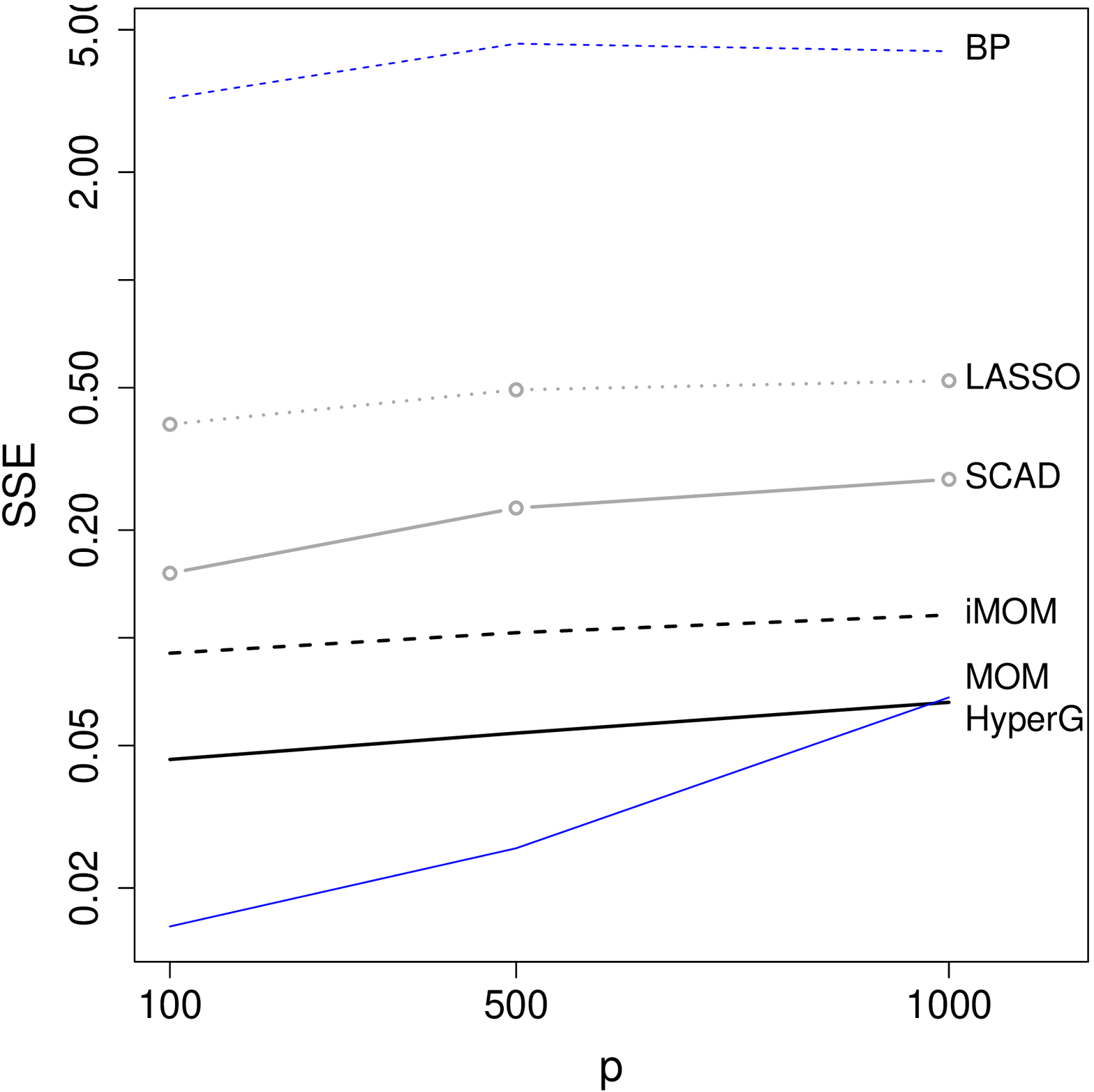} &
\includegraphics[width=0.45\textwidth,height=0.32\textwidth]{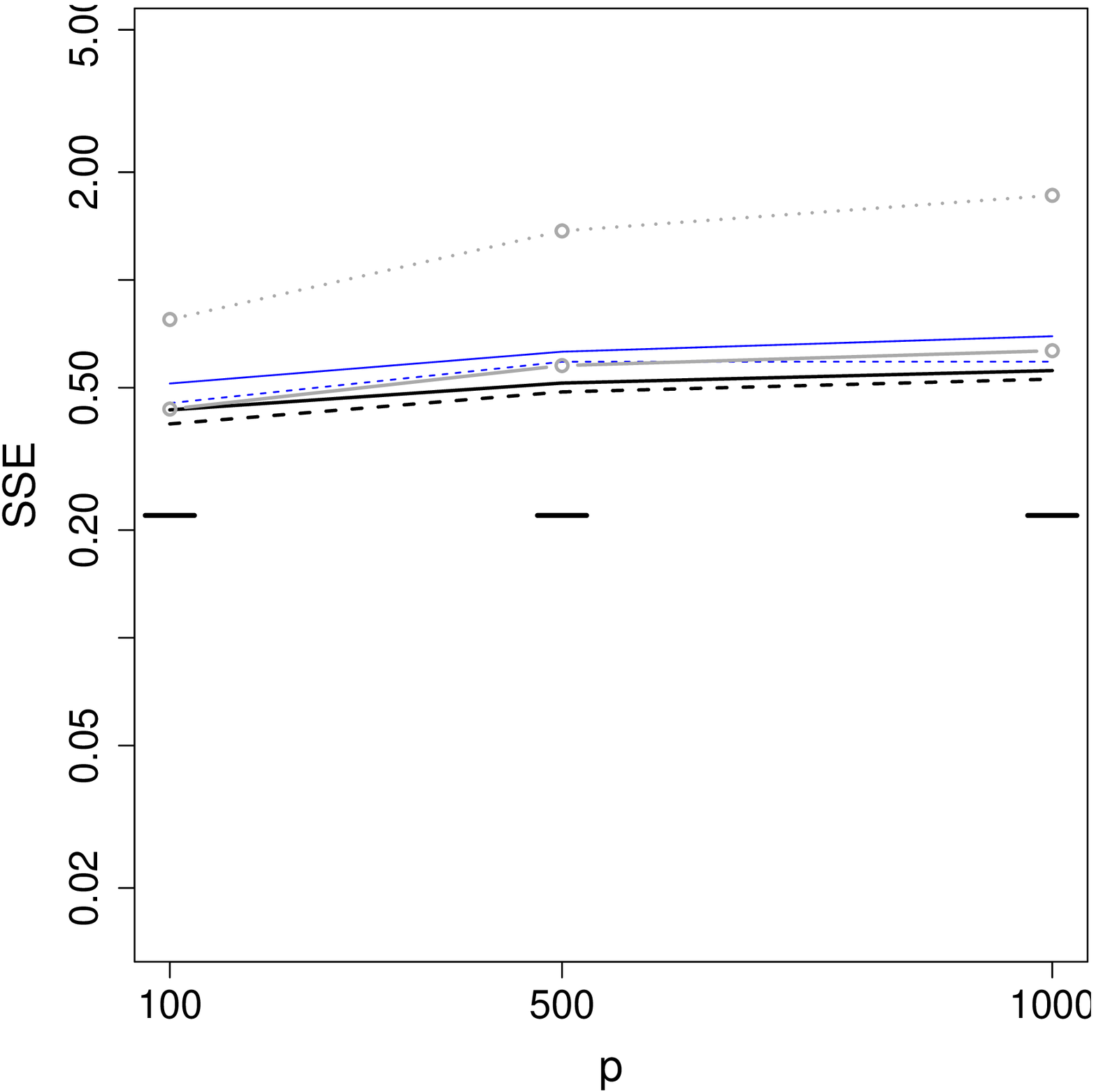} \\
\multicolumn{2}{c}{$\rho=0$, $\phi=8$} \\
\includegraphics[width=0.45\textwidth,height=0.32\textwidth]{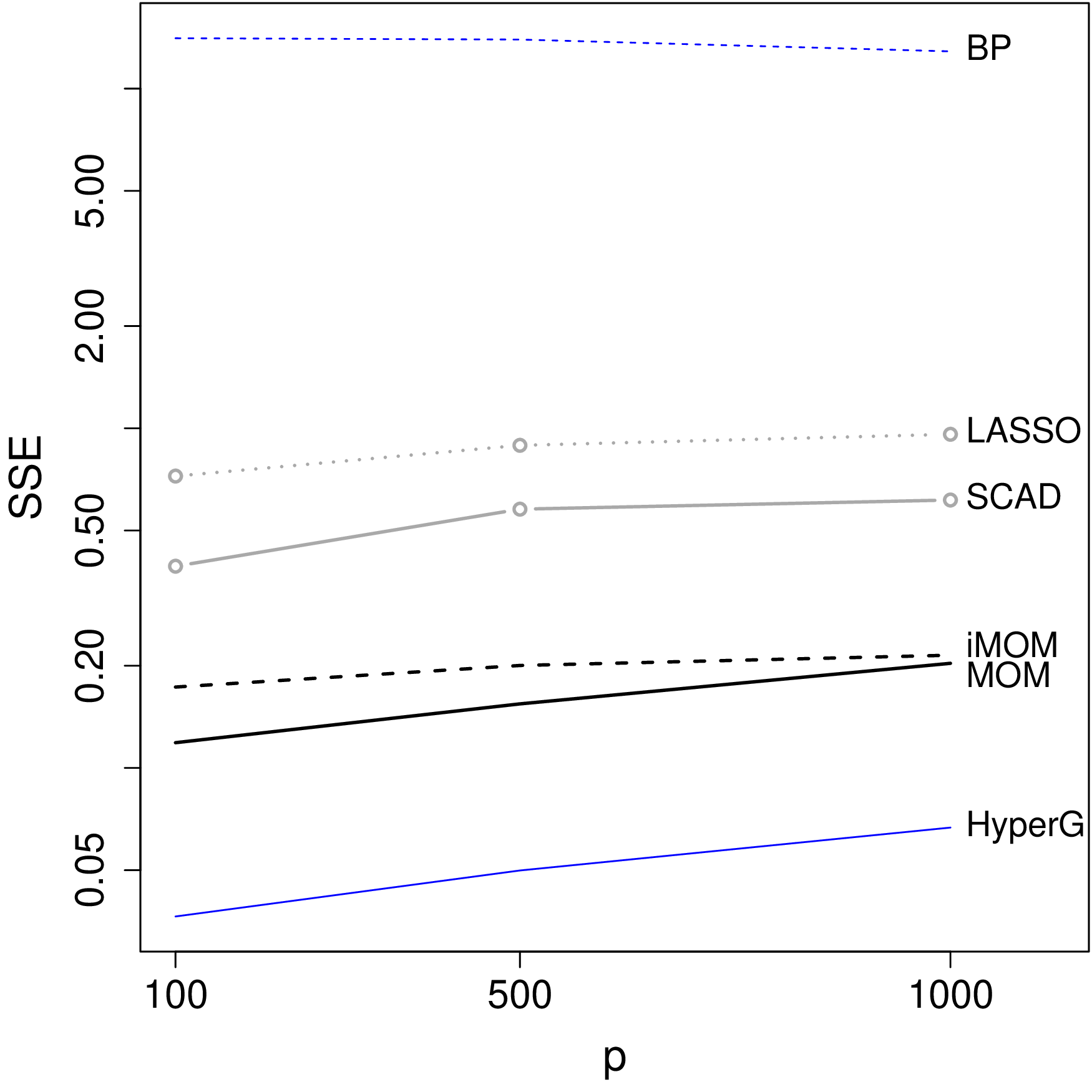} &
\includegraphics[width=0.45\textwidth,height=0.32\textwidth]{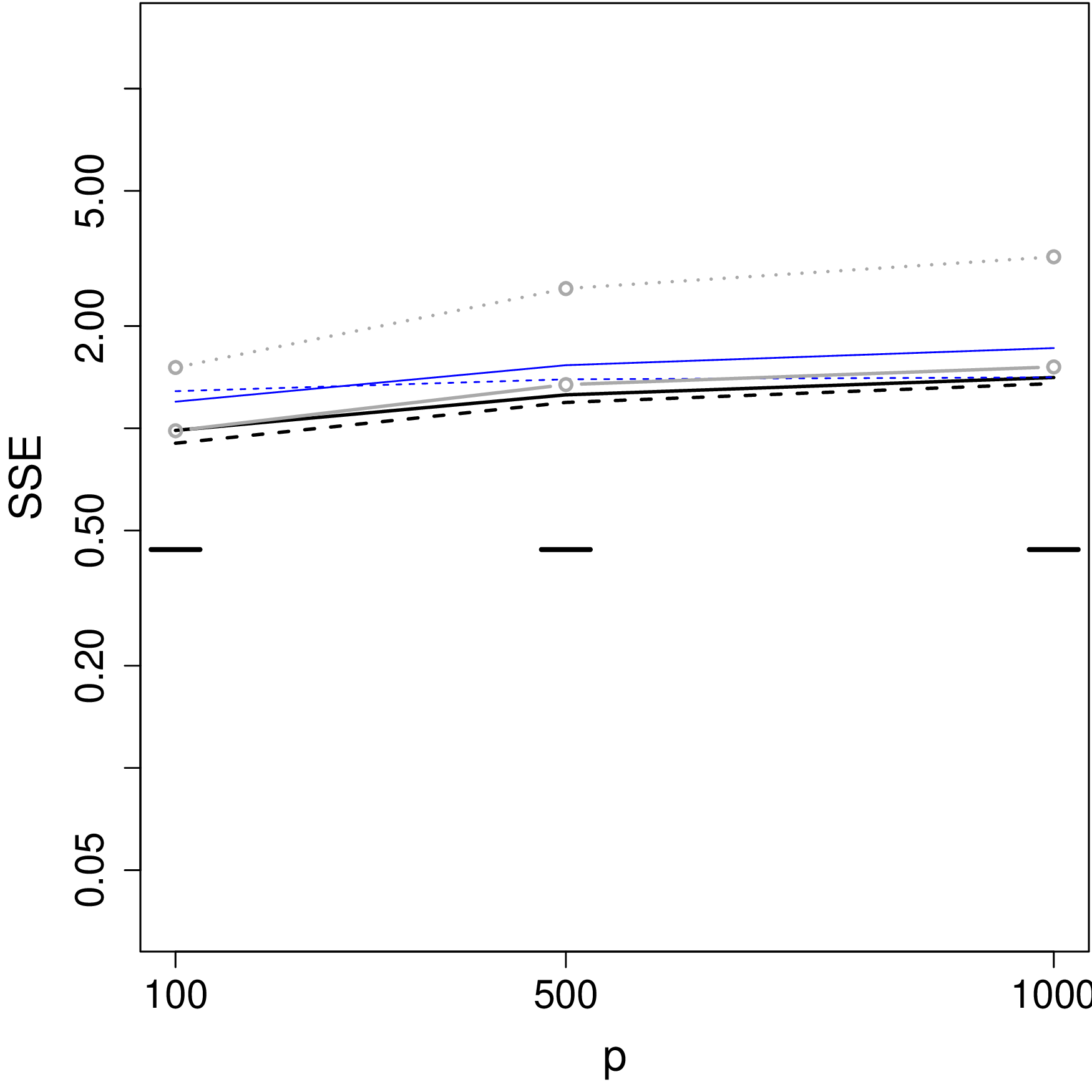} \\
\end{tabular}
\end{center}
\caption[{Mean SSE for $\theta_i=0$ and $\theta_i \neq 0$}]
{Mean SSE for $\theta_i=0$ (left) and $\theta_i \neq 0$ (right) when $\phi=1,4,8$.
Simulation settings: $\rho=0$, $n=100$, $p=100,500,1000$ and 5 non-zero coefficients $0.6,1.2,1.8,2.4,3.0$.}
\label{fig:mse_pgrow_cond_rho0}
\end{figure}

\begin{figure}
\begin{center}
\begin{tabular}{cc}
\multicolumn{2}{c}{$\rho=0.25$, $\phi=1$} \\
\includegraphics[width=0.45\textwidth,height=0.32\textwidth]{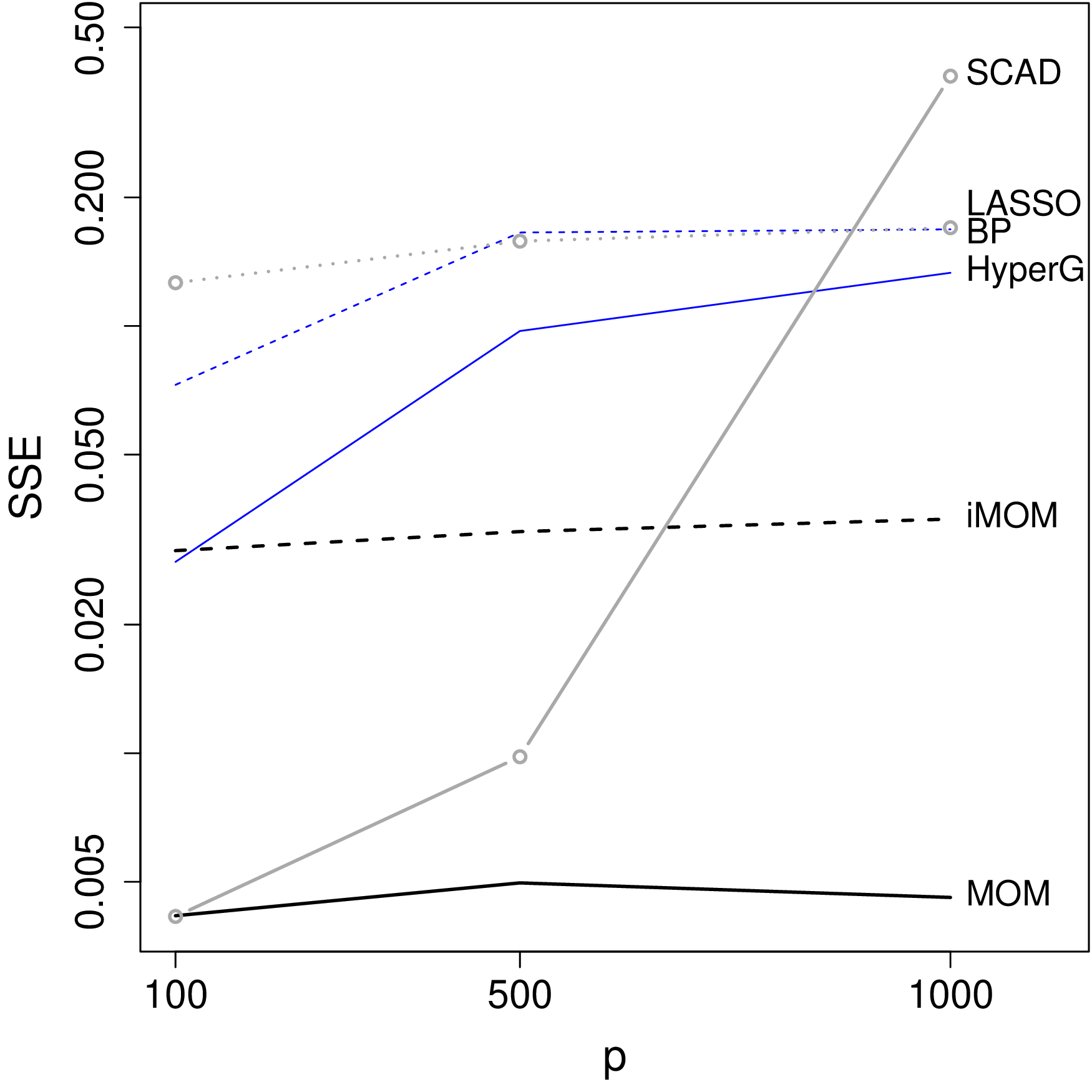} &
\includegraphics[width=0.45\textwidth,height=0.32\textwidth]{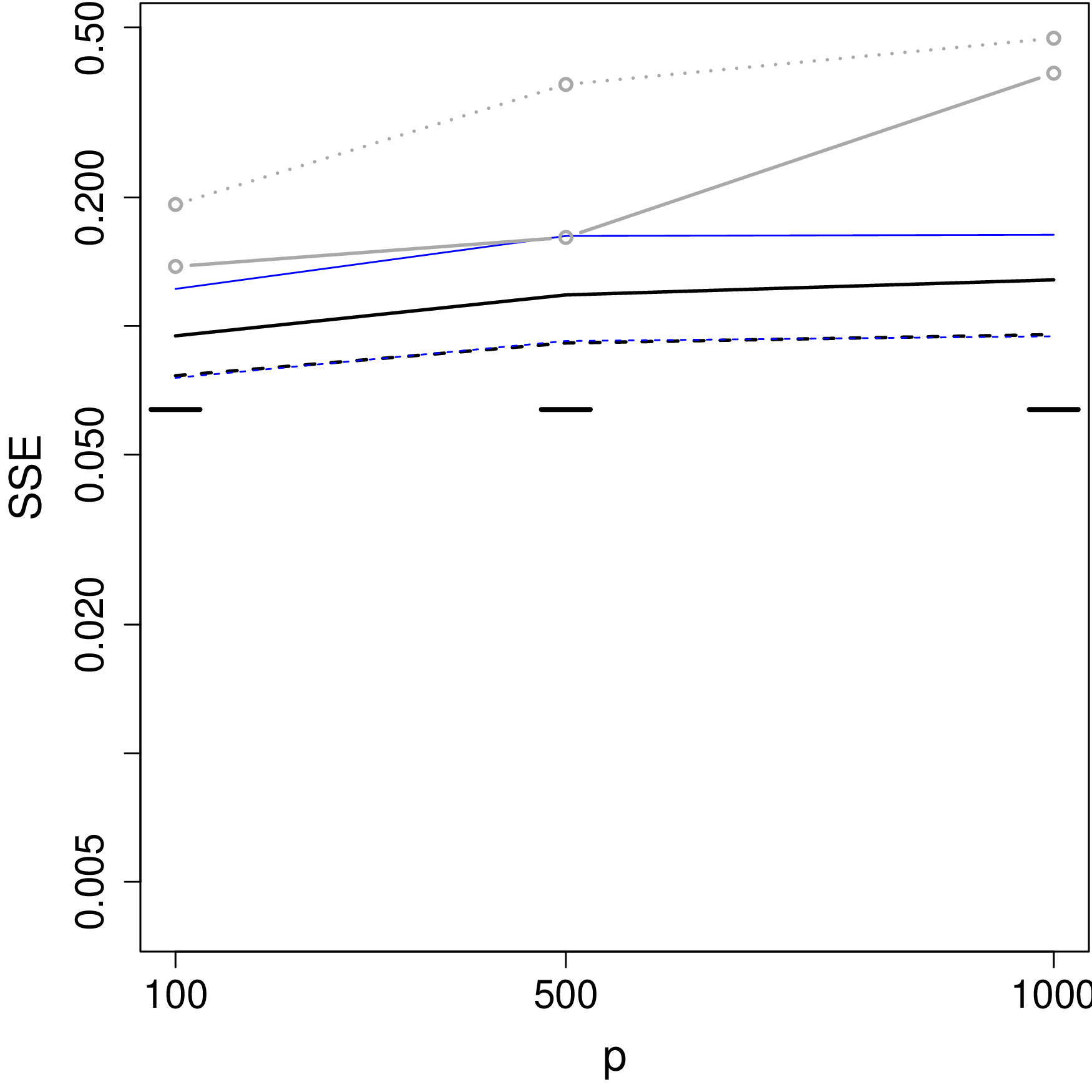} \\
\multicolumn{2}{c}{$\rho=0.25$, $\phi=4$} \\
\includegraphics[width=0.45\textwidth,height=0.32\textwidth]{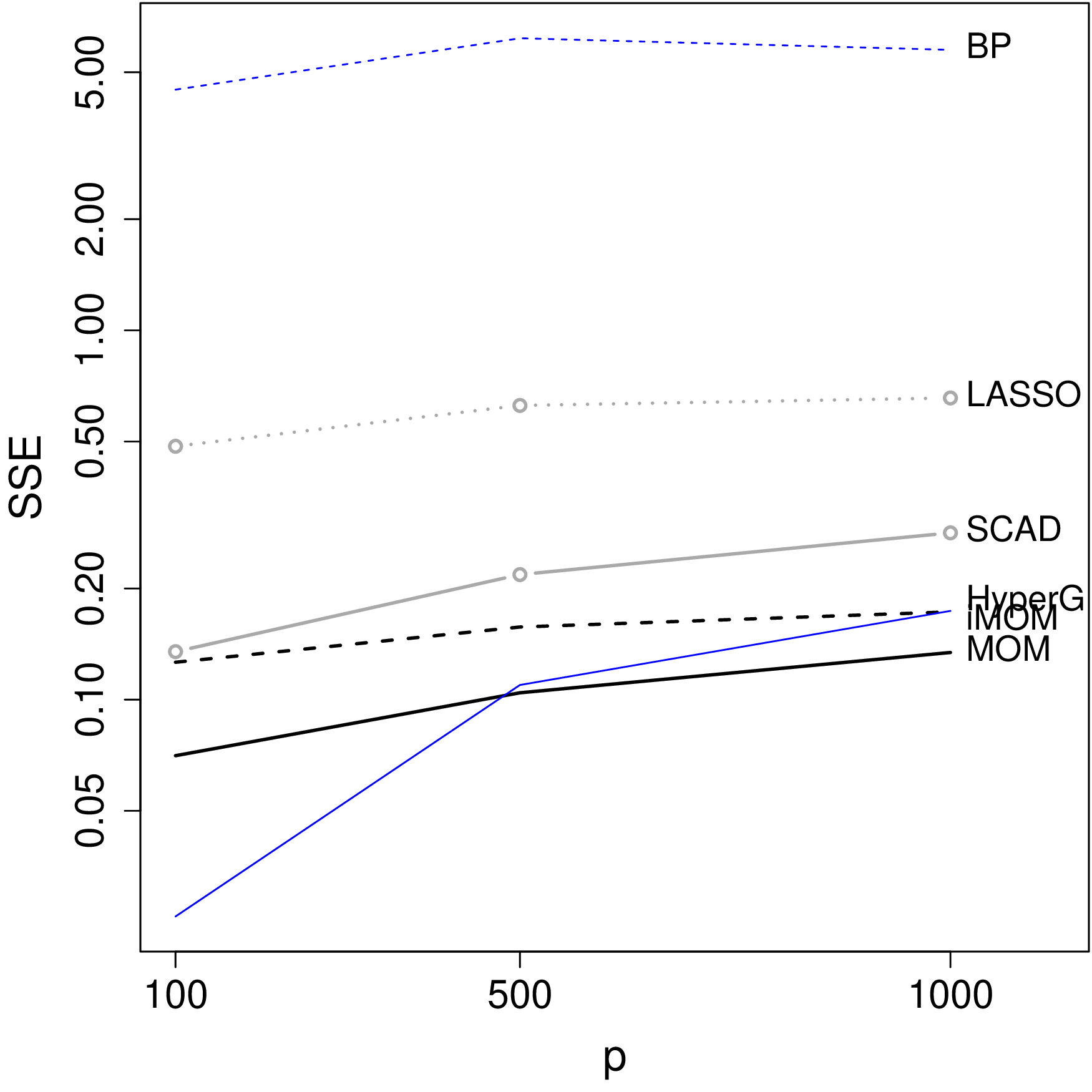} &
\includegraphics[width=0.45\textwidth,height=0.32\textwidth]{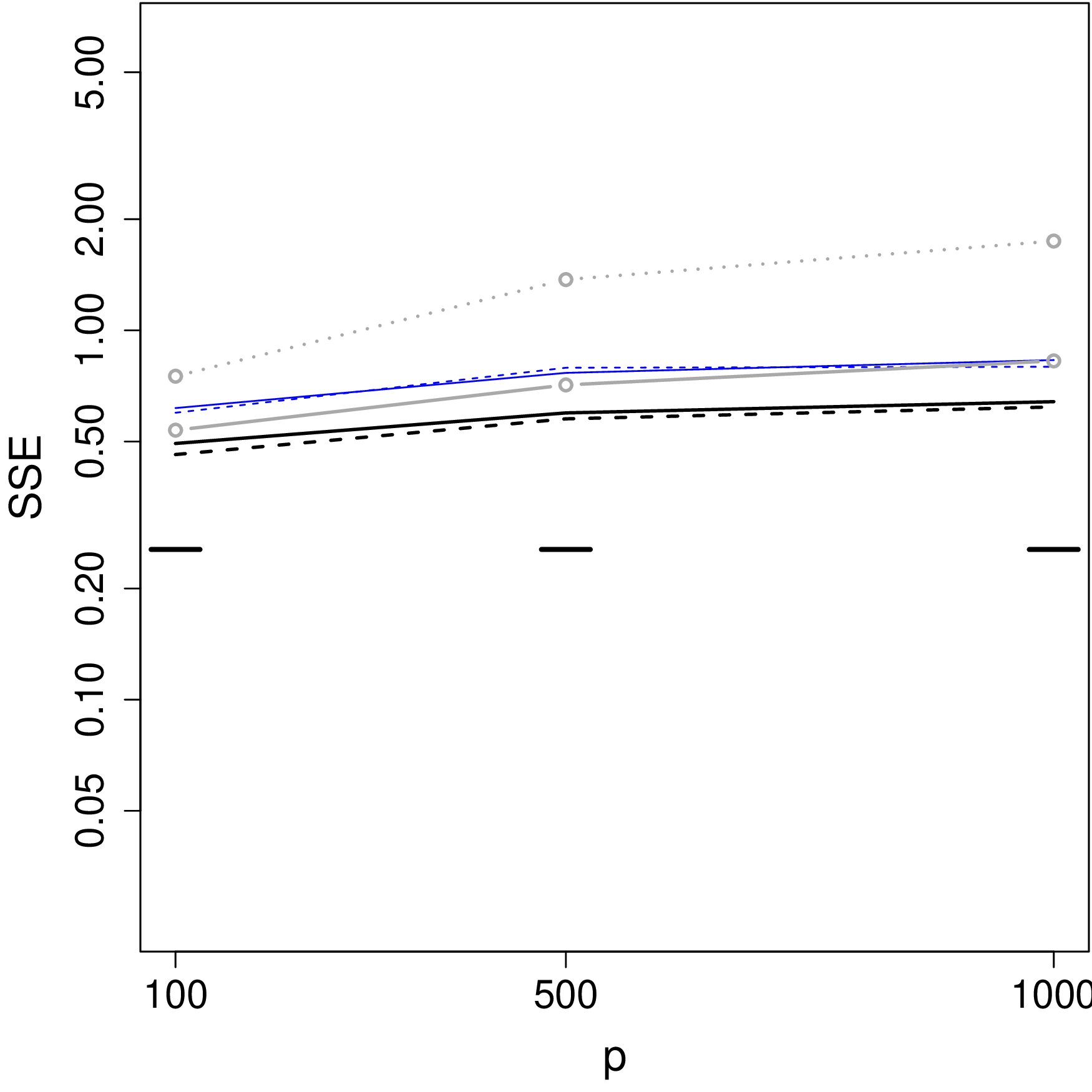} \\
\multicolumn{2}{c}{$\rho=0.25$, $\phi=8$} \\
\includegraphics[width=0.45\textwidth,height=0.32\textwidth]{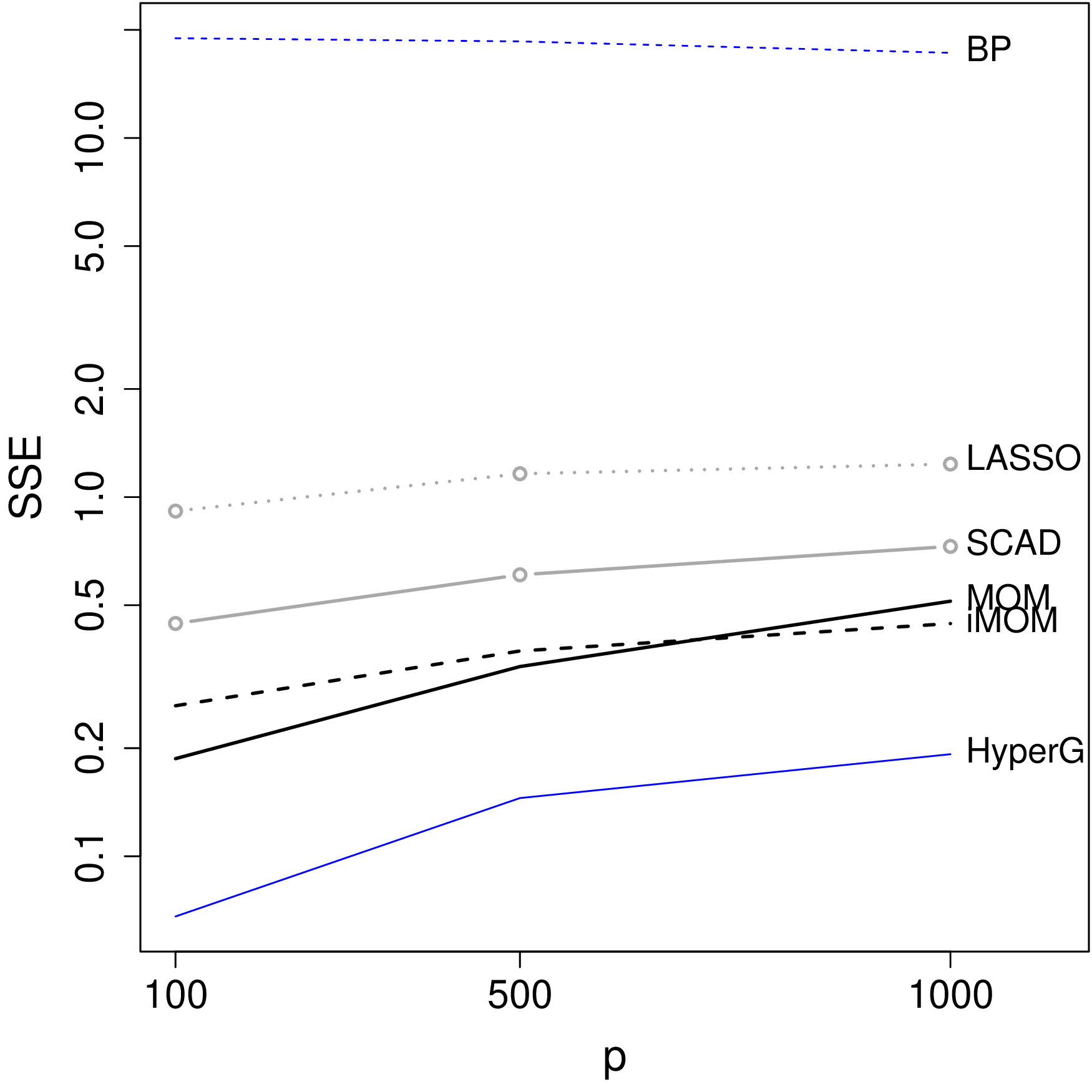} &
\includegraphics[width=0.45\textwidth,height=0.32\textwidth]{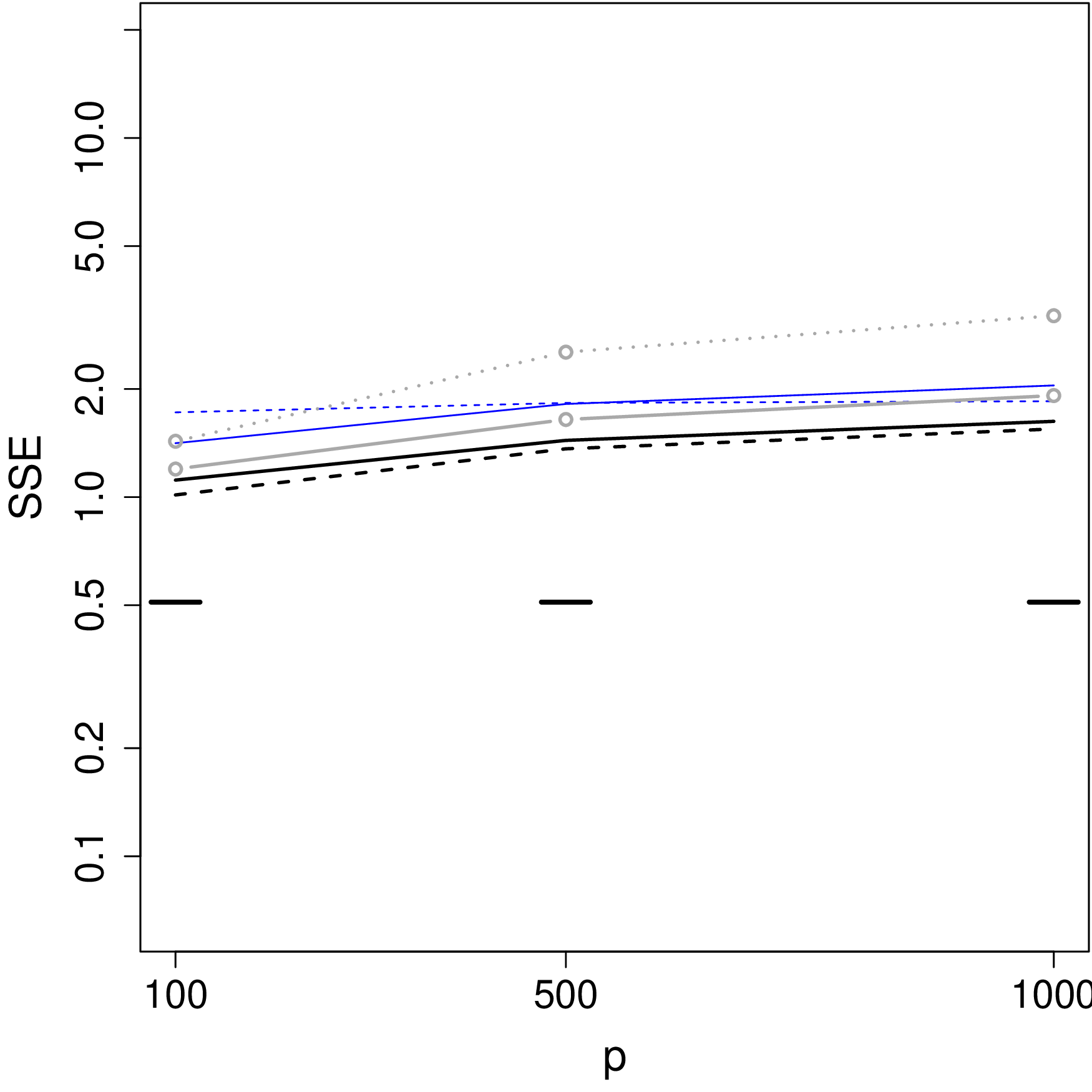} \\
\end{tabular}
\end{center}
\caption[{Mean SSE for $\theta_i=0$ and $\theta_i \neq 0$}]
{Mean SSE for $\theta_i=0$ (left) and $\theta_i \neq 0$ (right) when $\phi=1,4,8$.
Simulation settings: $\rho=0.25$, $n=100$, $p=100,500,1000$ and 5 non-zero coefficients $0.6,1.2,1.8,2.4,3.0$.}
\label{fig:mse_pgrow_cond_rho25}
\end{figure}

We perform a simulation study with $n=100$ and growing dimensionality $p=100,500,1000$.
We set $\theta_i=0$ for $i=1,\ldots,p-5$,
the remaining 5 coefficients to $(0.6,1.2,1.8,2.4,3)$
and consider residual variances $\phi=1,4,8$.
Covariates were sampled from ${\bf x} \sim N({\bf 0}, \Sigma)$,
where $\Sigma$ has unit variances and all pairwise correlations set to $\rho=0$ or $\rho=0.25$.
We remark that $\rho$ are population correlations,
the maximum absolute sample correlations when $\rho=0$
being $0.37,0.44,0.47$ for $p=100,500,1000$ (respectively),
%(95\% quantile=0.20)
and $0.54$, $0.60$, $0.62$ when $\rho=0.25$.
%(95\% quantile=0.39)
We simulated 1,000 data sets under each setup.

Parameter estimates $\hat{\bm{\theta}}$ were obtained via BMA.
Let $\delta$ be the model indicator.
For each simulated data set, we performed 1,000 full Gibbs iterations 
(100 burn-in) %\cite{johnson:2012},
which are equivalent to 1,000$\times p$ birth-death moves.
These provided posterior samples $\delta_1,\ldots,\delta_{1000}$ from $P(\delta \mid {\bf y})$.
We estimated model probabilities from the proportion of MCMC visits
and obtained %as many 
posterior draws for $\bm{\theta}$ using the algorithms in Section \ref{ssec:gibbslm}.
For benchmark (BP) and hyper-g (HG) priors we used the same strategy for $\delta$
and drew $\bm{\theta}$ from their corresponding posteriors.

Figure \ref{fig:mse_pgrow} shows sum of squared errors (SSE)
$\sum_{i=1}^{p} (\hat{\theta}_i - \theta_i)^2$ averaged across simulations
for $\phi=1,4,8$, $\rho=0,0.25$.
pMOM and piMOM perform similarly and
present an SSE between 1.15 and 10 times lower than other methods in all scenarios.
As $p$ grows, differences between methods tend to be larger.
To obtain more insight on how the lower SSE is achieved,
Figures \ref{fig:mse_pgrow_cond_rho0}-\ref{fig:mse_pgrow_cond_rho25} show SSE
separately for $\theta_i=0$ (left) and $\theta_i \neq 0$ (right).
The largest differences between methods were observed for $\theta_i=0$,
where SSE remains very stable as $p$ grows for pMOM and piMOM.
For $\theta_i \neq 0$ differences in SSE are smaller, iMOM slightly outperforming MOM.
Here SSE tends to be largest for LASSO.
For all methods as signal-to-noise ratios $|\theta_i|/\sqrt{\phi_i}$ decrease
the SSE worsens relative to the least squares oracle estimator
(Figures \ref{fig:mse_pgrow_cond_rho0}-\ref{fig:mse_pgrow_cond_rho25}, right panels, black horizontal segments).

\subsection{Gene expression data}
\label{ssec:geneexpr}

\begin{table}
\begin{center}
\begin{tabular}{|l|cc|ccc|} \hline
                    & \multicolumn{2}{c}{$p=172$} & \multicolumn{3}{|c|}{$p=10,172$} \\
                    & $\bar{p}$ & $R^2$ & $\bar{p}$ & $R^2$ & CPU time \\ \hline
MOM                 &      4.3  & 0.566 &     6.5   & 0.617 & 1m 52s \\
iMOM                &      5.3  & 0.560 &    10.3   & 0.620 & 59m \\
BP                  &      4.2  & 0.562 &   259.0   & 0.014 & 21h 27m \\
HG                  &     10.5  & 0.559 &   116.6   & 0.419 & 74 days \\
SCAD                &      29   & 0.565 &    81     & 0.535 & 16.7s \\ 
LASSO               &      42   & 0.586 &   159     & 0.570 & 23.7s \\ \hline 
\end{tabular}
\end{center}
\caption[Expression data with $p=172$ or $10,172$ genes]{Expression data with $p=172$ or $10,172$ genes.
$\bar{p}$: mean (MOM, iMOM, BP, HG) or selected number of predictors (SCAD, LASSO).
$R^2$ coefficient is between $(y_i,\hat{y}_i)$ (leave-one-out cross-validation).
CPU time on Linux OpenSUSE 13.1, 64 bits, 2.6GHz processor, 31.4Gb RAM
}
\label{tab:tgfb}
\end{table}

%\begin{center}
%\begin{tabular}{|l|ccc|ccc|} \hline
%         &    p=172                       &     $p=1,172$            &      $p=10,172$         \\ \hline
%         &$p^*$& $\bar{p}$ & SSR   &  $R^2$ & $p^*$& $\bar{p}$ & $R^2$ & $p^*$  & barp   & SSR     & $R^2$ \\ \hline
%LASSO    & 42  &      -    & 108.9 &  0.586 & 83  &  -        & 0.643  &  159   &  -     & 112.2   &  0.570\\ 
%SCAD     & 28  &      -    & 114.4 &  0.560 & 36  &  -        & 0.507  &   81   &  -     & 122.1   &  0.535\\ 
%MOM      & 4   &      4.3  & 113.3 &  0.566 & 3   &  4.5      & 0.578  &    6   & 6.5    & 100.1   &  0.617\\
%iMOM     & 4   &      5.3  & 115.1 &  0.560 & 5   &  6.2      & 0.571  &    6   & 10.3   &  99.4   &  0.620\\
%UIP      & 4   &     6.6   &       &  0.561 & 103 & 215.2     &        &        &        & 1,740.5 &  0.136\\
%BP       & 4   &     4.2   & 114.3 & 0.562 & 127 & 259.0     &  0.107  &  127   & 259.0  & 24,417.8&  0.014\\
%HG       & 5   &     10.5  & 118.1 & 0.559 & 8   & 34.7      &  0.591  &   39& 116.6  & 178.5   &  0.419\\ \hline
%\end{tabular}
%\end{center}

We assess predictive performance in high-dimensional gene expression data.
\citeasnoun{calon:2012} used mice experiments to identify 172 genes potentially
related to the gene TGFB,
and showed that these were related to colon cancer progression in an
independent data set with $n=262$ human patients.
TGFB plays a crucial role in colon cancer, hence it is important to understand
its relation to other genes.
Our goal is to predict TGFB in the human data,
first using only the $p=172$ genes and then adding 10,000 extra genes.
Both response and predictors were standardized to zero mean and unit
variance (data in Supplementary Material).
We assessed predictive performance via the leave-one-out cross-validated
$R^2$ coefficient between predictions and observations.
For Bayesian methods we report the posterior expected number of variables
in the model ({\it i.e.} the mean number of predictors used by BMA),
and for SCAD and LASSO the number of selected variables.

Table \ref{tab:tgfb} shows the results.
For $p=172$ all methods achieve similar $R^2$, that for LASSO being slightly higher, although
pMOM, piMOM and BP used substantially less predictors.
These results appear reasonable in a moderately dimensional setting where
genes are expected to be related to TGFB.
However, when using $p=10,172$ predictors important differences between
methods are observed.
The BMA estimates based on MOM and iMOM priors remain parsimonious
(6.5 and 10.3 predictors, respectively)
and the cross-validated $R^2$ increases roughly by 0.05.
In contrast, for the remaining methods the number of predictors increased sharply
(the smallest being 81 predictors for SCAD) and a drop in $R^2$ was observed.
Predictors with large marginal inclusion probabilities in MOM/iMOM
included genes related to various cancer types (ESM1, GAS1, HIC1, CILP, ARL4C, PCGF2),
TGFB regulators (FAM89B) or AOC3 which is used to alleviate certain cancer symptoms.
These findings suggest that NLPs were extremely effective in detecting
a parsimonious subset of predictors in this high-dimensional example.
We also note that computation times were orders of magnitude lower than for LPs,
pMOM being competitive with the penalized likelihood methods.
Although run times depend on implementation issues,
both BMS and mombf are coded in C and follow similar algorithms
({\it e.g.} storing marginal likelihoods in memory, updating one variable at a time).
NLPs focus posterior mass on smaller models, which greatly alleviates the burden required for matrix inversions
(non-linear cost in the model size).
Further, NLPs tend to concentrate posterior probability on a smaller subset of models,
which  tend to be revisited and hence the marginal likelihood need not be recomputed.
Regarding the efficiency of our proposed posterior sampler for $(\bm{\theta},\phi)$,
we ran 10 independent chains with 1,000 iterations each and obtained mean serial correlations of $0.32$ (pMOM) and $0.26$ (piMOM)
across all non-zero coefficients.
The mean correlation between $\hat{E}(\bm{\theta}\mid {\bf y}_n)$ across all chain pairs was $>0.99$ (pMOM and piMOM).

%MOM/iMOM select from p=172 genes
% - 208394_x_at (ESM1, related to colorectal cancer) 
% - 204457_s_at (GAS1, cancer-related variant)
% - 230218_at (HIC1, hyper-methylated in cancer)   
% - 206227_at (CILP, locally advanced breast cancer, reactive stroma of breast and prostate cancer)  

%MOM/iMOM select from p=10,172 genes. Disease associations from genecards.org
% - 202207_at (ARL4C, associated to pancreatic cancer)
% - 204894_s_at (AOC3, used to alleviate symptoms of cancer)
% - 205284_at (URB2, no known association with cancer/TGFB)
% - 32209_at (FAM89B, negatively regulates TGFB-induced signaling)
% - 203792_x_at (PCGF2, associated to leukemia, has tumor suppressor activity)
% - 220301_at (CCDC102B, no known associations with cancer/TGFB)

\section{Discussion}
\label{sec:discussion}

 We showed how combining BMA with NLPs gives a coherent joint framework, encouraging parsimony
in model selection and, for parameter estimation, selective shrinkage focused on spurious coefficients.  
 %Our other important contribution is 
 Coupled with a theoretical investigation of NLP properties, we provide constructions based
on truncation mixtures:  to motivate NLPs from first principles,
add flexibility in prior choice and facilitate posterior sampling,
which we fully developed for linear regression.
We obtained remarkable results when $p>>n$
in simulations and gene expression data, with parsimonious models
achieving accurate cross-validated predictions and good computation times.
We did not require procedures to pre-screen covariates.
Posterior samples a reasonable serial correlations,
captured multi-modalities and independent runs delivered virtually identical estimates.

Our results show that it is not only possible to use the same prior for estimation and selection,
but it may indeed be desirable.
%The shrinkage provided by BMA and point masses appears extremely competitive.
We remark that we used default informative priors,
which are relatively popular for testing, but perhaps less readily adopted for estimation.
Developing objective Bayes strategies to set the prior parameters is an interesting venue for future research,
as well as determining shrinkage rates in more general $p>>n$ cases,
and adapting the latent truncation construction beyond linear regression,
{\it e.g.} generalized linear, graphical or mixture models.

\appendix

\makeatletter   %% HAVE TO ADD SOMETHING HERE TO MAKE IT SAY "APPENDIX"
 \renewcommand{\@seccntformat}[1]{APPENDIX~{\csname the#1\endcsname}.\hspace*{1em}}
 \makeatother

\section{Proofs and Miscellanea}
\label{sec:appendix}

\subsection{Proof of Proposition \ref{prop:parsimony_general}}

We start by stating two useful lemmas.

\begin{lemma}
Let $\pi(\bm{\theta}_k,\phi_k)=d(\bm{\theta}_k,\phi_k) \pi^L(\bm{\theta}_k,\phi_k)$ be either the pMOM, peMOM or piMOM prior,
where $d(\bm{\theta}_k,\phi_k) \rightarrow 0$ as $\theta_{ki} \rightarrow 0$ for any $i=1,\ldots,\mbox{dim}(\bm{\theta}_k)$
and $\pi^L(\bm{\theta}_k,\phi_k)$ is a local prior. 
Then $\pi(\bm{\theta}_k,\phi_k)= \tilde{d}(\bm{\theta}_k,\phi_k) \tilde{\pi}^L(\bm{\theta}_k,\phi_k)$,
where $\tilde{d}(\bm{\theta}_k,\phi_k) \leq c_k$ for some constant $c_k$
and $\tilde{\pi}^L(\bm{\theta}_k,\phi_k)$ is a local prior.
\label{lem:bounded_penalty}
\end{lemma}

\begin{proof}
The result for the peMOM is direct with 
$\tilde{d}_k(\theta_{ki},\phi_k)= \prod_{i=1}^{|k|} e^{\sqrt{2}} e^{-\tau \phi / \theta_{ki}^2} \leq e^{\sqrt{2}|k|}$ and 
$\tilde{\pi}^L(\bm{\theta},\phi)= N(\bm{\theta}; {\bf 0}, \tau \phi I) \pi(\phi)$.
For the piMOM prior we multiply and divide the density by a Cauchy kernel, obtaining
\begin{align}
\pi_k^I(\theta_{ki}\mid \phi_k)&=
\frac{\sqrt{\tau \phi_k}}{\sqrt{\pi} \theta_{ki}^2} e^{-\tau\phi_k/\theta_{ki}^2}
\pi \left(1+\frac{\theta_{ki}^2}{\tau\phi_k} \right) \mbox{Cauchy}(\theta_{ki};0,\phi_k\tau) \nonumber \\
&=\tilde{d}_k(\theta_{ki},\phi_k) \mbox{Cauchy}(\theta_{ki};0,\phi_k\tau),
\label{eq:imom_as_cauchy}
\end{align}
where
$\tilde{d}_k(\theta_{ki},\phi_k)=\sqrt{\pi} \frac{\sqrt{\tau \phi_k}}{\theta_{ki}^2} e^{-\tau\phi_k/\theta_{ki}^2}
\left(1+\theta_{ki}^2/(\tau\phi_k) \right)$.
By performing a change of variables $\eta_i=\theta_{ki}/\sqrt{\tau\phi_k}$
we obtain the implied prior
$\pi_k^I(\eta_i \mid \phi_k)= \sqrt{\pi} \frac{1+\eta_i^2}{\eta_i^2} e^{-1/\eta_i^2} \pi^L(\eta_i \mid \phi_k)$.
Now, $h(\eta_i)=\sqrt{\pi} \frac{1+\eta_i^2}{\eta_i^2}e^{-1/\eta_i^2}$ is continuous,
has positive derivative for all $\eta_i>0$ and negative for $\eta_i<0$,
$\mathop {\lim }\limits_{\eta_i \to 0} h(\eta_i)=0$ 
and $\mathop {\lim }\limits_{\eta_i \to \pm \infty} h(\eta_i)=\sqrt{\pi}$,
and hence $h(\eta_i)\leq \sqrt{\pi}$.
In summary, $c_k=e^{\sqrt{2}|k|}$  for the product eMOM and $c=\pi^{|k|/2}$ for the product iMOM,
where $|k|=\mbox{dim}(\bm{\theta}_k)$.
\label{proof:bounded_penalty}

The pMOM prior density has an unbounded term
$\prod_{i \in M_k}^{}\frac{\theta_{ki}^{2r}}{(2r-1)!! \phi^r \tau^r}$,
but it can be rewritten as $\pi_k^M(\bm{\theta}_k \mid \phi_k)=$
\begin{align}
& \prod_{i \in M_k}^{}\frac{\theta_{ki}^{2r}}{(2r-1)!! \phi_k^r \tau^r} 
\frac{N(\theta_{ki}; 0, \tau \phi_k I)}{N(\theta_{ki}; 0, (1+\epsilon) \tau \phi_k I)}
N(\theta_{ki}; 0, (1+\epsilon) \tau \phi_k I)= \nonumber \\
&\prod_{i \in M_k}^{}\frac{\theta_{ki}^{2r}}{(2r-1)!! \phi_k^r \tau^r} 
\mbox{exp}\left\{ -\frac{1}{2} \frac{\theta_{ki}^2}{\phi_k\tau (1+\epsilon^{-1})}  \right\}
N(\theta_{ki}; 0, (1+\epsilon) \tau \phi_k I)= \nonumber \\
&=\prod_{i \in M_k}^{} \tilde{d}(\theta_{ki},\phi_k) N(\theta_{ki}; 0, (1+\epsilon) \tau \phi_k I)
\label{eq:mom_2xnormal}
\end{align}
for some $\epsilon \in (0,1)$, where it is straightforward to see that $\tilde{d}(\theta_{ki},\phi_k)$ is now bounded.
\end{proof}

\begin{lemma}
Let $d(\bm{\theta})$ be a continuous and differentiable function satisfying $0\leq d(\bm{\theta})<c$ for all $\bm{\theta} \in \Theta$.
Define
$$
g({\bf y}_n)= \int_{}^{}\!\, d(\bm{\theta}) \pi(\bm{\theta} \mid {\bf y}_n) d \bm{\theta},
$$
where $\mathop {\lim }\limits_{n \to \infty } \int_{\bm{\theta} \in N_{\epsilon}(A)}^{}\!\, \pi(\bm{\theta} \mid {\bf y}_n)=1$ almost surely
for any fixed $\epsilon>0$, some set $A$ and a corresponding suitably defined $\epsilon$-neighborhood $N_{\epsilon}(A)$.
If $d(\bm{\theta})= 0$ for all $\bm{\theta} \in A$ then $g({\bf y}_n) \longrightarrow 0$.
Likewise, if $d(\bm{\theta})>c'$ for all $\bm{\theta} \in A$ and some $c'>0$
then $P\left(g({\bf y}_n) \geq c'\right) \longrightarrow 1$ almost surely as $n \longrightarrow \infty$.
In particular, if $A= \{ \bm{\theta}_0 \}$ is a singleton, then $g({\bf y}_n) \longrightarrow g(\bm{\theta}_0)$.
\label{lem:penalty_conv}
\end{lemma}

\begin{proof}
Consider
\begin{align}
g({\bf y}_n)= \int_{\bm{\theta} \in N_{\epsilon}(A)}^{}\!\, d(\bm{\theta}) \pi(\bm{\theta} \mid {\bf y}_n) d \bm{\theta} +
\int_{\bm{\theta} \not\in N_{\epsilon}(A)}^{}\!\, d(\bm{\theta}) \pi(\bm{\theta} \mid {\bf y}_n) d \bm{\theta}=
\label{eq:expand_mean_penalty} \\
\leq \delta_{\epsilon} P\left(\bm{\theta} \in N_{\epsilon}(A)\right) + c P\left(\bm{\theta} \not \in N_{\epsilon}(A)\right) \leq 
\delta_{\epsilon} + c P\left(\bm{\theta} \not \in N_{\epsilon}(A)\right), \nonumber
\end{align}
where $\delta_{\epsilon}=\mbox{max}_{\bm{\theta} \in N_{\epsilon}(A)} d(\bm{\theta})$ and the second term can be made arbitrarily small.
Because $d(\bm{\theta})$ is continuous, if $d(\bm{\theta})=0$ for all $\bm{\theta} \in A$ then $\delta_{\epsilon}$ can also be made
arbitrarily small a.s. as $n \longrightarrow \infty$, and hence $g({\bf y}_n) \longrightarrow 0$.
Suppose now that $d(\bm{\theta})>c'$ for all $\bm{\theta} \in A$, then from (\ref{eq:expand_mean_penalty})
\begin{align}
g({\bf y}_n) > \int_{\bm{\theta} \in N_{\epsilon}(A)}^{}\!\, d(\bm{\theta}) \pi(\bm{\theta} \mid {\bf y}_n) d \bm{\theta} \geq
\delta_{\epsilon}' \int_{\bm{\theta} \in N_{\epsilon}(A)}^{}\!\, \pi(\bm{\theta} \mid {\bf y}_n) d \bm{\theta},
\label{eq:expand_mean_penalty2}
\end{align}
where due to continuity $\delta_{\epsilon}'=\mbox{min}_{\bm{\theta} \in N_{\epsilon}(A)} d(\bm{\theta})$ can be made arbitrarily close to $c'$
for small enough $\epsilon$ and the integral on the right hand side of (\ref{eq:expand_mean_penalty2}) can be made arbitrarily
close to 1 as $n \longrightarrow \infty$.
The proof for when $A=\{ \bm{\theta}_0 \}$ follows as an immediate implication.
\end{proof}

\underline{{\bf Proof of Proposition \ref{prop:parsimony_general}.}}
Part (i) follows from direct algebraic manipulation
\begin{align}
m_k({\bf y}_n)= \int_{}^{}\!\, \int_{}^{}\!\,
f_k({\bf y}_n \mid \bm{\theta}_k,\phi_k) d_k(\bm{\theta}_k,\phi_k) \pi^L(\bm{\theta}_k,\phi_k \mid M_k) d\bm{\theta}_k d\phi_k= \nonumber \\
\int_{}^{}\!\, \int_{}^{}\!\,
d_k(\bm{\theta}_k,\phi_k) \frac{f_k({\bf y}_n \mid \bm{\theta}_k,\phi_k) 
\pi^L(\bm{\theta}_k,\phi_k \mid M_k)}{m_k^L({\bf y}_n)}m_k^L({\bf y}_n) d\bm{\theta}_k d\phi_k= \nonumber \\
m_k^L({\bf y}_n) \int_{}^{}\!\, \int_{}^{}\!\,
d_k(\bm{\theta}_k,\phi_k) \pi^L(\bm{\theta}_k,\phi_k \mid {\bf y}_n, M_k) d\bm{\theta}_k,\phi_k=
m_k^L({\bf y}_n) g_k({\bf y}_n), \nonumber
%\label{eq:marlhood_lp_nlp}
\end{align}
as desired. In a slight abuse of notation, in the derivation above $d \bm{\theta}_k$ and $d\phi_k$
indicate integration with respect to the corresponding $\sigma$-finite dominating measures.

For Part (ii) we use Lemma \ref{lem:bounded_penalty}, which states that the piMOM and peMOM priors
can be written as $d_k(\bm{\theta}_k,\phi_k) \pi^L(\bm{\theta}_k,\phi_k)$ with bounded $d_k(\bm{\theta}_k,\phi_k)$, so that
\begin{align}
g_k({\bf y}_n)= \int_{}^{}\!\, \int_{}^{}\!\, 
d_k(\bm{\theta}_k,\phi_k) \frac{f_k({\bf y}_n \mid \bm{\theta}_k,\phi_k) \pi^L(\bm{\theta}_k,\phi_k)}{m_k^L({\bf y}_n)}
d \bm{\theta}_k d\phi_k
\label{eq:gk_proof_emom}
\end{align}
where by assumption $f_k({\bf y}_n \mid \bm{\theta}_k^*,\phi_k^*)/f_k({\bf y}_n \mid \tilde{\bm{\theta}}_k,\tilde{\phi}_k) \rightarrow \infty$
almost surely as $n \rightarrow \infty$
for any $(\bm{\theta}_k^*,\phi_k^*) \in A$ and $(\tilde{\bm{\theta}}_k,\tilde{\phi}_k) \not\in A$.
See {\it e.g.} \citeasnoun{redner:1981} for such MLE consistency under general settings.% or \citeasnoun{leroux:1992} for extensions to mixture models.
We note that $\pi^L(\bm{\theta}_k,\phi_k)$ associated to either pMOM, piMOM or peMOM priors
are products of independent Normal or Cauchy kernels assigning strictly positive density to any $\bm{\theta}_k \in \Theta_k$,
which combined with MLE consistency guarantee that the limiting posterior concentrates arbitrarily large probability on 
any $\epsilon$ neighborhood of $A$ as $n \rightarrow \infty$ \cite{ghosal:2002}.
Part (ii) follows from Lemma \ref{lem:penalty_conv}.
For the pMOM prior, from Lemma \ref{lem:bounded_penalty} $g_k({\bf y}_n)=$
\begin{align}
\int_{}^{}\!\, \int_{}^{}\!\, 
\tilde{d}_k(\bm{\theta}_k,\phi_k) \frac{f_k({\bf y}_n \mid \bm{\theta}_k,\phi_k) N(\bm{\theta}_k;{\bf 0},\tau \phi_k (1+\epsilon) I)}{m_{k,\tau}^L({\bf y}_n)}
d \bm{\theta}_k d\phi_k= \nonumber \\
 \frac{m_{k,\tau(1+\epsilon)}^L({\bf y}_n)}{m_{k,\tau}^L({\bf y}_n)} 
\int_{}^{}\!\, \int_{}^{}\!\, \tilde{d}_k(\bm{\theta}_k,\phi_k) \pi_{\tau(1+\epsilon)}(\bm{\theta}_k,\phi_k \mid {\bf y}_n)
d \bm{\theta}_k d\phi_k,
\label{eq:gk_proof_mom}
\end{align}
where $\epsilon \in (0,1)$, $\tilde{d}_k(\bm{\theta}_k,\phi_k)$ is bounded and $m_{k,\tau}({\bf y}_n)$ is the integrated likelihood
under a $N(\bm{\theta}_k;{\bf 0},(1+\epsilon)\tau\phi_kI)$ prior.
Part (ii) follows from Lemma \ref{lem:penalty_conv}, which guarantees convergence for the integral in (\ref{eq:gk_proof_mom}),
% the integral either converges to 0 or is greater than a positive constant with probability converging to 1 (as adequate),
and that by assumption $m_{k,\tau(1+\epsilon)}^L({\bf y}_n)/m_{k,\tau}^L({\bf y}_n) \rightarrow c \in (0,\infty)$ almost surely as $n \rightarrow \infty$.
We note that from Bayes theorem
\begin{align}
\frac{m_{k,\tau(1+\epsilon)}^L({\bf y}_n)}{m_{k,\tau}^L({\bf y}_n)}=
\frac{\pi_{(1+\epsilon)\tau}^L(\bm{\theta}_k,\phi_k \mid {\bf y}_n)}{\pi_{\tau}^L(\bm{\theta}_k,\phi_k \mid {\bf y}_n)}
\frac{\pi_{\tau}^L(\bm{\theta}_k,\phi_k)}{\pi_{(1+\epsilon)\tau}^L(\bm{\theta}_k,\phi_k)}
\label{eq:ratio_mlhood_pmom}
\end{align}
for any $(\bm{\theta}_k,\phi_k)$, where the second term in the right hand side is bounded ({\it e.g.} for $\bm{\theta}_k={\bf 0}$).
The first term is the ratio of posterior densities under $N(\bm{\theta}; {\bf 0}, (1+\epsilon)\tau\phi_kI)$ and $N(\bm{\theta}; {\bf 0}, \tau\phi_kI)$,
which for limiting normal posterior distributions with bounded covariance eigenvalues converges in probability to a bounded constant.

In the particular case where the data-generating density $f^*({\bf y}_n)$ belongs to the set of considered models,
{\it i.e.} $f^*({\bf y}_n)=f_t({\bf y}_n \mid \bm{\theta}_t^*,\phi_t^*)$ for some $t \in \{1,\ldots,K\}$.
By definition of NLP we obtain that $d_k(\bm{\theta}_k^*,\phi_k^*)=0$ if and only if $M_t \subset M_k$.
Therefore, $M_t \subset M_k$ implies that $g_k({\bf y}_n) \stackrel{P}{\longrightarrow} 0$
and $M_t \not\subset M_k$ implies 
$P\left( g_k({\bf y}_n) \geq c\right) \longrightarrow 1$ for some constant $c>0$.
We note that for non-identifiable models the set $A$ is no longer a singleton,
but when $M_t \subset M_k$ by definition $d_k(\bm{\theta}_k,\phi_k)=0$ for all $(\bm{\theta}_k,\phi_k) \in A$,
hence we still obtain $g_k({\bf y}_n) \stackrel{P}{\longrightarrow} 0$.
Also by NLP definition, when $M_t \not\subset M_k$ then $d_k(\bm{\theta}_k,\phi_k)>0$ for all $(\bm{\theta}_k,\phi_k) \in A$,
which implies $g_k({\bf y}_n) \stackrel{P}{\longrightarrow} c>0$.

\subsection{Proof of Proposition \ref{prop:parsimony_lm}}

We start by stating two lemmas.

\begin{lemma}
Let ${\bf y}_n \sim N(X_{k,n} \bm{\theta}_k, \phi_k)$ be a linear model as in Proposition \ref{prop:parsimony_lm},
and consider a NLP $\pi(\bm{\theta}_k \mid \phi_k)= d_k(\bm{\theta}_k,\phi_k) \pi^L(\bm{\theta}_k \mid \phi_k)$.
Assume that the NLP penalty takes the product form 
$d_k(\bm{\theta}_k,\phi_k)=\prod_{i \in M_k}^{} d(\theta_{ki},\phi_k)$,
where 
$d_k(\theta_{ki},\phi_k)= \frac{\theta_{ki}^{2r}}{(2r-1)!!(\tau \phi_k)^r}$
is either the MOM penalty
or $d_k(\theta_{ki},\phi_k) \leq c$ for all $(\theta_{ki},\phi_k)$ and some constant $c$,
as in the eMOM or iMOM penalties.
Then $g_k({\bf y}_n)$ is a continuous function of
${\bf s}_{k,n}=({\bf m}_{k,n}, X_{k,n}'X_{k,n}, \hat{\phi}_{k,n})$.
\label{lem:gcont_nlps}
\end{lemma}

\begin{proof}

To prove the result for bounded penalties $d(\theta_{ki},\phi_k) \leq c$
recall that ${\bf s}_{k,n}$ is sufficient under $M_k$
and hence we may write $g_k({\bf y}_n)=g_k({\bf s}_{k,n})=$
\begin{align}
\int_{}^{}\!\, \int_{}^{}\!\, \prod_{i \in M_k}^{} d_k(\bm{\theta}_{ki},\phi_k) \pi^L(\bm{\theta}_k \mid {\bf m}_{k,n}, X_{k,n}'X_{k,n}, \phi_k, M_k) \nonumber \\
\pi^L(\phi_k \mid {\bf m}_{k,n}, X_{k,n}'X_{k,n}, \hat{\phi}_{k,n}, M_k)
d \bm{\theta}_k d\phi_k \leq \nonumber \\
\int_{}^{}\!\, \int_{}^{}\!\, c^{|k|} \pi^L(\bm{\theta}_k \mid {\bf m}_{k,n}, X_{k,n}'X_{k,n}, \phi_k, M_k) \nonumber \\
\pi^L(\phi_k \mid {\bf m}_{k,n}, X_{k,n}'X_{k,n}, \hat{\phi}_{k,n}, M_k)
d \bm{\theta}_k d\phi_k= c^{|k|}.
\end{align}
Now,
letting ${\bf z}=({\bf z}_1,X_{k,n}'X_{k,n},z_2) \to ({\bf m}_{k,n},X_{k,n}'X_{k,n},\hat{\phi}_{k,n})$ 
and using the Dominated Convergence Theorem we obtain
\begin{align}
\mathop {\lim }\limits_{{\bf z} \to {\bf s}_{k,n}} g_k(z) c^{-|k|}=
\int_{}^{}\!\, \int_{}^{}\!\, \prod_{i \in M_k}^{} c^{-1} d_k(\bm{\theta}_{ki},\phi_k) \pi^L(\bm{\theta}_k \mid {\bf z}_1, X_{k,n}'X_{k,n}, \phi_k, M_k) \nonumber \\
\pi^L(\phi_k \mid {\bf z}_1, X_{k,n}'X_{k,n}, z_2, M_k) d \bm{\theta}_k d\phi_k= \nonumber \\
\int_{}^{}\!\, \int_{}^{}\!\, 
\prod_{i \in M_k}^{} c^{-1} d_k(\bm{\theta}_{ki},\phi_k) \pi^L(\bm{\theta}_k \mid {\bf m}_{k,n}, X_{k,n}'X_{k,n}, \phi_k, M_k) \nonumber \\
\pi^L(\phi_k \mid {\bf m}_{k,n}, X_{k,n}'X_{k,n}, \hat{\phi}_{k,n}, M_k) d \bm{\theta}_k d\phi_k,
\end{align}
and hence
\begin{align}
\mathop {\lim }\limits_{{\bf z} \to {\bf s}_{k,n}} g_k(z)=  
\int_{}^{}\!\, \int_{}^{}\!\, 
\prod_{i \in M_k}^{} d_k(\bm{\theta}_{ki},\phi_k) \pi^L(\bm{\theta}_k \mid {\bf m}_{k,n}, X_{k,n}'X_{k,n}, \phi_k, M_k) \nonumber \\
\pi^L(\phi_k \mid {\bf m}_{k,n}, X_{k,n}'X_{k,n}, \hat{\phi}_{k,n}, M_k) d \bm{\theta}_k d\phi_k,
\end{align}
showing that $g_k({\bf s}_{k,n})$ is continuous.

Consider now the MOM prior case.
For the particular prior choice $\phi_k \sim \mbox{IG}(\alpha,\lambda)$,
\citeasnoun{johnson:2012} showed that
$g_k({\bf s}_{k,n})=E \left(\prod_{i \in M_k}^{} \theta_{ki}^{2r} \right)$ where
$\bm{\theta}_k \sim T_{\nu}({\bf m}_{k,n}, V_{k,n})$,
with $\nu=2r|k|+n+\alpha$ and
$V_{k,n}=S_{k,n} \nu / (\lambda + {\bf y}_n' {\bf y}_n - {\bf y}_n'X_{k,n}{\bf m}_{k,n}$.
\citeasnoun{kan:2008} gave explicit expressions for such products as a sum of continuous
functions, and hence $g_k({\bf s}_{k,n})$ is continuous.
Lemma \ref{lem:bounded_penalty} ensures that the pMOM penalty is also bounded
for more general priors $\pi_k(\phi_k)$.
Therefore, $g_k({\bf s}_{k,n})=$
\begin{align}
\int_{}^{}\!\, \int_{}^{}\!\, \prod_{i \in M_k}^{} d(\theta_{ki},\phi_k) 
\frac{N({\bf y}_n; X_{k,n} \bm{\theta}_k; \phi_k I) N(\bm{\theta}_k; {\bf 0}, 2\tau \phi_k I)}
{m_{k,\tau}^L({\bf y}_n)} \pi_k(\phi_k)
d\bm{\theta}_k d\phi_k= \nonumber \\
\frac{m_{k,2\tau}^L({\bf y}_n)}{m_{k,\tau}^L({\bf y}_n)}
\int_{}^{}\!\, \int_{}^{}\!\, \prod_{i \in M_k}^{} d(\theta_{ki},\phi_k) 
\pi_{k,2\tau}^L(\bm{\theta}_k \mid \phi_k, {\bf s}_n)
\pi_k(\phi_k) d\bm{\theta}_k d\phi_k
\label{eq:mom_2xnormal_cont}
\end{align}
where $m_{k,\tau}^L({\bf y}_n)$ is the integrated likelihood with respect to
$N(\bm{\theta}_k; {\bf 0}, \tau \phi_k I)$
and
$\pi_{k,2\tau}^L(\bm{\theta}_k \mid \phi_k, {\bf s}_n)$ is the Normal posterior implied by
the $N(\bm{\theta}_k; {\bf 0}, 2\tau\phi I)$ prior.
Because $d(\theta_{ki},\phi_k) \leq c$ for some constant $c$,
the Dominated Convergence Theorem gives that
\begin{align}
\mathop {\lim }\limits_{{\bf z} \to {\bf s}_{k,n}} g_k({\bf z})
\frac{m_{k,\tau}^L({\bf y}_n)}{m_{k,2\tau}^L({\bf y}_n)}=
\int_{}^{}\!\, \int_{}^{}\!\, \prod_{i \in M_k}^{} d(\theta_{ki},\phi_k) 
\pi_{k,2\tau}^L(\bm{\theta}_k \mid \phi_k, {\bf s}_n)
\pi_k(\phi_k) d\bm{\theta}_k d\phi_k,
\label{eq:glim_mom}
\end{align}
so that direct algebraic manipulation after adding the integrated likelihood terms delivers
\begin{align}
\mathop {\lim }\limits_{{\bf z} \to {\bf s}_{k,n}} g_k({\bf z})=
\int_{}^{}\!\, \int_{}^{}\!\, \prod_{i \in M_k}^{} \frac{\theta_{ki}^{2r}}{(2r-1)!! \phi^r \tau^r}
\pi_{k,\tau}^L(\bm{\theta}_k \mid \phi_k, {\bf s}_n)
\pi_k(\phi_k) d\bm{\theta}_k d\phi_k,
\label{eq:glim_mom2}
\end{align}
which proves that $g_k({\bf s}_{k,n})$ is continuous.
\label{proof:gcont_nlps}
\end{proof}

\begin{lemma}
Let $d_k(\bm{\theta}_k,\phi_k)$ be as in Lemma \ref{lem:gcont_nlps}
and $c_n=$
\begin{align}
\int_{}^{}\!\, \int_{}^{}\!\, &d_k(\bm{\theta}_k,\phi_k)
\pi_k^L(\bm{\theta}_k \mid \bm{\theta}_k^*, X_{k,n}'X_{k,n}, \phi_k)
\pi_k^L(\phi_k \mid \bm{\theta}_k^*, X_{k,n}'X_{k,n}, \phi_k^*)
d\bm{\theta}_k d\phi_k
\label{eq:cn_lemma}
\end{align}
as in (\ref{eq:gconv_to_cn}).
Then 
\begin{align}
\mathop {\lim }\limits_{n \to \infty} c_n=
\int_{}^{}\!\, \int_{}^{}\!\, &d_k(\bm{\theta}_k,\phi_k) \mathop {\lim }\limits_{n \to \infty} 
\pi_k^L(\bm{\theta}_k \mid \bm{\theta}_k^*, X_{k,n}'X_{k,n}, \phi_k) \nonumber \\
&\pi_k^L(\phi_k \mid \bm{\theta}_k^*, X_{k,n}'X_{k,n}, \phi_k^*)
d\bm{\theta}_k d\phi_k
\label{eq:cn_lemma2}
\end{align}
\label{lem:cn_convergence}
\end{lemma}

\begin{proof}
The proof runs analogous to that in Lemma \ref{lem:gcont_nlps},
except that now the limit is taken with respect to $n$
and the $c^{-|k|}$ term may grow as $n \to \infty$.
That is, for bounded $d(\theta_{ki},\phi_k)$ the argument proceeds by
using the Dominated Convergence Theorem to obtain
$\mathop {\lim }\limits_{n \to \infty} c_n c^{-|k|}=$
\begin{align}
\int_{}^{}\!\, \prod_{i \in M_k}^{} c^{-1} d_k(\theta_{ki},\phi_k)
\mathop {\lim }\limits_{n \to \infty}
\pi_k^L(\bm{\theta}_k \mid \bm{\theta}_k^*, X_{k,n}'X_{k,n}, \phi_k)
\pi_k^L(\phi_k) d\phi_k,
\end{align}
so that 
\begin{align}
\mathop {\lim }\limits_{n \to \infty} c_n=
\int_{}^{}\!\, \prod_{i \in M_k}^{} d_k(\theta_{ki},\phi_k)
\mathop {\lim }\limits_{n \to \infty}
\pi_k^L(\bm{\theta}_k \mid \bm{\theta}_k^*, X_{k,n}'X_{k,n}, \phi_k)
\pi_k^L(\phi_k) d\phi_k.
\end{align}

For the MOM prior we adjust the argument slightly.
From (\ref{eq:hn}) we obtain $c_n=$
\begin{align}
\int_{}^{}\!\, \int_{}^{}\!\, &\prod_{i \in M_k}^{} \frac{\theta_{ki}^{2r}}{(2r-1)!! \tau^2 \phi^2}
N(\bm{\theta}_k; \bm{\theta}_k^*, \phi_k(X_{k,n}'X_{k,n})^{-1})
\mbox{IG}\left( \phi_k; \frac{n}{2}, \frac{n\phi_k^*}{2} \right) \nonumber \\
&c_{\phi}^*(X_{k,n},\tau) N(\bm{\theta}_k; 0, \tau \phi_k I) \pi_k(\phi_k) d\bm{\theta}_k^* d\phi_k,
\label{eq:cn_momprior}
\end{align}
where $c_{\phi}^*(X_{k,n},\tau)=c_{\phi}({\bf s}_{k,n})$ in (\ref{eq:penalty_as_funsuffstat2})
plugging in $\bm{\theta}_{k,n}=\bm{\theta}_k^*$, $\hat{\phi}_k=\phi_k^*$.
Following the same argument as in the proof of Lemma \ref{lem:gcont_nlps},
we divide and multiply by a $N(\bm{\theta}_k; {\bf 0}, 2\tau)$ kernel to obtain

\begin{align}
c_n=\int_{}^{}\!\, \int_{}^{}\!\, \prod_{i \in M_k}^{} d(\theta_{ki},\phi_k)
N(\bm{\theta}_k; \bm{\theta}_k^*, \phi_k(X_{k,n}'X_{k,n})^{-1})
N(\bm{\theta}_k; 0, 2\tau \phi_k I) \nonumber \\
 c_{\phi}^*(X_{k,n},\tau) \mbox{IG}\left( \phi_k; \frac{n}{2}, \frac{n\phi_k^*}{2} \right)
\pi_k(\phi_k) d\bm{\theta}_k^* d\phi_k= \nonumber \\
\frac{c_{\phi}^*(X_{k,n},\tau)}{c_{\phi}^*(X_{k,n},2\tau)} 
\int_{}^{}\!\, \int_{}^{}\!\, \prod_{i \in M_k}^{} d(\theta_{ki},\phi_k)
\frac{N(\bm{\theta}_k; {\bf m}_{2\tau}, S_{2\tau})}{c_{\theta}^*(\phi,X_{k,n},2\tau)} \nonumber \\
c_{\phi}^*(X_{k,n},2\tau) \mbox{IG}\left( \phi_k; \frac{n}{2}, \frac{n\phi_k^*}{2} \right) 
 \pi_k(\phi_k) d\bm{\theta}_k^* d\phi_k,
\label{eq:cn_momprior2}
\end{align}
where 
$d(\theta_{ki},\phi_k)=\frac{\theta_{ki}^{2r} N(\theta_{ki};0,\tau\phi_k)}{(2r-1)!! \tau^2 \phi^2 N(\theta_{ki};0,2\tau\phi_k)} \leq c$
for some constant $c$,
$S_{2\tau}= X_{n,k}'X_{n,k} + (2\tau)^{-1}I$,
${\bf m}_{2\tau}= S_{2\tau}^{-1}(X_{k,n}'X_{k,n}) \bm{\theta}_k^*$,
and 
$1/c_{\theta}^*(\phi,X_{k,n},2\tau)=
\int_{}^{}\!\, N(\bm{\theta}_k; \bm{\theta}_k^*, \phi_k(X_{k,n}'X_{k,n})^{-1})
N(\bm{\theta}_k; 0, 2\tau \phi_k I) d\bm{\theta}_k$.
Now, because $d(\theta_{ki},\phi_k)$ is bounded and the remaining expression
in (\ref{eq:cn_momprior2}) is a probability density function on $(\bm{\theta}_k,\phi_k)$,
the Dominated Convergence Theorem gives
\begin{align}
\mathop {\lim }\limits_{n \to \infty} c_n
\frac{c_{\phi}^*(X_{k,n},2\tau)}{c_{\phi}^*(X_{k,n},\tau)} c^{-|k|/2}=
\int_{}^{}\!\, \int_{}^{}\!\, \mathop {\lim }\limits_{n \to \infty}
\prod_{i \in M_k}^{} d(\theta_{ki},\phi_k)
\frac{N(\bm{\theta}_k; {\bf m}_{2\tau}, S_{2\tau})}{c_{\theta}^*(\phi,X_{k,n},2\tau)} \nonumber \\
c_{\phi}^*(X_{k,n},2\tau) \mbox{IG}\left( \phi_k; \frac{n}{2}, \frac{n\phi_k^*}{2} \right) 
 \pi_k(\phi_k) d\bm{\theta}_k^* d\phi_k,
\label{eq:limcn_momprior}
\end{align}
which after rearranging terms gives
\begin{align}
\mathop {\lim }\limits_{n \to \infty} c_n=
\int_{}^{}\!\, \int_{}^{}\!\, \mathop {\lim }\limits_{n \to \infty}
\prod_{i \in M_k}^{} \frac{\theta_{ki}^{2r}}{\tau^2\phi_k^2 (2r-1)!!}
N(\bm{\theta}_k; {\bf m}_{k,n}, \phi_k S_{k,n}^{-1}) \nonumber \\
c_{\phi}(X_{k,n},\tau) \mbox{IG}\left( \phi_k; \frac{n}{2}, \frac{n\phi_k^*}{2} \right) 
 \pi_k(\phi_k) d\bm{\theta}_k^* d\phi_k,
\label{eq:limcn_momprior2}
\end{align}
concluding the proof.

\end{proof}

We now proceed to prove Proposition \ref{prop:parsimony_lm}.

For Part (1) we note that 
$an < l_1(X_{k,n}'X_{k,n}) < l_{k}(X_{k,n}'X_{k,n}) < bn$
gives $||(X_{k,n}'X_{k,n})^{-1}||_2^2 \leq 1/an \to 0$ for fixed $a$,
which in turn guarantees
$\hat{\bm{\theta}}_{k,n} \stackrel{a.s.}{\longrightarrow} \bm{\theta}_k^*$ \cite{lai:1979}.
This implies 
$\hat{\phi}_{k,n}=n^{-1}({\bf y}_n-X_{k,n} \hat{\bm{\theta}}_{k,n})'({\bf y}_n-X_{k,n} \hat{\bm{\theta}}_{k,n})
\stackrel{a.s.}{\longrightarrow} n^{-1} ({\bf y}_n-X_{k,n} \bm{\theta}_k^*)'({\bf y}_n-X_{k,n} \bm{\theta}_k^*)
\stackrel{a.s.}{\longrightarrow} \phi_k^*$,
given that $V(Y-X_{k,n} \bm{\theta}_k^*)=\phi_k^*<\infty$ by assumption.
Hence, $d_k(\hat{\bm{\theta}}_{k,n},\hat{\phi}_{k,n}) \stackrel{a.s.}{\longrightarrow} d_k(\bm{\theta}_k^*,\phi_k^*)$.
Since $m_{k,n} \stackrel{P}{\longrightarrow} \hat{\bm{\theta}}_{k,n}$ as $n \to \infty$
and $d_k(\bm{\theta}_k,\phi_k)$ is assumed continuous,
the Continuous Mapping Principle gives $d_k({\bf m}_{k,n},\phi_k) \stackrel{a.s.}{\longrightarrow} d_k(\bm{\theta}_k^*,\phi_k^*)$.

To show that $g_k({\bf y}_n) \stackrel{P}{\longrightarrow} d_k({\bf m}_{k,n},\phi_k)$,
we note that ${\bf s}_{k,n}=({\bf m}_{k,n}, X_{k,n}'X_{k,n}, \hat{\phi}_{k,n})$
is a one-to-one function with the sufficient statistic 
$(\hat{\bm{\theta}}_{k,n}, X_{k,n}'X_{k,n}, \hat{\phi}_{k,n})$ under $M_k$.
Hence ${\bf s}_{k,n}$ is also sufficient and $g_k({\bf y}_n)$ depends only on ${\bf s}_{k,n}$,
so that we may write $g_k({\bf s}_{k,n})=$
\begin{align}
\int_{}^{}\!\, \int_{}^{}\!\, d_k(\bm{\theta}_k,\phi_k) \pi_k^L(\bm{\theta}_k \mid {\bf m}_{k,n}, X_{k,n}'X_{k,n}, \phi_k)
\pi_k^L(\phi_k \mid {\bf s}_{k,n})
d \bm{\theta}_k d\phi_k,
\label{eq:penalty_as_funsuffstat}
\end{align}
where straightforward algebra shows that
$\pi_k^L(\bm{\theta}_k \mid {\bf m}_{k,n}, X_{k,n}'X_{k,n}, \phi_k)=$
\begin{align}
c_{\theta}(\phi_k, {\bf s}_{k,n})
N(\bm{\theta}_k; \hat{\bm{\theta}}_{k,n},\phi_k(X_{k,n}'X_{k,n})^{-1}) \pi_k^L(\bm{\theta}_k \mid \phi_k) \nonumber \\
\pi_k^L(\phi_k \mid {\bf s}_{k,n})=
\frac{c_{\phi}({\bf s}_{k,n})}{c_{\theta}(\phi_k, {\bf s}_{k,n})}
\phi_k^{-(n-k)/2} e^{-\frac{1}{2\phi_k}({\bf y}_n' {\bf y}_n - \hat{\bm{\theta}}_{k,n}'X_{k,n}'X_{k,n}\hat{\bm{\theta}}_{k,n})},
\label{eq:penalty_as_funsuffstat2}
\end{align}
where $c_{\theta}(\phi_k, {\bf s}_{k,n})$
is the normalization constant for $\bm{\theta}_k$
(which may depend on $\phi_k$)
and $c_{\phi}({\bf s}_{k,n})$ that for the marginal posterior of $\phi_k$.

Lemma \ref{lem:gcont_nlps} gives that
$g_k({\bf s}_{k,n})$ is continuous
in ${\bf s}_{k,n}=({\bf m}_{k,n}, X_{k,n}'X_{k,n}, \hat{\phi}_{k,n})$,
hence by the Continuous Mapping Principle
\begin{align}
g_k({\bf s}_{k,n}) \stackrel{P}{\longrightarrow} \int_{}^{}\!\, \int_{}^{}\!\, &d_k(\bm{\theta}_k,\phi_k)
\pi_k^L(\bm{\theta}_k \mid \bm{\theta}_k^*, X_{k,n}'X_{k,n}, \phi_k)  \nonumber \\
&\pi_k^L(\phi_k \mid \bm{\theta}_k^*, X_{k,n}'X_{k,n}, \phi_k^*)
d\bm{\theta}_k d\phi_k=c_n,
\label{eq:gconv_to_cn}
\end{align}
where $\pi_k^L(\phi_k \mid \bm{\theta}_k^*, X_{k,n}'X_{k,n}, \phi_k^*) \propto
c_{\theta}(\phi_k, \bm{\theta}_k^*, X_{k,n}'X_{k,n})^{-1}
\phi_k^{-(n-k)/2} e^{-\frac{n\phi_k^*}{2\phi_k}}$.
%\begin{align}
%%\frac{c_{\phi}(\bm{\theta}_k^*,X_{k,n}'X_{k,n},\phi_k^*)}{c_{\theta}(\phi_k, \bm{\theta}_k^*, X_{k,n}'X_{k,n})}
%c_{\theta}(\phi_k, \bm{\theta}_k^*, X_{k,n}'X_{k,n})^{-1}
%\phi_k^{-(n-k)/2} e^{-\frac{n\phi_k^*}{2\phi_k}}
%\end{align}
For a fixed sequence of $X_{k,n}$ (\ref{eq:gconv_to_cn}) is just a sequence in $n$.
To complete the proof we just need to show that
$\mathop {\lim }\limits_{n \to \infty } c_n \to d_k(\bm{\theta}_k^*,\phi_k^*)$
for any sequence $X_{k,n}$ satisfying the theorem assumptions,
which combined with
$d_k({\bf m}_{k,n},\phi_k^*) \stackrel{P}{\longrightarrow} d_k(\bm{\theta}_k^*,\phi_k^*)$
would give that $g_k({\bf s}_{k,n}) \stackrel{P}{\longrightarrow} d_k({\bf m}_{k,n},\phi_k^*)$.
By Lemma \ref{lem:cn_convergence},
\begin{align}
\mathop {\lim }\limits_{n \to \infty} c_n=
\int_{}^{}\!\, \int_{}^{}\!\, & d_k(\bm{\theta}_k,\phi_k) \mathop {\lim } \limits_{n \to \infty}
\pi_k^L(\bm{\theta}_k \mid \bm{\theta}_k^*, X_{k,n}'X_{k,n}, \phi_k)  \nonumber \\
&\pi_k^L(\phi_k \mid \bm{\theta}_k^*, X_{k,n}'X_{k,n}, \phi_k^*)
d\bm{\theta}_k d\phi_k^* = \nonumber \\
\int_{}^{}\!\, \int_{}^{}\!\, & d_k(\bm{\theta}_k,\phi_k) \mathop {\lim } \limits_{n \to \infty}
h_n(\bm{\theta}_k,\phi_k) d\bm{\theta}_k d\phi_k^*.
\label{eq:lim_cn}
\end{align}
Now, from (\ref{eq:penalty_as_funsuffstat})-(\ref{eq:gconv_to_cn})
we obtain
$h_n(\bm{\theta}_k,\phi_k) \propto$
\begin{align}
N(\bm{\theta}_k; \bm{\theta}_k^*, \phi_k(X_{k,n}'X_{k,n})^{-1})
\mbox{IG}(\phi_k; n/2, n\phi_k^*/2)
\pi_k^L(\bm{\theta}_k \mid \phi_k) \pi_k^L(\phi_k),
\label{eq:hn}
\end{align}
where $\mbox{IG}$ denotes the inverse gamma density function.

Informally,
given the assumptions on $\pi_k^L(\bm{\theta}_k \mid \phi_k)$,
for (\ref{eq:hn}) to converge to a point mass at $(\bm{\theta}_k^*,\phi_k^*)$
we need the trace of $(X_{k,n}'X_{k,n})^{-1}$ to converge to 0.
Note that $\mbox{tr}((X_{k,n}'X_{k,n})^{-1}) \leq k/l_1$,
which is satisfied as long as $l_1$ grows faster with $n$ than $k$ does,
and that under our assumptions $k/l_1< k/(an) \rightarrow 0$.
Formally, the conditions on the eigenvalues of $X_{k,n}'X_{k,n}$ imply that for $n>n_0$,
\begin{align}
h_n(\bm{\theta}_k) \leq \mbox{IG}(\phi_k; n/2, n\phi_k^*/2) \pi^L(\phi_k) \times \nonumber \\
\times \left( \frac{b}{a} \right)^{|k|/2} \frac{ (na)^{|k|/2}}{(2\pi)^{|k|/2} \phi_k^{|k|/2}}
\mbox{exp}\left\{ -\frac{na}{2\phi_k} (\bm{\theta}_k - \bm{\theta}_k^*)'(\bm{\theta}_k-\bm{\theta}_k^*) \right\}
\pi_k^L(\bm{\theta}_k \mid \phi_k) 
\label{eq:hn_bound}
\end{align}
We first study the second line in (\ref{eq:hn_bound}).
Given that $|k|=o(n)$, for bounded $\pi^L$ we have
 $\pi^L(\bm{\theta}_k \mid \phi_k) < \infty$ for all $\bm{\theta}_k$
the second line in (\ref{eq:hn_bound}) converges to 0 as $n \to \infty$ for 
any given $\phi_k$ and all $\bm{\theta}_k \neq \bm{\theta}_k^*$,
{\it i.e.} $\pi^L(\bm{\theta}_k \mid \bm{\theta}_k^*, X_{k,n}'X_{k,n}, \phi_k)$
converges to a point mass at $\bm{\theta}_k^*$.

Now suppose that $\pi^L(\bm{\theta}_k \mid \phi_k)$ is unbounded in a 0 Lebesgue measure
set $\tilde{\Theta}_k$.
In this case it also holds that
\begin{align}
\mathop {\lim }\limits_{\bm{\theta}_k \to \tilde{\bm{\theta}}_k}
\mbox{exp}\left\{ -\frac{na}{2\phi_k} (\bm{\theta}_k - \bm{\theta}_k^*)'(\bm{\theta}_k-\bm{\theta}_k^*) \right\}
\pi_k^L(\bm{\theta}_k \mid \phi_k) = 0
\end{align}
for any $\tilde{\bm{\theta}}_k \in \tilde{\Theta}_k$.
This can be seen by contradiction, {\it i.e.} assume that 
for $||\bm{\theta}_k - \tilde{\bm{\theta}}_k||^2<\epsilon$
and an arbitrary small $\epsilon$ there exists some $\delta>0$ such that
$\mbox{exp}\left\{ -\frac{na}{2\phi_k} (\bm{\theta}_k - \bm{\theta}_k^*)'(\bm{\theta}_k-\bm{\theta}_k^*) \right\}
\pi_k^L(\bm{\theta}_k \mid \phi_k) > \delta$
for some arbitrarily large values of $n$.
Then the prior probability of $||\bm{\theta}_k - \tilde{\bm{\theta}}_k||^2<\epsilon$
\begin{align}
\int_{||\bm{\theta}_k - \tilde{\bm{\theta}}_k||^2<\epsilon}^{}\!\, \pi_k^L(\bm{\theta}_k \mid \phi_k) d\bm{\theta}_k>
\delta \int_{||\bm{\theta}_k - \tilde{\bm{\theta}}_k||^2<\epsilon}^{}\!\, 
\mbox{exp}\left\{ \frac{na}{2\phi_k} (\bm{\theta}_k - \bm{\theta}_k^*)'(\bm{\theta}_k-\bm{\theta}_k^*) \right\},
\label{eq:prob_neigh_unbounded}
\end{align}
but the integrand is positive and increasing with $n$
and hence by the Monotone Convergence Theorem (\ref{eq:prob_neigh_unbounded}) converges to $\infty$
as $n \to \infty$,
which would imply that $\pi_k^L(\bm{\theta}_k \mid \phi_k)$ is improper.

Finally, we note that given that $\pi_k^L(\phi_k)$ is bounded and continuous
the first line in (\ref{eq:hn_bound}) converges to 0 as $n \to \infty$
for any $\phi_k \neq \phi_k^*$,
hence (\ref{eq:lim_cn}) converges to $d(\bm{\theta}_k^*,\phi_k^*)$,
which completes the proof.

\subsection{Proof of Proposition \ref{prop:postmode}, Part (i)}

For ease of notation we drop the subindex $k$ indicating the model and we denote $\mbox{dim}(\bm{\theta})=|k|$.
Consider first the pMOM prior and take $\tau=1$ without loss of generality. The log-posterior density is
\begin{align}
L_n(\bm{\theta},\phi) + \sum_{i=1}^{|k|} \mbox{log}(\theta_i^2) - p\mbox{log}(\phi) - \frac{1}{2\phi} \sum_{i=1}^{|k|} \theta_i^2 + \mbox{log}\pi(\phi),
\label{eq:pmom_logpos}
\end{align}
where $L_n(\bm{\theta},\phi)$ is the log-likelihood and $\pi(\phi)$ is the prior density on $\phi$.
Suppose that the sampling model satisfies the conditions in \citeasnoun{walker:1969},
then $L_n(\bm{\theta},\phi)$ can be approximated by a second order Taylor expansion around an MLE of $L_n(\bm{\theta},\phi)$.
Performing this expansion and setting the partial derivative with respect to $\theta_i$ to 0 delivers
\begin{align}
\sum_{j=1}^{|k|} h_{ij} (\tilde{\theta}_j - \hat{\theta}_j) + \frac{2}{\tilde{\theta}_i} - \frac{\tilde{\theta}_i}{\phi}=0,
\label{eq:pmom_partialderiv}
\end{align}
where $h_{ij}$ is the $(i,j)$ element of the Hessian of $L_n(\bm{\theta},\phi)$ evaluated at $(\bm{\theta},\phi)=(\hat{\bm{\theta}},\hat{\phi})$.
Rearranging terms we obtain
\begin{align}
\tilde{\theta}_i \left( \frac{\tilde{\theta}_i}{n\phi} - \frac{h_{ii}}{n} (\tilde{\theta}_i - \hat{\theta}_i) \right)
- \tilde{\theta}_i \sum_{j \neq i}^{} \frac{h_{ij}}{n} (\tilde{\theta}_j - \hat{\theta}_j) - \frac{2}{n} =0
\label{eq:pmom_partialderiv2}
\end{align}
We note that the Taylor approximation to (\ref{eq:pmom_logpos}) is a quadratic form in $\bm{\theta}$,
which is convex in $\bm{\theta}$, plus $\sum_{i=1}^{|k|}\mbox{log}(\theta_i^2)$ which is convex in each quadrant of $\mathbb{R}^p$
({\it i.e.} for fixed sign of $\theta_1,\ldots,\theta_{|k|}$) and converges to $-\infty$ as any $\theta_i \longrightarrow 0$.
Therefore the function to maximize has a global maxima at the quadrant where $\hat{\bm{\theta}}$ occurs and a local maxima in each other quadrant.
Consider first the two modes for $\theta_i$ occurring when $\mbox{sign}(\tilde{\theta}_j)=\mbox{sign}(\hat{\theta}_j)$ for $j \neq i$.
Under Walker's conditions $\tilde{\theta}_j - \hat{\theta}_j \stackrel{P}{\longrightarrow} 0$,
$h_{ij}/n \stackrel{P}{\longrightarrow} J_{ij}$ with finite $J_{ij}$ (condition B4),
and $\tilde{\phi} \stackrel{P}{\longrightarrow} \phi^*$,
where $(\bm{\theta}^*,\phi^*)$ minimizes KL divergence to the data-generating model
and we assume that $\phi^*>0$.
Incorporating these facts into (\ref{eq:pmom_partialderiv2}) gives that any posterior mode must satisfy
%\begin{align}
%nJ_{ii}\tilde{\theta}_i(\tilde{\theta}_i-\hat{\theta}_i) -2 \stackrel{P}{\longrightarrow} 0,
%\label{eq:pmom_partial_asymp}
%\end{align}
%which implies that 
\begin{align}
n\tilde{\theta}_i(\tilde{\theta}_i-\hat{\theta}_i) \stackrel{P}{\longrightarrow} c
\label{eq:pmom_partial_asymp2}
\end{align}
with $0<c<\infty$.
We note that (\ref{eq:pmom_partial_asymp2}) remains valid for linear models with bounded eigenvalues as stated in the proposition assumptions,
given that then $L_n(\bm{\theta},\phi)$ is exactly quadratic and the condition ensures almost sure convergence of the MLE,
which implies $\tilde{\theta}_j - \hat{\theta}_j \stackrel{P}{\longrightarrow} 0$
and $\tilde{\phi} \stackrel{P}{\longrightarrow} \phi^*$.
Now suppose that $\theta_i^* \neq 0$, then for one mode
$\tilde{\theta}_i \stackrel{P}{\longrightarrow} \theta_i^*$ and thus
$n(\tilde{\theta}_i-\hat{\theta}_i) \stackrel{P}{\longrightarrow} c/\theta_i^*$,
whereas solving (\ref{eq:pmom_partial_asymp2} gives that the other mode $\sqrt{n} \tilde{\theta}_i \stackrel{P}{\longrightarrow} c'$.
Next assume that $\theta_i^* =0$, then $\hat{\theta}_i \stackrel{P}{\longrightarrow} 0$
and hence $n \tilde{\theta}_i^2 \stackrel{P}{\longrightarrow} c$.
This implies that $\sqrt{n} \tilde{\theta}_i \stackrel{P}{\longrightarrow} \sqrt{c}$,
which from (\ref{eq:pmom_partial_asymp2}) implies that
$\sqrt{n} (\tilde{\theta}_i - \hat{\theta}_i) \stackrel{P}{\longrightarrow} c'$ with $0<c'<\infty$.
To summarize, the modes $\tilde{\theta}_i$ when $\mbox{sign}(\tilde{\theta}_j)=\mbox{sign}(\hat{\theta}_j)$ for $j \neq i$
are either $O_p(n^{-1})$ from the MLE or $O_p(n^{-1/2})$ from 0.
All other modes are given by the intersection of the contours of an ellipse centered at $\hat{\bm{\theta}}$ and all axis lengths
shrinking at rate $O(n^{-1})$ (from boundedness of eigenvalues)
with the isocontours $\sum_{i=1}^{|k|} \mbox{log}(\theta_i)=c$ for some $0<c<\infty$ (which do not depend on $n$).
Hence all modes $\tilde{\bm{\theta}}$ occurring at quadrants other than that of $\hat{\bm{\theta}}$
shrink towards 0 at the same rate, {\it i.e.} $\tilde{\theta}_i=O_p(n^{-1/2})$.

The proof for the piMOM and peMOM priors follows in an analogous fashion.
Performing a second order Taylor approximation to the piMOM posterior around an MLE $\hat{\bm{\theta}}$
and setting the partial derivative with respect to $\theta_i$ to 0 delivers
that $\tilde{\theta}_i$ and $\tilde{\phi}_i$ must satisfy
\begin{align}
\sum_{j=1}^{|k|} h_{ij} (\tilde{\theta}_i-\hat{\theta}_i) + h_{i,|k|+1} (\tilde{\phi}-\hat{\phi}) - \frac{2}{\tilde{\theta}_i} + \frac{2\tilde{\phi}}{\tilde{\theta}_i^3}=0,
\label{eq:pimom_partial}
\end{align}
where as before $h_{ij}$ indicates the Hessian of $L_n(\bm{\theta},\phi)$ evaluated at $(\hat{\bm{\theta}},\hat{\phi})$.
Rearranging terms delivers
\begin{align}
\frac{h_{ii}}{n} \tilde{\theta}_i^3(\tilde{\theta}_i-\hat{\theta}_i) +
\tilde{\theta}_i^3 \left( \sum_{j \neq i}^{} \frac{h_{ij}}{n} (\tilde{\theta}_i-\hat{\theta}_i) + \frac{h_{i,|k|+1}}{n} (\tilde{\phi}-\hat{\phi}) \right)
- \frac{2}{n} \tilde{\theta}_i^2 + \frac{2\phi}{n}=0.
\label{eq:pimom_partial2}
\end{align}
We again consider the modes for $\tilde{\theta}_i$ when $\mbox{sign}(\tilde{\theta}_j)=\mbox{sign}(\hat{\theta}_j)$ for $j \neq i$.
Either Walker's conditions for general models or the eigenvalue conditions in the linear model case
guarantee that $h_{ij}/n \stackrel{P}{\longrightarrow} J_{ij}$ for all $i,j$,
whereas MLE consistency gives that $(\tilde{\theta}_i-\hat{\theta}_i) \stackrel{P}{\longrightarrow} 0$,
and $\tilde{\phi}-\hat{\phi} \stackrel{P}{\longrightarrow} 0$.
Therefore, $\tilde{\theta}_i$ must satisfy
%$n \tilde{\theta}_i^3(\tilde{\theta}_i-\hat{\theta}_i) - 2\tilde{\theta}_i^2 + 2\tilde{\phi} \stackrel{P}{\longrightarrow} 0$, so that
\begin{align}
n \tilde{\theta}_i^3(\tilde{\theta}_i-\hat{\theta}_i) \stackrel{P}{\longrightarrow} c,
\label{eq:pimom_partial_asymp}
\end{align}
where $0<c<\infty$. Consider the case where the true parameter value $\theta_i^* \neq 0$,
then for one mode $\tilde{\theta}_i^3 \stackrel{P}{\longrightarrow} (\theta_i^*)^3 \neq 0$
and hence $n(\tilde{\theta}_i-\hat{\theta}_i) \stackrel{P}{\longrightarrow} c'$ with $0<c'<\infty$,
whereas from $\ref{eq:pimom_partial_asymp}$ the other mode $n\tilde{\theta}_i^4 \stackrel{P}{\longrightarrow} c'$.
Now consider the case when $\theta_i^*=0$, then $\hat{\theta}_i \stackrel{P}{\longrightarrow} 0$
%and hence $nJ_{ii}\tilde{\theta}_i^4+2\tilde{\phi} \stackrel{P}{\longrightarrow} 0$,
and hence $n\tilde{\theta}_i^4 \stackrel{P}{\longrightarrow} c''$ with $0<c''<\infty$.
Similarly to the pMOM proof, all axes corresponding to the quadratic expansion contract exactly at rate $n^{-1}$,
hence all other modes $\tilde{\theta}_i=O_p(n^{-1/4})$.
The proof for the peMOM case proceeds identically, with the only difference that term
$-2\tilde{\theta}_i^2/n$ in (\ref{eq:pimom_partial2}) changes for
$-\tilde{\theta}_i^4/(n\phi) \stackrel{P}{\longrightarrow} 0$, hence one obtains the same convergence in probability for $\tilde{\theta}_i$.

\subsection{Proof of Proposition \ref{prop:postmode}, Part (ii)}

We first state a lemma regarding the derivatives of the univariate log-MOM, eMOM and iMOM prior densities with prior dispersion $\tau=1$.
The do not prove the lemma, as it follows from straightforward algebra.
\begin{lemma}
Let $l(\theta_i,\phi)=\mbox{log}\left( \pi(\theta_i \mid \phi) \right)$.
\begin{enumerate}[(i)]
\item Let $\pi(\theta_i \mid \phi) \propto \phi^{-3/2} \theta_i^2 \mbox{exp}\{-\frac{1}{2}\theta_i^2/\phi\}$ be the MOM density, then
\begin{align}
\frac{\partial^2 l}{\partial \theta_i^2}= -\frac{2}{\theta_i^2} - \frac{1}{\phi};
\frac{\partial^2 l}{\partial \theta_i \partial \phi}= \frac{\theta_i}{\phi^2};
\frac{\partial^2 l}{\partial \phi^2}= \frac{3}{2\phi^2} - \frac{\theta_i^2}{\phi^3};
\frac{\partial^3 l}{\partial \theta_i^3}= \frac{4}{\theta_i^3}.
\nonumber
\end{align}
\item Let $\pi(\theta_i \mid \phi) \propto \mbox{exp}\{ -\phi/\theta_i^2 \} \phi^{-1/2} \mbox{exp}\{-\frac{1}{2}\theta_i^2/\phi\}$ be the eMOM density, then
\begin{align}
\frac{\partial^2 l}{\partial \theta_i^2}= -\frac{6\phi}{\theta_i^4} - \frac{1}{\phi};
\frac{\partial^2 l}{\partial \theta_i \partial \phi}= \frac{2}{\theta_i^3} + \frac{\theta_i}{\phi^2};
\frac{\partial^2 l}{\partial \phi^2}= \frac{1}{2\phi^2} - \frac{\theta_i^2}{\phi^3};
\frac{\partial^3 l}{\partial \theta_i^3}= \frac{24\phi}{\theta_i^5}.
\nonumber
\end{align}
\item Let $\pi(\theta_i \mid \phi) \propto \phi^{1/2} \theta_i^{-2} \mbox{exp}\{-\phi/ \theta_i^2 \}$ be the eMOM density, then
\begin{align}
\frac{\partial^2 l}{\partial \theta_i^2}= \frac{2}{\theta_i^2} - \frac{6\phi}{\theta_i^4};
\frac{\partial^2 l}{\partial \theta_i \partial \phi}= \frac{2}{\theta_i^3};
\frac{\partial^2 l}{\partial \phi^2}= -\frac{1}{2\phi^2};
\frac{\partial^3 l}{\partial \theta_i^3}= -\frac{4}{\theta_i^3} + \frac{24\phi}{\theta_i^5}.
\nonumber
\end{align}
\end{enumerate}
\label{lem:deriv_nlps}
\end{lemma}

\subsection{Proof of Proposition \ref{prop:postmode}, Part (ii)}

Consider Proposition \ref{prop:postmode}(ii) for general models that satisfy the conditions in \citeasnoun{walker:1969}.
For ease of notation we drop the subindex $k$ and conditioning on model $M_k$. The posterior expectation of interest is
$E(\theta_i \mid {\bf y}_n)=$
\begin{align}
\frac{\int_{}^{}\!\, \int_{}^{}\!\, \theta_i \mbox{exp}\left\{\mbox{log} (\pi(\bm{\theta} \mid \phi)) + L_n(\bm{\theta},\phi) + \mbox{log}(\pi(\phi))\right\}  d \bm{\theta} d\phi}
{\int_{}^{}\!\, \int_{}^{}\!\, \mbox{exp}\left\{\mbox{log} (\pi(\bm{\theta} \mid \phi)) + L_n(\bm{\theta},\phi) + \mbox{log}(\pi(\phi))\right\}  d \bm{\theta} d\phi}= \nonumber \\
\frac{\int_{}^{}\!\, \int_{}^{}\!\, \theta_i e^{-n h_n(\bm{\theta},\phi)} d \bm{\theta} d\phi}{\int_{}^{}\!\, \int_{}^{}\!\, e^{-n h_n(\bm{\theta},\phi)} d \bm{\theta} d\phi},
\label{eq:postmean}
\end{align}
where $L_n(\bm{\theta},\phi)$ is the log-likelihood function.
We shall use Theorem 4 in \citeasnoun{kass:1990} to obtain a Laplace approximation to (\ref{eq:postmean}) by expanding
$h_n(\bm{\theta},\phi)$ around its main posterior mode $(\tilde{\bm{\theta}},\tilde{\phi})$.
We note that when the true parameter value $\theta_i^* \neq 0$ the posterior multi-modality does not vanish even as $n \rightarrow \infty$,
but defer discussion of this point to later in the proof.
We note that Walker's conditions ensure that the model is Laplace regular and hence Theorem 4 in \citeasnoun{kass:1990} can be used.
To use the theorem we set $g(\bm{\theta},\phi)=\theta_i$, $b(\bm{\theta},\phi)=1$ and $\gamma(\bm{\theta},\phi)=\pi(\bm{\theta} \mid \phi) \pi(\phi)$
and note that $g(\bm{\theta},\phi)$ are four times $\gamma(\bm{\theta},\phi)$ and six times differentiable.
We also note that when $\pi(\bm{\theta} \mid \phi)$ is either the eMOM or iMOM prior density, it is infinitely differentiable but not analytical at $\theta_i=0$,
but $\tilde{\theta}_i$ cannot occur at 0 (the prior density is 0) and hence we may ignore this set with 0 Lebesgue measure.
Direct application of Theorem 4 in \citeasnoun{kass:1990} gives

\begin{align}
E(\theta_i \mid {\bf y}_n)= \tilde{\theta}_i + \frac{1}{n} \sum_{j=1}^{|k|+1} h_{ij} \left( -\frac{1}{2} \sum_{r,s}^{} h^{rs} h_{rsj} \right) + O \left( n^{-2} \right)
\label{eq:postmean_approx}
\end{align}
where $|k|=\mbox{dim}(\bm{\theta})$, $h_{ij}$ denotes the $(i,j)$ element of the Hessian of $h_n(\bm{\theta},\phi)$ evaluated at $(\tilde{\bm{\theta}},\tilde{\phi})$,
$h^{ij}$ that of the inverse Hessian and $h_{rsj}$ are third derivatives.
That is,
\begin{align}
h_{ii}= \frac{1}{n} \frac{\partial^2}{\partial \theta_i^2} L_n(\bm{\theta},\phi) + \frac{1}{n} \frac{\partial^2}{\partial \theta_i^2}  \mbox{log}( \pi(\theta_i \mid \phi) ), \nonumber \\
h_{ij}= \frac{1}{n} \frac{\partial^2}{\partial \theta_i \partial \theta_j} L_n(\bm{\theta},\phi), \nonumber \\
h_{i,|k|+1}= \frac{1}{n} 
\frac{\partial^2}{\partial \theta_i \partial \phi} L_n(\bm{\theta},\phi) + \frac{1}{n} \frac{\partial^2}{\partial \theta_i \partial \phi}
\mbox{log}( \pi(\theta_i \mid \phi) ).
\label{eq:hessian_elem} 
\end{align}
From the Normal approximation to the likelihood we obtain that $h_{rsj} \stackrel{P}{\longrightarrow} 0$ unless $r=s=j$,
in which case $h_{jjj} \stackrel{P}{\longrightarrow} \frac{\partial^3}{\partial \theta_j^3} \mbox{log}(\pi(\tilde{\theta}_j\mid \phi))$.
Hence,
\begin{align}
E(\theta_i \mid {\bf y}_n) \stackrel{P}{\longrightarrow} \tilde{\theta}_i - \frac{1}{2n} \left(\sum_{j=1}^{|k|+1} h_{ij} h^{jj} h_{jjj} \right)
%E(\theta_i \mid {\bf y}_n) \stackrel{P}{\longrightarrow} \tilde{\theta}_i - \frac{1}{2n} \left(h_{ii}h^{ii}h_{iii} + \sum_{j\neq i, j=1}^{|k|+1} h_{ij} h^{jj} h_{jjj} \right)
\label{eq:postmean_approx_asymp}
\end{align}
Walker's conditions ensure that $h_{ij}$ for $i\neq j$ converge in probability to a finite $J_{ij}$.
Regarding $h_{ii}$, the first term converges to $h_{ii}$ whereas the second term is $O_p(n^{-1})$ when $\theta_i^*\neq 0$
and $O_p(1)$ when $\theta_i^*=0$ for either the MOM, eMOM or iMOM prior
(Proposition \ref{prop:postmode}(i) and Lemma \ref{lem:deriv_nlps}), hence $h_{ii}=O_p(1)$.
This in turn implies that the Hessian converges in probability to $J$ plus diagonal terms that either converge to 0 or are $O_p(1)$,
and hence the elements in its inverse $h^{jj}=O_p(1)$.
Finally consider $h_{jjj}$. 
From Lemma \ref{lem:deriv_nlps} when $\theta_j^* \neq 0$ we obtain $h_{jjj}=O_p(1)$ for either the MOM, eMOM or iMOM priors.
When $\theta_j^*=0$ for the MOM prior
$h_{jjj} \stackrel{P}{\longrightarrow} 4 n^{-1} \tilde{\theta}_i^{-3}=n^{-1} O_p(n^{3/2})=O_p(n^{1/2})$
for $j=1,\ldots,|k|$ and $h_{jjj} \stackrel{P}{\longrightarrow} O_p(1)$ for $j=|k|+1$.
For the eMOM and iMOM priors
$h_{jjj} \stackrel{P}{\longrightarrow} 24 n^{-1} \tilde{\phi}\tilde{\theta}_i^{-5}=n^{-1} O_p(n^{5/4})=O_p(n^{1/4})$ for $j=1,\ldots,|k|$
and again $h_{jjj} \stackrel{P}{\longrightarrow} O_p(1)$ for $j=|k|+1$.
Therefore, from (\ref{eq:postmean_approx}) we obtain that
$E(\theta_i \mid {\bf y}_n) \stackrel{P}{\longrightarrow} \tilde{\theta}_i + O_p(n^{-1})$ if $\theta_j^* \neq 0$ for $j=1,\ldots,|k|$
and $E(\theta_i \mid {\bf y}_n) \stackrel{P}{\longrightarrow} \tilde{\theta}_i + O_p(n^{-1/2})$ if $\theta_j^* \neq 0$ for any $j=1,\ldots,|k|$.
In particular, in cases of parameter orthogonality where $h_{ij}=0$ for all $i \neq j$ then the difference between the posterior
mean and posterior mode of $\theta_i$ is $O_p(n^{-1})$ whenever $\theta_i^* \neq 0$.
To conclude the proof, we recall that the posterior is multi-modal and hence 
approximate $E(\theta_i \mid {\bf y}_n)$ by adding (\ref{eq:postmean_approx}) across the $2^{|k|}$ modes.
Proposition \ref{prop:postmode} that for such modes
$\tilde{\theta}_i=O_p(n^{-1/2})$ for pMOM and $\tilde{\theta}_i=O_p(n^{-1/4})$ for peMOM and piMOM, hence
$E(\theta_i \mid {\bf y}_n)= \hat{\theta}_i + O_p(n^{-1/2})= \theta_i^* + O_p(n^{-1/2})$ for MOM
and $E(\theta_i \mid {\bf y}_n)= \hat{\theta}_i + O_p(n^{-1/4})=\hat{\theta}_i^* + O_p(n^{-1/4})$ for eMOM or iMOM.

%If $\theta_i^* \neq 0$ then for any mode such that $\tilde{\theta}_i-\hat{\theta}_i$
%does not converge to zero we have $n\tilde{\theta}_i \stackrel{P}{\longrightarrow} c>0$ (MOM)
%and $n\tilde{\theta}_i^3 \stackrel{P}{\longrightarrow} c>0$ (eMOM, iMOM).
%See the proof of Proposition~\ref{prop:postmode}(i) for details.
%Hence for MOM such modes satisfy $\tilde{\theta}_i=O_p(n^{-1})$
%so that $E(\theta_i \mid {\bf y}_n)= \hat{\theta}_i + O_p(n^{-1})$,
%whereas for eMOM and iMOM $\tilde{\theta}_i=O_p(n^{-1/3})$ and $E(\theta_i \mid {\bf y}_n)= \hat{\theta}_i + O_p(n^{-1/3})$.

\subsection{Proof of Proposition \ref{prop:postmode}, Part (iii)}

We consider linear models of growing dimensionality, again dropping the model subindex $k$ for ease of notation.
Although we assume that $X_n'X_n$ is a diagonal matrix, we state part of the argument for general $X_n'X_n$
(subject to the eigenvalue conditions in Proposition \ref{prop:parsimony_lm}) and make explicit where the orthogonality assumption is needed.
As argued during the proof of Proposition \ref{prop:postmode}(i), the rates for posterior modes remain valid for linear models
with such bounded eigenvalues.
Regarding the posterior mean, the assumed conditions on the eigenvalues of $X_n'X_n$ 
guarantee Laplace regularity \cite{kass:1990} and hence the expansion (\ref{eq:postmean_approx}) remains valid,
where now $h_n(\bm{\theta},\phi)=$
\begin{align}
 \frac{1}{2}\mbox{log}(\phi) + \frac{1}{2\phi} (\bm{\theta}-\hat{\bm{\theta}})' \frac{X_n'X_n}{n} (\bm{\theta}-\hat{\bm{\theta}})
-\frac{1}{n} \sum_{i=1}^{|k|} \mbox{log}(\pi(\theta_i\mid \phi)) - \frac{1}{n} \mbox{log}(\pi(\phi))
\label{eq:hn_linearmodel}
\end{align}
Therefore $h_{ij}$ is given by the $(i,j)$ element in $\frac{X'X}{n\tilde{\phi}}$ for $i=1,\ldots,|k|$, $i\neq j$, which is $O_p(1)$.
For $h_{ii}$ we add $\frac{1}{n}\frac{\partial^2}{\partial \theta_i^2} \mbox{log}(\pi(\theta_i \mid \phi)$,
which from Lemma \ref{lem:deriv_nlps} and Proposition \ref{prop:postmode}
is $O_p(n^{-1})$ for $\tilde{\theta}_i \stackrel{P}{\longrightarrow} \theta_i^*$ and $O_p(1)$ for $\tilde{\theta}_i \stackrel{P}{\longrightarrow} 0$
(for pMOM, peMOM and piMOM), hence $h_{ii}=O_p(1)$.
The elements $h_{1,|k|+1},\ldots,h_{|k|,|k|+1}$ are given by the vector
\begin{align}
-\frac{1}{\tilde{\phi}^2} \frac{X_n'X_n}{n} (\tilde{\bm{\theta}}-\hat{\bm{\theta}}) - \frac{1}{n} g(\tilde{\bm{\theta}},\tilde{\phi}),
\label{eq:hn_crossderiv}
\end{align}
where $g(\tilde{\bm{\theta}},\tilde{\phi})$ contains
$\frac{\partial^2}{\partial \theta_i \partial \phi}\mbox{log}(\pi(\theta_i \mid \phi)$ for $i=1,\ldots,|k|$.
Given that the eigenvalues of $X_n'X_n/n$ are bounded
the first term in (\ref{eq:hn_crossderiv}) converges in probability to 0 for the main mode and is $O_p(1)$ for all other modes.
From Lemma \ref{lem:deriv_nlps} and Proposition \ref{prop:postmode}(i) it is straightforward to see that
$n^{-1} g(\tilde{\bm{\theta}},\tilde{\phi}) \stackrel{P}{\longrightarrow} {\bf 0}$, hence $h_{i,|k|+1}=O_p(1)$ for $i=1,\ldots,|k|$.
Similarly, $h_{|k|+1,|k|+1}=$
\begin{align}
-\frac{1}{2\phi^2} + \frac{1}{\phi^3}(\tilde{\bm{\theta}}-\hat{\bm{\theta}})' \frac{X_n'X_n}{n} (\tilde{\bm{\theta}}-\hat{\bm{\theta}}) 
- \frac{1}{n} \sum_{i=1}^{|k|} \frac{\partial^2}{\partial \phi^2} \mbox{log}(\pi(\theta_i,\phi)) \nonumber \\
+ \frac{1}{n} \frac{\partial^2}{\partial \phi^2} \mbox{log}(\pi(\phi)),
\label{eq:hn_doublephideriv}
\end{align}
which from Proposition \ref{prop:postmode}(i) and Lemma \ref{lem:deriv_nlps} is $O_p(1)$.

Regarding the elements in the inverse Hessian $h^{ij}$,
the Hessian is positive definite with $h_{ij}=O_p(1)$ 
and hence $h^{ij}=O_p(1)$ for $i,j=1,\ldots,|k|+1$.
%$h^{ij} \stackrel{P}{\longrightarrow} 0$ for $i=1,\ldots,|k|$, $j=|k|+1$ and $h^{|k|+1,|k|+1}=O_p(1)$.
%Thus $h^{ij}=O_p(1)$ for all $i,j$.

Finally we obtain third derivatives $h_{rsj}$.
Because $h_{rs}$ is given by the corresponding element $X_n'X_n/(n\tilde{\phi})$, $h_{rsj}=0$ for $r,s,j \in \{1,\ldots,|k|\}$.
For $r=s=j$, for the main mode $h_{jjj}=O_p(1)$ under either a pMOM, peMOM or piMOM prior (Lemma \ref{lem:deriv_nlps}),
whereas for other modes $h_{jjj}=O_p(n^{1/2})$ under a pMOM or $O_p(n^{1/4})$ under a peMOM or piMOM priors (Proposition \ref{prop:postmode}(i)).
From (\ref{eq:postmean_approx}), the contribution to $E(\theta_i \mid {\bf y}_n)$ from each mode is
\begin{align}
\tilde{\theta}_i - \frac{1}{2n} \sum_{j=1}^{|k|+1} h_{ij}h^{jj}h_{jjj}
\label{eq:postmean_lm}
\end{align}
plus a lower order term.

Consider now that $X_n'X_n$ is orthogonal. In that case $h_{ij}=0$ for $i \neq j$ and the two values
$\tilde{\theta}_i^{(1)},\tilde{\theta}_i^{(2)}$ maximizing the posterior are independent of $\theta_j$ for $j \neq i$.
Therefore under a pMOM prior
\begin{align}
E(\theta_i \mid {\bf y}) = \tilde{\theta}_i^{(1)} + \tilde{\theta}_i^{(2)} - \frac{1}{2n} O_p(n^{1/2})=\theta_i^* + O_p(n^{-1/2})
\label{eq:postmean_pmom_ortho}
\end{align}
whereas
\begin{align}
E(\theta_i \mid {\bf y}) = \tilde{\theta}_i^{(1)} + \tilde{\theta}_i^{(2)} - \frac{1}{2n} O_p(n^{1/4})=\theta_i^* + O_p(n^{-1/4})
\label{eq:postmean_pemom_ortho}
\end{align}
under either a peMOM or piMOM prior, which concludes the proof.

\subsection{Proof of Proposition \ref{prop:bma_postmean}, Part (i)}

Consider models $M_k$ for $k=1,\ldots,K$, all satisfying the conditions in \citeasnoun{walker:1969}.
Let $M_t$ be the true model and let $k$ be such that $M_t \subset M_k$.
Consider first the pMOM prior.
The marginal likelihood $m_t({\bf y}_n)$ under $M_t$ can be approximated by a Laplace expansion around each posterior mode
$(\tbtheta_t^{(m)},\tphi_t^{(m)})$ for $m=1,\ldots,2^{|t|}$, so that $m_t({\bf y}_n) \approx$
\begin{align}
e^{L_n(\tbtheta_t^{(m)},\tphi_t^{(m)})} \prod_{i}^{} \frac{(\ttheta_{ti}^{(m)})^2}{\tau \tphi^{(m)}}
N(\tbtheta_t^{(m)}; {\bf 0}, \tau \tphi_t^{(m)} I) \pi(\tphi_t^{(m)}) \left| H(\tbtheta_t^{(m)},\tphi_t^{(m)}) \right|^{-1/2},
\label{eq:mlhood_pmom}
\end{align}
where $L_n(\cdot)$ is the log-likelihood and $H(\tbtheta_t^{(m)},\tphi_t^{(m)})$ the Hessian 
of the log-likelihood plus the log-prior density evaluated at $(\tbtheta_t^{(m)},\tphi_t^{(m)})$.
Expressions for the elements in $H(\tbtheta_t^{(m)},\tphi_t^{(m)})$ are given in the proof of Proposition \ref{prop:postmode} for pMOM, peMOM and piMOM priors.

Without loss of generality denote by $(\tbtheta_t^{(1)},\tphi_t^{(1)})$ the mode located in the same quadrant as the MLE $(\hbtheta,\hat{\phi})$.
As seen in Proposition \ref{prop:postmode}, under Walker's conditions
 $(\tbtheta_t^{(1)},\tphi_t^{(1)}) \stackrel{P}{\longrightarrow} (\btheta_t^*,\phi_t^*)$ and 
$n^{-1} H(\tbtheta_t^{(m)},\tphi_t^{(m)}) \stackrel{P}{\longrightarrow} J$ for a positive-definite $J$,
hence (\ref{eq:mlhood_pmom}) converges in probability to
\begin{align}
e^{L_n(\tbtheta_t^{(1)},\tphi_t^{(1)})} c_1 n^{-t/2} c_2,
\label{eq:mlhood_pmom_asymp}
\end{align}
where $c_1,c_2 > 0$.
For modes in any other quadrant
$e^{L_n(\tbtheta_t^{(m)},\tphi_t^{(m)}) - L_n(\tbtheta_t^{(1)},\tphi_t^{(1)})} \stackrel{P}{\longrightarrow} e^{-nc_3}$,
where $c_3>0$ is the Kullback-Leibler divergence between the data-generating model $f_k(\bm{\theta}_t^*,\phi_t^*)$
and that where some elements in $\bm{\theta}_t$ are set to 0 (which is positive by assumption).
Further, in such quadrants
$\prod_{i}^{} \frac{(\ttheta_{ti}^{(m)})^2}{\tau \tphi^{(m)}}=O_p(n^{-|t|})$ so that the sum of (\ref{eq:mlhood_pmom}) across all modes $m=1,\ldots,2^{|t|}$
gives that the marginal likelihood $m_t({\bf y}_n) \approx$
\begin{align}
e^{L_n(\tbtheta_t^{(1)},\tphi_t^{(1)})} &\left( Z_1 Z_2 n^{-t/2} + \sum_{m}^{} e^{L_n(\tbtheta_t^{(m)},\tphi_t^{(m)})-L_n(\tbtheta_t^{(1)},\tphi_t^{(1)})} O_p(n^{-|t|}) Z_3 \right) \nonumber \\
&\stackrel{P}{\longrightarrow} n^{-t/2} e^{L_n(\tbtheta_t^{(1)},\tphi_t^{(1)})} Z_4,
\label{eq:mlhood_pmom_final}
\end{align}
where $Z_j \stackrel{P}{\longrightarrow} c_j>0$ for $j=1,\ldots,4$.

Now consider $M_k$ such that $M_t \subset M_k$. Denote by $\bm{\theta}_{k1}$ the subset of $\bm{\theta}_k$ such that $\theta_{ki}^*=0$
and $\bm{\theta}_{k2}^*$ that for $\theta_{ki}^* \neq 0$, where $\btheta_k^*$ minimizes Kullback-Leibler divergence to the data-generating model,
which under our assumptions is $M_t$ and hence $\mbox{dim}(\btheta_{k1})=|k|-|t|$.
Following the same argument as for $M_t$, it suffices to focus on modes for which $\tbtheta_{k2}$ lies in the same quadrant as $\btheta_{k2}^*$.
Adding up the Laplace approximations across all $2^{|k|-|t|}$ such modes delivers the Bayes factor
$\mbox{BF}_{kt}=\frac{m_k({\bf y}_n)}{m_t({\bf y}_n)} \stackrel{P}{\longrightarrow}$
\begin{align}
\sum_{m}^{} \frac{e^{L_n(\tbtheta_k^{(m)},\tphi_k^{(m)})}}{e^{L_n(\tbtheta_t^{(1)},\tphi_t^{(1)})}}
\frac{\prod_{i}^{} \frac{(\ttheta_{ki}^{(m)})^2}{\tau \tphi_k^{(m)}}}{\prod_{i}^{} \frac{(\ttheta_{ti}^{(1)})^2}{\tau \tphi_t^{(1)}}}
\frac{\pi(\tphi_k^{(m)})}{\pi(\tphi_t^{(1)})} \frac{n^{-|k|/2}}{n^{-|t|/2}}
\frac{\left| n^{-1} H(\tbtheta_k^{(m)},\tphi_k^{(m)})  \right|}{\left| n^{-1} H(\tbtheta_t^{(1)},\tphi_t^{(1)}) \right|},
\label{eq:bf_pmom}
\end{align}
where the first term is $O_p(1)$, the second term converges in probability to $n^{-(|k|-|t|)} Z_5$ for some random variable $Z_5=O_p(1)$,
the third and fourth terms converge in probability to a positive constant ($\pi(\phi)$ is bounded by assumption).
Therefore each summand in (\ref{eq:bf_pmom}) is $O_p(n^{-\frac{3}{2}(|k|-|t|)})$, and given that we are adding up a finite number of terms
$\mbox{BF}_{kt}=O_p(n^{-\frac{3}{2}(|k|-|t|)})$.
Next consider $k$ such that $M_t \not\subset M_k$. By assumption, the minimum Kullback-Leibler divergence $\mbox{KL}(M_t,M_k)$
between $f_t(\bm{\theta}_t^*,\phi_t^*)$
and any $f_k(\bm{\theta}_k^*,\phi_k^*)$ with $(\bm{\theta}_k,\phi_k) \in (\Theta_k,\Phi)$ is strictly positive. Hence by the law of large numbers
$e^{L_n(\tbtheta_k^{(m)},\tphi_k^{(m)}) - L_n(\tbtheta_t^{(1)},\tphi_t^{(1)})} \stackrel{a.s.}{\longrightarrow} e^{-n \mbox{KL}(M_t,M_k)}$
and $\mbox{BF}_{kt}=O_p(e^{-n})$.

The proof for the peMOM and piMOM are largely analogous.
The marginal likelihood for $M_t$ is $m_t({\bf y}_n) \approx$
\begin{align}
e^{L_n(\tbtheta_t^{(1)},\tphi_t^{(1)})} \pi(\tphi_t^{(1)}) \left| H(\tbtheta_t^{(1)},\tphi_t^{(1)}) \right|^{-1/2} Z_1
\prod_{i}^{} e^{-\tau\tphi^{(1)} / (\ttheta_{ti}^{(1)})^2} \nonumber \\
\stackrel{P}{\longrightarrow} n^{-|t|/2} e^{L_n(\tbtheta_t^{(1)},\tphi_t^{(1)})} Z_2
\label{eq:mlhood_pemom}
\end{align}
where $Z_1=O_p(1)$ for the peMOM under any model, whereas for the piMOM $Z_1=O_p(1)$ under $M_t$
and $Z_1=o_p(1)$ under any other $M_k$,
and consequently $Z_2=O_p(1)$.
Consider $k$ such that $M_t \subset M_k$, then from Proposition \ref{prop:postmode}(i)
for all modes with $\tbtheta_{k2}$ in the same quadrant as $\btheta_{k2}^*$ we have
$\prod_{i}^{} \mbox{exp}\{-\sqrt{n} \tau\tphi^{(1)} / (n^{1/4}\ttheta_{ti}^{(1)})^2 \}=\prod_{i}^{} \mbox{exp}\{-\sqrt{n} Z_{3i} \}= e^{-\sqrt{n} Z_4}$,
where $Z_4=O_p(1)$.
Thus the Bayes factor
\begin{align}
\mbox{BF}_{kt} \stackrel{P}{\longrightarrow} 
\sum_{m}^{} \frac{e^{L_n(\tbtheta_k^{(m)},\tphi_k^{(m)})}}{e^{L_n(\tbtheta_t^{(1)},\tphi_t^{(1)})}} \frac{e^{-\sqrt{n}Z_4} n^{-|k|/2}}{n^{-|t|/2} Z_2}=O_p(e^{\sqrt{n}}).
\label{eq:bf_pemom}
\end{align}
The proof for the $M_t \not\subset M_k$ case proceeds in the same manner as for the pMOM.

\subsection{Proof of Proposition \ref{prop:bma_postmean}, Part (ii)}

We start by using the Bayes factor rates proven in Part (i) to derive rates for posterior model probabilities.
Consider a model $M_k$ such that $M_t \subset M_k$ and note that $P(M_k \mid {\bf y}_n) < (1 + \mbox{BF}_{tk} P(M_t)/P(M_k))^{-1}$.
Under a pMOM prior
\begin{align}
P(M_k \mid {\bf y}_n)<
\frac{1}{1+O_p(1) n^{\frac{3}{2}(|k|-|t|)} \frac{P(M_t)}{P(M_k)}}= \nonumber \\
\frac{n^{-\frac{3}{2}(|k|-|t|)} \frac{P(M_t)}{P(M_k)}}{n^{-\frac{3}{2}(|k|-|t|)} \frac{P(M_t)}{P(M_k)}+O_p(1)}=
n^{-\frac{3}{2}(|k|-|t|)} \frac{P(M_k)}{P(M_t)} O_p(1),
\label{eq:posprob_pmom_ineq}
\end{align}
where the last equality follow from the assumption that $P(M_k)/P(M_t)=o(n^{\frac{3}{2}(|k|-|t|)})$ and hence the denominator is $O_p(1)$.
The same argument applies under a peMOM or piMOM prior, where now $\mbox{BF}_{kt}=e^{-\sqrt{n}}$ and hence
$P(M_k \mid {\bf y}_n)< e^{-\sqrt{n}} \frac{P(M_k)}{P(M_t)} O_p(1)$.
Finally, for models $M_k$ such that $M_t \not \subset M_k$, from Proposition \ref{prop:bma_postmean}(i)
$P(M_k \mid {\bf y}_n)< \left(1+e^{nO_p(1)}P(M_t)/P(M_k)\right)^{-1}=e^{-nO_p(1)} P(M_k)/P(M_t)$.

The BMA posterior mean is $E(\theta_i\mid {\bf y}_n)=$
\begin{align}
E(\theta_i \mid M_t,{\bf y}_n) P(M_t \mid {\bf y}_n) + \sum_{k:M_t\subset M_k}^{} E(\theta_i \mid M_k, {\bf y}_n) P(M_k \mid {\bf y}_n) + \nonumber \\
\sum_{k: M_t \not\subset M_k}^{} E(\theta_i \mid M_k, {\bf y}_n) P(M_k \mid {\bf y}_n).
\label{eq:bma_expanded}
\end{align}
Suppose first that $\theta_i^* \neq 0$.
From Proposition \ref{prop:postmode}(ii), $E(\theta_i \mid M_t, {\bf y}_n)= \htheta_i + O_p(n^{-1})$ for pMOM, peMOM and piMOM, where $\htheta_i$ is the MLE.
Also, $E(\theta_i \mid M_k, {\bf y}_n)$ in the second term of (\ref{eq:bma_expanded})
is $O_p(1)$ and $P(M_k \mid {\bf y}_n)$ is either $O_p(n^{-\frac{3}{2}(|k|-|t|)})$ (pMOM) or $O_p(e^{-\sqrt{n}})$ (peMOM, piMOM).
Further, $P(M_k)/P(M_t)=o(n^{|k|-|t|})$ by assumption and hence the whole second term in (\ref{eq:bma_expanded}) is $O_p(n^{-1})$.
Regarding the third term in (\ref{eq:bma_expanded}), $E(\theta_i \mid M_k, {\bf y}_n)=O_p(1)$ and
$P(M_k \mid {\bf y}_n)=O_p(e^{-n})$.
Summarizing, when $\theta_i^* \neq 0$ for the pMOM we have that $E(\theta_i \mid {\bf y}_n)=$
\begin{align}
\left(\htheta_{ti} + O_p(n^{-1}) \right)\left(1+n^{-2} O_p(1)\right)^{-1} + O_p(n^{-2})=\htheta_{ti} + O_p(n^{-1})
\label{eq:mean_mom_nonzero}
\end{align} 
and for the peMOM or piMOM $E(\theta_i \mid {\bf y}_n)=$
\begin{align}
 \left(\htheta_{ti} + O_p(n^{-1}) \right)\left(1+e^{-\sqrt{n}O_p(1)} \right)^{-1} + O_p(e^{-\sqrt{n}}) =\htheta_{ti} + O_p(n^{-1}).
\label{eq:mean_emom_nonzero}
\end{align}

Now consider the case $\theta_i^*=0$. Obviously, $M_t$ only includes non-zero coefficients and hence $E(\theta_i \mid M_t, {\bf y}_n)=0$.
In the second term of (\ref{eq:bma_expanded}), from Proposition \ref{prop:postmode}(ii) we have 
that $E(\theta_i \mid M_k, {\bf y}_n)$ is $O_p(n^{-1/2})$ for pMOM and $O_p(n^{-1/4})$ for peMOM and piMOM.
Thus the whole second term is $O_p(n^{-2}) \pi_{|t|+1}/P(M_t)$ for pMOM
and $O_p(e^{-\sqrt{n}}) \pi_{|t|+1}/P(M_t)$ for peMOM and piMOM, where $\pi_{|t|+1}= \mbox{max}_{k: |k|=|t|+1} P(M_k)$ for $M_t \subset M_k$.
As in the $\theta_i^*\neq 0$ case, the third term is $O_p(e^{-n})$.
Summarizing, when $\theta_i^*=0$ we obtain $E(\theta_i \mid {\bf y}_n)=$
\begin{align}
O_p(n^{-2}) \frac{\pi_{|t|+1}}{P(M_t)}
\label{eq:mean_mom_zero}
\end{align} 
and for the peMOM or piMOM $E(\theta_i \mid {\bf y}_n)=$
\begin{align}
O_p(e^{-\sqrt{n}}) \frac{\pi_{|t|+1}}{P(M_t)},
\label{eq:mean_emom_zero}
\end{align}
as desired.

\subsection{Proof of Proposition \ref{prop:bma_postmean}, Part (iii)}

We adjust the notation of the previous sections slightly to ease the upcoming exposition.
Let $\theta_i$ for $i=1,\ldots,p$ (where $p<n$) be the coefficient corresponding to variable $i$,
and let $\delta_i=\mbox{I}(\theta_i \neq 0)$ be variable inclusion indicators.
We aim to characterize $E(\theta_i \mid {\bf y}_n)= E(\theta_i \mid {\bf y}_n, \delta_i=1) P(\delta_i=1 \mid {\bf y}_n)$.
We first derive $P(\delta_i \mid {\bf y}_n)$. Let $\bm{\delta}=(\delta_1,\ldots,\delta_p)$ and
$\bm{\delta}_{-i}$ be the result from removing $\delta_i$ from $\bm{\delta}$,
and note that 
$P(\delta_i=1 \mid {\bf y}_n)= \sum_{\bm{\delta}_{-i}}^{} P(\delta_i=1 \mid \bm{\delta}_{-i}, {\bf y}_n) P(\bm{\delta}_{-i} \mid {\bf y}_n)$.
Denote by $\pi_{-i}=P(\delta_i=1 \mid \bm{\delta}_{-i})$,
because $X_n'X_n$ is orthogonal the likelihood factors across $i=1,\ldots,p$, and given that the pMOM, peMOM and piMOM priors also factors
straightforward algebra shows that $P(\delta_i=1 \mid \bm{\delta}_{-i}, {\bf y}_n)=$
\begin{align}
\frac{ \int_{}^{}\!\, d_i(\theta_i,\phi) N(\theta_i; m_i, \phi v_i) d \theta_i  m_{-i}({\bf y}_n) \pi_{-i}}
{\int_{}^{}\!\, d_i(\theta_i,\phi) N(\theta_i; m_i, \phi v_i) d \theta_i m_{-i}({\bf y}_n) \pi_{-i} + N(0; m_i, \phi v_i) m_{-i}({\bf y}_n) (1-\pi_{-i})}= \nonumber \\
\frac{ \int_{}^{}\!\, d_i(\theta_i,\phi) N(\theta_i; m_i, \phi v_i) d \theta_i \pi_{-i}}
{\int_{}^{}\!\, d_i(\theta_i,\phi) N(\theta_i; m_i, \phi v_i) d \theta_i \pi_{-i} + N(0; m_i, \phi v_i) (1-\pi_{-i})}
\label{eq:condpp_ortho}
\end{align}
where $m_{-i}({\bf y}_n)=$
\begin{align}
\prod_{j\neq i, \delta_j=1}^{} \int_{}^{}\!\, d_j(\theta_j,\phi) N(\theta_j; m_j, \phi v_j) d\theta_j \prod_{j \neq i, \delta_j=0}^{}m_{j0}({\bf y}_n).
\end{align}
For the pMOM prior $d_i(\theta_i,\phi)=\theta_i^2/\phi\tau$, $v_i=\tau/(n\tau+1)$ and $m_i=v_i \sum_{l=1}^{n} x_{il} y_l$,
for the peMOM prior $d_i(\theta_i,\phi)=e^{-\tau\phi/\theta_i^2}$ and again $v_i=\tau/(n\tau+1)$, $m_i=v_i \sum_{l=1}^{n} x_{il} y_l$,
and for the piMOM prior $d_i(\theta_i,\phi)= \sqrt{\tau\phi}\theta_i^{-2} e^{-\tau\phi/\theta_i^2}$, $v_i=n^{-1}$, $m_i=v_i \sum_{l=1}^{n} x_{il} y_l$.
The assumption that $\delta_1,\ldots,\delta_p$ are exchangeable a priori is equivalent to stating that 
$\pi_{-i}=P(\delta_i=1 \mid \bm{\delta}_{-i}, \omega)= P(\delta_i \mid \omega)=\pi_{\omega}$ for a certain hyper-parameter $\omega$,
and hence
$P(\delta_i=1 \mid {\bf y}_n, \omega)= \sum_{\bm{\delta}_{-i}}^{} P(\delta_i \mid \bm{\delta}_{-i}, {\bf y}_n,\omega) P(\bm{\delta}_{-i} \mid {\bf y}_n,\omega)=$
\begin{align}
\frac{ \int_{}^{}\!\, d_i(\theta_i,\phi) N(\theta_i; m_i, \phi v_i) d \theta_i \pi_{\omega}}
{\int_{}^{}\!\, d_i(\theta_i,\phi) N(\theta_i; m_i, \phi v_i) d \theta_i \pi_{\omega} + N(0; m_i, \phi v_i) (1-\pi_{\omega})}.
\label{eq:margpp_ortho_cond}
\end{align}
Denoting by $\pi(\omega)$ the prior density of $\omega$,
\begin{align}
P(\delta_i=1 \mid {\bf y}_n) \propto \int_{}^{}\!\, d_i(\theta_i,\phi) N(\theta_i; m_i, \phi v_i) d \theta_i \int_{}^{}\!\, \pi_{\omega} \pi(\omega) d\omega \nonumber \\
= \int_{}^{}\!\, d_i(\theta_i,\phi) N(\theta_i; m_i, \phi v_i) d \theta_i P(\delta_i=1)
\label{eq:margpp_ortho}
\end{align}
and hence $P(\delta_i=1 \mid {\bf y}_n)=$
\begin{align}
\frac{\int_{}^{}\!\, d_i(\theta_i,\phi) N(\theta_i; m_i, \phi v_i) d \theta_i P(\delta_i=1)}
{\int_{}^{}\!\, d_i(\theta_i,\phi) N(\theta_i; m_i, \phi v_i) d \theta_i P(\delta_i=1) + N(0; m_i, \phi v_i) P(\delta_i=0)}.
\label{eq:margpp_ortho_final}
\end{align}

Following the same argument as in Proposition \ref{prop:bma_postmean}(ii),
if $\theta_i^* \neq 0$ then $P(\delta_i=1 \mid {\bf y}_n)= \left(1-e^{-nO_p(1)}P(\delta_i=0)/P(\delta_i=1)\right)$
under either a pMOM, peMOM or piMOM prior.
If $\theta_i^* =0$ then
$P(\delta_i=1 \mid {\bf y}_n)=n^{-\frac{3}{2}(|k|-|t|)} P(\delta_i=1)/P(\delta_i=0)$ for pMOM
and $P(\delta_i=1 \mid {\bf y}_n)=e^{-\sqrt{n}O_p(1)} P(\delta_i=1)/P(\delta_i=0)$ for peMOM and piMOM.

We now characterize $E(\theta_i \mid \delta_i=1, {\bf y}_n, \phi)$.
Again, because of orthogonality this posterior mean is the same under any model with $\delta_i=1$, giving
\begin{align}
E(\theta_i \mid \delta_i=1, {\bf y}_n, \phi)=\frac{ \int_{}^{}\!\, \theta_i d_i(\theta_i,\phi) N(\theta_i; m_i, \phi v_i) d \theta_i}
{\int_{}^{}\!\, d_i(\theta_i,\phi) N(\theta_i; m_i, \phi v_i) d \theta_i},
\label{eq:modelmean_ortho}
\end{align}
As before for the pMOM prior $d_i(\theta_i,\phi)=\theta_i^2/(\phi\tau)$ and hence by using Normal moments of up to order 3 (\ref{eq:modelmean_ortho})
becomes $m_i\left(1+ \frac{2\phi v_i}{m_i^2+\phi v_i}\right)$, where $v_i=\tau/(n\tau +1)$ and $m_i=v_i \sum_{i=1}^{n} x_{ji} y_i$.
For the peMOM and piMOM, using a Laplace approximation \cite{kass:1990} around the two modes as in Proposition \ref{prop:postmode}
gives that if $\theta_i^*\neq 0$ then $E(\theta_i \mid \delta_i, {\bf y}_n, \phi)= \htheta_i + O_p(n^{-1})$ (where $\htheta_i$ is the MLE)
and if $\theta_i^*=0$ then $E(\theta_i \mid \delta_i, {\bf y}_n, \phi)= O_p(n^{-1/4})$.

Combining the rates derived above for $P(\delta_i=1 \mid {\bf y}_n, \phi)$ and $E(\theta_i \mid \delta_i=1, {\bf y}_n, \phi)$,
if $\theta_i^* \neq 0$ then
$E(\theta_i \mid {\bf y}_n,\phi)= \htheta_i + O_p(n^{-1})$ for pMOM, peMOM and piMOM,
whereas if $\theta_i^*=0$ then
$E(\theta_i \mid {\bf y}_n,\phi)= O_p(n^{-2}) P(\delta_i=1)/P(\delta_i=0)$ for pMOM
and $E(\theta_i \mid {\bf y}_n,\phi)= e^{-\sqrt{n}O_p(1)} P(\delta_i=1)/P(\delta_i=0)$ for peMOM and piMOM.

\subsection{Proof of Proposition \ref{prop1}}
\label{proof_prop1}

The goal is to show that for all $\epsilon>0$ there exists $\eta>0$ such that
$d(\bm{\theta}) < \eta$ implies $\pi(\bm{\theta}) < \epsilon$.
By construction, the conditional prior density 
is $\pi(\bm{\theta} \mid \lambda)= \pi^L(\bm{\theta}) \mbox{I}(d(\bm{\theta})>\lambda) / h(\lambda)$,
where 
$h(\lambda)= P_u(d(\bm{\theta}) > \lambda)= \int_{}^{}\!\, \pi^L(\bm{\theta}) \mbox{I}(d(\bm{\theta}) > \lambda) d\bm{\theta} $.
Let $\bm{\theta}$ be a value such that $d(\bm{\theta}) < \eta$, and express the
prior density as
\begin{align}
\pi(\bm{\theta}) = \int_{}^{}\!\, \pi(\bm{\theta} \mid \lambda) \pi(\lambda) d\lambda = \nonumber \\
\int_{\lambda \leq \eta}^{}\!\, \frac{\pi^L(\bm{\theta}) \mbox{I}(d(\bm{\theta})>\lambda)}{h(\lambda)} \pi(\lambda) d\lambda + 
\int_{\lambda > \eta}^{}\!\, \frac{\pi^L(\bm{\theta}) \mbox{I}(d(\bm{\theta})>\lambda)}{h(\lambda)} \pi(\lambda) d\lambda
\label{eq:expand_prior}
\end{align}
The second term in (\ref{eq:expand_prior}) is 0, as by assumption $d(\bm{\theta}) < \eta$.
Now, consider that for $\lambda \leq \eta$,
$h(\lambda)= P_u(d(\bm{\theta}) > \lambda)$ is minimized at $\lambda=\eta$, and therefore (\ref{eq:expand_prior})
can be bounded by
\begin{align}
\pi(\bm{\theta}) \leq \frac{\pi^L(\bm{\theta}) \int_{\lambda \leq \eta}^{}\!\, \mbox{I}(d(\bm{\theta})>\lambda) \pi(\lambda) d\lambda}{h(\eta)} 
= \frac{\pi^L(\bm{\theta}) P \left( \lambda < \mbox{min}\{ \eta , d(\bm{\theta}) \} \right)}{h(\eta)}
\label{eq:bound_prior}
\end{align}
Notice that the numerator can be made arbitrarily small by decreasing $\eta$,
since
$\pi^L(\bm{\theta})$ is bounded around $\bm{\theta}_0$,
by assumption there is no prior mass at $\lambda=0$ so that the cdf in the numerator converges to 0 as $\eta \rightarrow 0$,
and that denominator converges to 1 as $\eta \rightarrow 0$.
That is, it is possible to choose $\eta$ such that $\pi(\bm{\theta}) \leq \epsilon$,
which gives the result. 
\qed

\subsection{Proof of Corollary \ref{cor2}}
\label{proof_cor2}

Replace $\mbox{I} (d(\bm{\theta}) > \lambda)$ by $\prod_{i=1}^{p}\mbox{I}(d(\theta_i)>\lambda_i)$
in the proof of Proposition \ref{prop1}.
Letting any %an individual 
$\lambda_i$ go to 0 and applying the same argument delivers the result.

\subsection{Proof of Proposition \ref{prop3}}
\label{proof_prop3}

We first note that in order for $\pi(\bm{\theta})$ to be proper
the random variable $d( \bm{\theta} )$
must have finite expectation with respect to $\pi^L(\bm{\theta})$.
Now, the marginal prior for $\bm{\theta}$ is
\begin{align}
\pi(\bm{\theta})= \int_{}^{}\!\, \frac{\pi^L(\bm{\theta}) 
\mbox{I}(d(\bm{\theta} ) > \lambda)}{P_u( d( \bm{\theta} ) > \lambda )} \pi(\lambda) d\lambda=
\pi^L(\bm{\theta}) \int_{0}^{d( \bm{\theta} )} \!\, \frac{\pi(\lambda)}{h(\lambda)} d\lambda. 
\label{eq:marg_theta}
\end{align}
Suppose we set $\pi(\lambda) \propto h(\lambda)$, which we can do as long as $\int_{}^{}\!\, h(\lambda) d\lambda < \infty$.
Then $\pi(\bm{\theta}) \propto \pi^L(\bm{\theta}) d( \bm{\theta} )$, which proves the result.
The only step left is to show that indeed $\int_{}^{}\!\, h(\lambda) d\lambda < \infty$.
In general
\begin{align}
\int_{}^{}\!\, h(\lambda) d\lambda = \int_{}^{}\!\, P_u( d( \bm{\theta} ) > \lambda) d\lambda=
\int_{}^{}\!\, S_{d( \bm{\theta} )}(\lambda) d\lambda,
\label{eq:htau}
\end{align}
where $S_{d( \bm{\theta} )}(\lambda)$ is the survival function of the positive random
variable $d( \bm{\theta} )$
and therefore (\ref{eq:htau}) is equal to its expectation 
$E_u \left( d( \bm{\theta} ) \right)$ with respect to $\pi^L(\bm{\theta})$,
which is finite as discussed at the beginning of the proof.
\qed

\subsection{Proof of Corollary \ref{cor4}}
\label{proof_cor4}

Analogously to the proof of Proposition \ref{prop3}
the marginal prior for $\bm{\theta}$ is $\pi(\bm{\theta})=$
\begin{align}
\int_{}^{}\!\, \ldots \int_{}^{}\!\, \frac{\pi^L(\bm{\theta}) \prod_{i=1}^{p}
  \mbox{I}(d_i(\theta_i)>\lambda_i)}{P_u\left(
    d_1(\theta_1)>\lambda_1,\ldots,d_p(\theta_p)>\lambda_p \right)}
\pi(\bm{\lambda}) d \lambda_1,\ldots,d \lambda_p =  \nonumber \\
\pi^L(\bm{\theta}) \int_{0}^{d_1(\theta_1)}\!\, \ldots
\int_{0}^{d_p(\theta_p)}\!\, \frac{\pi(\bm{\lambda})}{h(\bm{\lambda})} d
\lambda_1,\ldots,d \lambda_p \propto
\pi^L(\bm{\theta}) \prod_{i=1}^{p} d_i(\theta_i),
\label{eq:marg_theta_multtrunc}
\end{align}
as by assumption $\pi(\bm{\lambda}) \propto h(\bm{\lambda})$.
\qed

\subsection{Proof of Proposition \ref{prop5}}
\label{proof_prop5}

By definition, the marginal density $\pi(\btheta^{(m)})=$
\begin{align}
\pi^L(\btheta^{(m)}) \int_{}^{}\!\, \frac{\pi(\lambda)}{h(\lambda)} \prod_{i=1}^{p} \mbox{I}(d( \theta_i ) > \lambda)  d \lambda =
\pi^L(\btheta^{(m)}) \int_{}^{}\!\, \mbox{I}(\lambda < d_{min}(\btheta^{(m)})) \frac{\pi(\lambda)}{h(\lambda)} d \lambda= \nonumber \\
\pi^L(\btheta^{(m)}) P_{\lambda}(d_{min}(\btheta^{(m)})) \int_{}^{}\!\, \frac{1}{h(\lambda)} \pi(\lambda \mid \lambda < d_{min}(\btheta^{(m)})) d \lambda,
\label{eq:proofboundary1}
\end{align}
where $P_{\lambda}(d_{min}(\btheta^{(m)}))= P(\lambda < d_{min}(\btheta^{(m)}))$ is the cdf of $\lambda$
evaluated at $d_{min}(\btheta^{(m)})$.
As $d_{min}(\btheta^{(m)}) \rightarrow 0$ we have that $\pi(\lambda \mid \lambda < d_{min}(\btheta^{(m)}))$ converges to a point mass at zero
and hence the integral in the right hand side of (\ref{eq:proofboundary1}) converges to $1/h(0)=1$.
To finish the proof of (i) notice that $P_{\lambda}(d_{min}(\btheta^{(m)}))= d_{min}(\btheta^{(m)}) \pi(\lambda^{(m)})$
for some $\lambda^{(m)} \in (0,d_{min}(\btheta^{(m)}))$
by the Mean Value Theorem,
as long as $P_{\lambda}(\cdot)$ is differentiable and continuous at $0^{+}$,
{\it i.e.} $\lambda$ is a continuous random variable.
In the particular case $\pi(\lambda)=c h(\lambda)$, note that $h(\cdot)$ is continuous and $h(0)=1$.

To prove (ii) notice that $P_{\lambda}(d_{min}(\btheta^{(m)})) \rightarrow 1$ as $d_{min}(\btheta^{(m)}) \rightarrow \infty$
and that the integral in the right hand side of (\ref{eq:proofboundary1}) is $m(d_{min}(\btheta^{(m)}))=E(1/h(\lambda) \mid \lambda < d_{min}(\btheta^{(m)}))$, 
which is increasing with $d_{min}(\btheta^{(m)})$ as $h(\lambda)$ is monotone decreasing in $\lambda$.
Hence, 
$\mathop {\lim }\limits_{m \to \infty } \pi(\bm{\theta}^{(m)})/\pi^L(\bm{\theta}^{(m)})= \mathop {\lim }\limits_{m \to \infty } m(d_{min}(\bm{\theta}^{(m)}))$
where $m(d_{min}(\btheta^{(m)}))$ increases as $m \rightarrow \infty$.
Furthermore, if $\int_{}^{}\!\, \frac{\pi(\lambda)}{h(\lambda)} < \infty$ the Monotone Converge Theorem applies and 
$m(d_{min}(\btheta^{(m)}))$ converges to a finite constant.
\qed

\subsection{Proof of Corollary \ref{cor6}}
\label{proof_cor6}

Because $\lambda_1,\ldots,\lambda_p$ have independent marginals, 
$\pi(\btheta^{(m)})=\pi^L(\btheta^{(m)}) \prod_{i=1}^{p} P_{\lambda_i}(d_i( \theta_i )) \times$
\begin{align}
\int_{}^{}\!\, \ldots \int_{}^{}\!\, \frac{1}{h(\bm{\lambda})}
\pi \left(\bm{\lambda} \mid
  \lambda_1<d_1(\theta^{(m)}_1),\ldots,\lambda_p<d_p(\theta^{(m)}_p) \right) d\lambda_1
\ldots d\lambda_p= \nonumber \\
\pi^L(\btheta^{(m)}) E \left( h(\bm{\lambda})^{-1} \mid \forall \lambda_i < d_i(
  \theta^{(m)}_i ) \right) \prod_{i=1}^{p} P_{\lambda_i}(d_i( \theta^{(m)}_i )) , 
\label{eq:corboundary1}
\end{align}
where $h(\bm{\lambda})$ is a multivariate survival function
and decreases as $d_i(\theta^{(m)}_i) \rightarrow 0$.
Hence as $d_i( \theta^{(m)}_i ) \rightarrow 0$
$E \left( h(\bm{\lambda})^{-1} \mid \forall \lambda_i < d_i( \theta^{(m)}_i ) \right)$ decreases.
To find the limit as $d_i( \theta^{(m)}_i ) \rightarrow 0$ we note that the integral
is bounded by the finite integral obtained plugging
$d_i( \theta^{(m)}_i)=1$ into the integrand.
Hence, the Dominated Convergence Theorem applies and
$\mathop {\lim }\limits_{m \to \infty } E\left(h(\bm{\lambda})^{-1} \mid \forall \lambda_i < d_i(\theta^{(m)}_i)\right)=E \left( h({\bf 0})  \right)=1$
and from (\ref{eq:corboundary1})
$\mathop {\lim }\limits_{m \to \infty } \pi(\btheta^{(m)}) / \left( \pi^L(\btheta^{(m)}) \prod_{i=1}^{p} P_{\lambda_i}(d_i(\theta^{(m)}_i)) \right)=1$.
Since $\lambda_1,\ldots,\lambda_p$ are continuous the Mean Value Theorem applies,
so that $P_{\lambda_i}\left( d_i( \theta^{(m)}_i ) \right)= d_i( \theta^{(m)}_i ) \pi(\lambda_i^{(m)})$
for some $\lambda_i^{(m)} \in (0, d_i( \theta^{(m)}_i ))$.
To prove (ii),
notice that $P_{\lambda_i}( d_i( \theta^{(m)}_i ) ) \rightarrow 1$ as $d_i( \theta^{(m)}_i ) \rightarrow \infty$
and that $m(\bm{\theta}^{(m)})= E \left( h(\bm{\lambda})^{-1} \mid \forall \lambda_i < d_i( \theta^{(m)}_i ) \right)$
increases as $d_i( \theta^{(m)}_i ) \rightarrow \infty$.
Hence, $\mathop {\lim }\limits_{m \to \infty } \pi(\btheta^{(m)}) / \left(\pi^L(\btheta^{(m)}) m(\btheta^{(m)}) \right)=1$ where
$m(\btheta^{(m)})$  increases with $d_i( \theta^{(m)}_i )$,
which proves (ii).
Further, if $E( h(\bm{\lambda})^{-1}) < \infty$ the Monotone Convergence Theorem applies and 
$\mathop {\lim }\limits_{m \to \infty } m(\btheta^{(m)})=c$ for finite $c>0$.
\qed

\subsection{Multivariate Normal sampling under outer rectangular truncation}
\label{app:alg_tnorm_sample}

The goal is to sample
$\bm{\theta} \sim N(\bm{\mu}, \Sigma) \mbox{I} \left( \bm{\theta} \in T  \right)$
with truncation region
$T=\left\{\bm{\theta} : \theta_i < l_i \mbox{ or } \theta_i > u_i,
  i=1,\ldots,p \right\}$.
We generalize the Gibbs sampling of \citeasnoun{rodriguezyam:2004}
and importance sampling of \citeasnoun{hajivassiliou:1993} and \citeasnoun{keane:1993}
to the non-convex region $T$.

Let $D= \mbox{chol}(\Sigma)$ be the Cholesky decomposition of $\Sigma$
and $K= D^{-1}$ its inverse, so that
$K \Sigma K'= K D D' K'= I$ is the identity matrix,
and define $\bm{\alpha}= K \bm{\mu}$.
The random variable ${\bf Z}= K \bm{\theta}$ follows a 
$N(\bm{\alpha},I) \mbox{I} \left( {\bf Z} \in S \right)$ distribution 
with truncation region $S$.
Since $\bm{\theta}= K^{-1} {\bf Z}= D {\bf Z}$,
denoting ${\bf d}_{i.}$ as the $i^{th}$ row in $D$
we obtain the truncation region
$S= \left\{ {\bf Z}: {\bf d}_{i.} Z \leq l_i \mbox{ or } {\bf d}_{i.} Z \geq
  u_i, i=1,\ldots,p  \right\}$.

The full conditionals for $Z_i$ given
$Z_{(-i)}=(Z_1,\ldots,Z_{i-1},Z_{i+1},\ldots,Z_p)$
needed for Gibbs sampling follow from straightforward algebra.
%Denote by ${\bf d}_{.i}$ the $i^{th}$ column in $D$
%and by $D_{(-i)}$ the $p \times (p-1)$ matrix obtained by removing
%${\bf d}_{.i}$ from $D$.
Denote by $d_{jk}$ the $(j,k)$ element in $D$,
then 
$Z_i \mid Z_{(-i)} \sim N(\alpha_i, 1)$
truncated so that either
$d_{ji} Z_i \leq l_j - \sum_{k \neq i}^{} d_{jk} Z_k$
or $d_{ji} Z_i \geq u_j - \sum_{k \neq i }^{} d_{jk} Z_k$
hold simultaneously
for $j=1,\ldots,p$.
We now adapt the algorithm to address
the fact that this truncation region is non-convex.

The region excluded from sampling
can be written as
$S_i^c= \bigcup_{j=1}^p (a_j, b_j)$,
$a_j= (l_j - \sum_{k \neq i}^{} d_{jk} Z_k)/d_{ji}$
when $d_{ji}>0$ and
$a_j= (u_j - \sum_{k \neq i}^{} d_{jk} Z_k)/d_{ji}$
when $d_{ji}<0$
(analogously for $b_j$).
$S_i^c$ as given is the union of possibly
non-disjoint intervals,
which complicates sampling.
Fortunately, it can be expressed
as a union of disjoint intervals
$S_i = \bigcup_{j=1}^{K} (\tilde{a}_j, \tilde{b}_j)$
with the following algorithm.
Suppose that $l_i$ are sorted increasingly,
set $\tilde{l}_1=l_1$, $\tilde{u}_1=u_1$
and $K=1$.
For $j=2,\ldots,p$ repeat the following two steps.
\begin{enumerate}
\item If $l_j > \tilde{u}_K$ set $K=K+1$, $\tilde{l}_K= l_j$ and
  $\tilde{u}_K=u_j$,
else if $l_j \leq \tilde{u}_K$ and $u_j \geq \tilde{u}_K$ set $\tilde{u}_K=u_j$.
\item Set $j=j+1$.
\end{enumerate}
Finally, because
$(\tilde{l}_1,\tilde{u}_1),\ldots,(\tilde{l}_K,\tilde{u}_K)$
are disjoint and increasing,
we may draw a uniform number $u$ in $(0,1)$ excluding
intervals $(\Phi(\tilde{l}_j), \Phi(\tilde{u}_j))$
and set $Z_i= \Phi^{-1}(u)$,
where $\Phi(\cdot)$
is the inverse Normal$(\alpha_i,1)$ cdf.

\subsection{Monotonicity and inverse of iMOM prior penalty}
\label{app:imompen}

Consider the product iMOM prior as given in (\ref{eq:pimompen}).
We first study the monotonicity of the penalty $d(\theta_i,\lambda)$,
which for simplicity here we denote as $d(\theta)$,
and then provide an algorithm to evaluate its inverse function.
Equivalently, it is convenient to consider the log-penalty $\mbox{log}\left( d(\theta) \right)=$
\begin{align}
 \frac{1}{2} \left( \mbox{log}(\tau \tau_N) + 2\mbox{log}(\phi) + \mbox{log}(2) \right)
- \mbox{log}\left( (\theta-\theta_0)^2 \right) - \frac{\tau \phi}{(\theta - \theta_0)^2}
+ \frac{1}{2 \tau_N \phi} (\theta - \theta_0)^2,
\label{eq:log_imompenalty}
\end{align}
as its inverse uniquely determines the inverse of $d(\theta)$.
Denoting $z= (\theta-\theta_0)^2$, (\ref{eq:log_imompenalty}) can be written as
\begin{align}
g(z)= \frac{1}{2} \left( \mbox{log}(\tau \tau_N) + 2\mbox{log}(\phi) + \mbox{log}(2) \right) - \mbox{log}(z) - \frac{\tau \phi}{z}
+ \frac{1}{2 \tau_N \phi} z.
\label{eq:log_imompenalty2}
\end{align}
To show the monotonicity of (\ref{eq:log_imompenalty2}) we compute its derivative
$g'(z)= -\frac{1}{z} + \frac{\tau \phi}{z^2} + \frac{1}{2 \tau_N \phi}$
and show that it is positive for all $z$.
Clearly, both when $z \rightarrow 0$ 
and $z \rightarrow \infty$ we have positive $g'(z)$.
Hence we just need to see that there is some $\tau_N$ for which all roots of $g'(z)$ are imaginary,
so that $g'(z)>0$ for all $z$.
Simple algebra shows that the roots of $g'(z)$ are
$z= \tau_N \phi \pm \tau_N \phi \sqrt{1 - \frac{2 \tau}{\tau_N}} $,
so that for $\tau_N \leq 2\tau$ there are no real roots.
Hence, for $\tau_N \leq 2\tau$ $g(z)$ is monotone increasing.

We now provide an algorithm to evaluate the inverse.
That is, given a threshold $t$
we seek $z_0$ such that $g(z_0)=t$.
Our strategy is to obtain an initial guess from an approximation to $g(z)$
and then use continuity and monotonicity to bound the desired $z_0$ and conduct a linear interpolation based search.
Inspecting the expression for $g(z)$ in (\ref{eq:log_imompenalty2})
we see that the term $\mbox{log}(z)$ is dominated by $\tau \phi / z$ when $z$ approaches 0
and by $\frac{z}{2\tau_N \phi}$ when $z$ is large.
Hence, we approximate $g(z)$ by dropping the $\mbox{log}(z)$ term, obtaining
\begin{align}
g(z) \approx \frac{1}{2}(\mbox{log}(\tau \tau_N) + 2 \mbox{log}(\phi) + \mbox{log}(2)) - \frac{\tau \phi}{z} + \frac{1}{2 \tau_N \phi} z.
\label{eq:gd_approx}
\end{align}
Setting (\ref{eq:gd_approx}) equal to $t$ and solving for $z$ gives
$z_0=\tau_N \phi \left( -b + \sqrt{b^2 - 2 \frac{\tau}{\tau_N}} \right)$
as an initial guess,
where $b=\mbox{log}(\tau \tau_N) + 2 \mbox{log}(\phi) + \mbox{log}(2)-t$.

If $g(z_0)<t$ we set a lower bound $z_l=z_0$ and an upper bound $z_u$
obtained by increasing $z_0$ by a factor of 2 until $g(z_0)>t$.
Similarly, if $g(z_0)>t$ we set the upper bound $z_u=z_0$
and find a lower bound by successively dividing $z_0$ by a factor of 0.5.
Once $(z_l,z_u)$ are determined, we use a linear interpolation to update $z_0$,
evaluate $g(z_0)$ and update either $z_l$ or $z_u$.
The process continues until $|g(z_0) - t|$ is below some tolerance (we used $10^{-5}$).
In our experience the initial guess is often quite good and the algorithm converges in very few iterations.

\section*{Acknowledgments}
This research was partially funded by the NIH grant R01 CA158113-01.

\bibliographystyle{ECA_jasa}
\bibliography{NLPmixTR} 

\end{document}